\documentclass[11pt]{article}

\usepackage{mathrsfs}

\usepackage[T1]{fontenc}
\usepackage[utf8]{inputenc}
\usepackage{lmodern}
\usepackage{amssymb,amsmath,amsthm}
\usepackage{bbm}
\usepackage{graphicx}
\usepackage{float}
\usepackage{xcolor}
\usepackage[colorlinks,linkcolor=black!70!black,citecolor=black]{hyperref} 
\usepackage[a4paper,margin=1.5cm]{geometry}
\usepackage{tikz}
\usepackage{tikz-cd}

\usetikzlibrary{decorations.pathreplacing, patterns,shapes,snakes}

\usepackage{pgfplots}

\long\def\metanote#1#2{{\color{#1}\
\ifmmode\hbox\fi{\sffamily\mdseries\upshape [#2]}\ }}

\theoremstyle{plain}
\newtheorem{theo}{Theorem}[section]
\newtheorem{lem}[theo]{Lemma}

\newtheorem{prop}[theo]{Proposition}

\theoremstyle{remark}

\newtheorem{rmk}[theo]{Remark}
\newtheorem{defin}[theo]{Definition}
\newtheorem{cond}[theo]{Condition}

\newcommand{\ra}{\rightarrow}

\newcommand{\be}{\begin{equation}}
\newcommand{\ee}{\end{equation}}
\newcommand{\bi}{\begin{itemize}}
\newcommand{\ei}{\end{itemize}}

\newcommand{\commentout}[1]{}
\newcommand{\Cts}{\text{Cts}}
\newcommand{\Lip}{\text{Lip}}
\newcommand{\TV}{\text{TV}}

\newcommand{\Wah}{\text{W}}

\newcommand{\Law}{{\mathcal{L}}}

\newcommand{\calM}{{\mathcal{M}}}

\newcommand{\K}{{\mathbb{K}}}

\newcommand{\E}{{\mathcal{E}}}
\newcommand{\F}{{\mathcal{F}}}
\newcommand{\calA}{{\mathcal{A}}}
\newcommand{\calB}{{\mathcal{B}}}

\newcommand{\calD}{{\mathcal{D}}}
\newcommand{\calE}{{\mathcal{E}}}
\newcommand{\calF}{{\mathcal{F}}}
\newcommand{\calG}{{\mathcal{G}}}

\newcommand{\calO}{{\mathcal{O}}}

\newcommand{\calP}{{\mathcal{P}}}

\newcommand{\bfS}{{\bf{S}}}
\newcommand{\calK}{{\mathcal{K}}}
\newcommand{\calY}{{\mathcal{Y}}}
\newcommand{\Nm}{{\mathbb{N}}}
\newcommand{\MG}{{\text{MG}}}
\newcommand{\FV}{{\text{FV}}}

\newcommand{\Rm}{{\mathbb R}}
\newcommand{\Pm}{{\mathbb P}}

\newcommand{\expE}{{\mathbb E}}

\newcommand{\Ind}{{\mathbbm{1}}}

\newcommand{\Wat}{\text{W}_a}

\long\def\metanote#1#2{{\color{#1}\
\ifmmode\hbox\fi{\sffamily\mdseries\upshape [#2]}\ }}

\begin{document}
\setcounter{page}{1}

\title{Scaling Limit of the Fleming-Viot Multi-Colour Process}
\author{Oliver Tough\footnote{Department of Mathematics, University of Bath, okt24@bath.ac.uk}}
\date{July 31, 2023}

\maketitle

\begin{abstract}
We consider the $N$-particle Fleming-Viot process associated to a normally reflected diffusion with soft catalyst killing. The Fleming-Viot multi-colour process is obtained by attaching genetic information to the particles in the Fleming-Viot process. We establish that, after rescaling time by $t\mapsto Nt$, this genetic information converges to the (very different) Fleming-Viot process from population genetics, as $N\rightarrow\infty$. An extension is provided to dynamics given by Brownian motion with hard catalyst killing at the boundary of its domain.
\end{abstract}

\section{Introduction and main result}

In this paper we study the behaviour of a system of interacting diffusion processes, known as a Fleming-Viot particle system, first introduced by Burdzy, Hołyst and March in \cite{Burdzy2000}. We will establish that if one attaches genetic information to the Fleming-Viot particle system and rescales time by $t\mapsto Nt$, this genetic information evolves for large $N$ like the (very different) Fleming-Viot process from population genetics, which we refer to in this article as a \textit{Wright-Fisher process} for the avoidance of confusion. This is our main theorem, Theorem \ref{theo:Convergence to FV diffusion}. We emphasise that, despite sharing the same name, no link had previously been established between the Fleming-Viot particle system (or any similar particle system) and the Wright-Fisher process.

Throughout this paper, $(X_t)_{0\leq t<\tau_{\partial}}$ will be defined to be a diffusion process evolving in the closure $\bar D$ of an open, connected, bounded domain $D\subseteq \Rm^d$, normally reflected at the $C^{\infty}$ boundary $\partial D$, and killed at position dependent rate $\kappa(X_t)$ (\textit{soft killing}). That is, prior to the killing time $\tau_{\partial}$, $X_t$ evolves according to the SDE
\begin{equation}\label{eq:reflected diffusion SDE}
\begin{split}
dX_t=b(X_t)dt+\sigma(X_s)dW_s+\vec{n}(X_t)d\xi_t\in \bar D,\quad 0\leq t<\tau_{\partial},\\
\text{with}\quad \Ind(\tau_{\partial}>t)+\int_0^t\kappa(X_s)\Ind(\tau_{\partial}>s)ds\quad\text{being a martingale,}
\end{split}
\end{equation}
whereby $\xi_t$ is the boundary local time of $X_t$ at $\partial D$ and $\hat{n}(x)$ is the unit interior normal at $x\in \partial D$. A precise definition of such processes is given in Appendix \ref{appendix:reflected diffusions}. We assume throughout that $\kappa\in C^{\infty}(\Rm^d;\Rm_{\geq 0})$ and is strictly positive somewhere on $\bar D$. We also assume that $b\in C^{\infty}(\Rm^d;\Rm^d)$ and $\sigma\in C^{\infty}(\Rm^d;\Rm^{d\times m})$, with $\sigma \sigma^T$ uniformly positive definite. 

The Fleming-Viot particle system is defined as follows.
\begin{defin}[Fleming-Viot particle system]
The Fleming-Viot particle system $(\vec{X}^N_t)_{t\geq 0}$ consists of $N \geq 2$ particles
\[
\vec{X}^N_t=(X^{N,1}_t,\ldots,X^{N,N}_t),\quad t\geq 0,
\]
evolving independently in the domain $\bar D$ according to \eqref{eq:reflected diffusion SDE}. When a particle is killed we relocate it to the position of a different particle chosen independently and uniformly at random. 
\label{defin:FV particle system reflected diffusions}
\end{defin}
In general, it is not clear that the Fleming-Viot particle system is well-posed due to the possibility of infinitely-many jumps in finite time. In the present setting, however, this is not an issue as the killing rate is bounded.

The Fleming-Viot particle system was introduced by Burdzy, Hołyst and March \cite{Burdzy2000} in the case of Brownian dynamics with instantaneous killing at the boundary (\textit{hard killing}), where it was shown to provide an approximation method both for the heat equation with Dirichlet boundary conditions and the principal eigenfunction of the Dirichlet Laplacian. The Fleming-Viot particle system with soft killing was considered by Grigorescu in \cite{Grigorescu2007}. The Fleming-Viot particle system has been shown to provide a general approximation method for absorbed strong Markov processes by Villemonais \cite{Villemonais2011}, and has been shown to provide an approximation method for quasi-stationary distributions (QSDs) in a variety of settings \cite{Burdzy2000,Asselah2011,Asselah2016,Tough2023}. When a killed Markov process is Feller, quasi-stationary distributions correspond to left eigenmeasures of its infinitesimal generator \cite[Proposition 4]{Meleard2011}.

\subsection{The Fleming-Viot multi-colour process}

We attach genetic information (``colours'') to the Fleming-Viot particle system, resulting in the Fleming-Viot multi-colour process, which was introduced by Grigorescu and Kang in \cite[Section 5.1]{Grigorescu2012}. Whereas the colours in the construction of \cite[Section 5.1]{Grigorescu2012} are assumed to belong to a finite space, the present article develops this by instead assuming the colours belong to a complete, separable metric space. This space is referred to as the ``colour space'' and is denoted by $\mathbb{K}$. The colour $\eta^i_t\in\K$ gives the genetic information of the particle $X^i_t$, for $i=1,\ldots,N$. A precise definition of the Fleming-Viot multi-colour process is given by the following.

\begin{defin}[Fleming-Viot multi-colour process]
We take $(\K,d)$ to be an arbitrary complete separable metric space, which we call the colour space. We define $(\vec{X}^N_t,\vec{\eta}^N_t)_{0\leq t<\infty}=\{(X^{N,i}_t,\eta^{N,i}_t)_{0\leq t<\infty}:i=1,\ldots,N\}$ as follows:

\begin{equation}  
\left \{ \begin{split}
(i) & \quad \text{Initial condition: $((X^{N,1}_0,\eta^{N,1}_0),\ldots,(X^{N,N}_0,\eta^{N,N}_0))\sim \upsilon^N\in \mathcal{P}((\bar D\times\K)^N)$.}\\
(ii) & \quad \text{For $t \in [0,\infty)$ and between killing times the particles $(X^{N,i}_t,\eta^{N,i}_t)$ evolve and are killed}\\
& \quad\text{independently according to $\eqref{eq:reflected diffusion SDE}$ in the first variable, and are constant in the second}\\
& \quad\text{variable.}\\
(iii) & \quad \text{We write $\tau^i_k$ for the death times of particle $(X^{N,i},\eta^{N,i})$ (with $\tau^i_0:=0$). When particle}\\
& \quad\text{$(X^{N,i},\eta^{N,i})$ is killed at time $\tau^i_k$ it jumps to the location of particle $(X^{N,j},\eta^{N,j})$, with}\\
& \quad\text{$j=U^i_k\in\{1,\ldots,N\}\setminus\{i\}$ chosen independently and uniformly at random, at which}\\
& \quad \text{ time we set $(X^{N,i}_{\tau^i_k},\eta^{N,i}_{\tau^i_k})=(X^{N,j}_{\tau^i_k-},\eta^{N,j}_{\tau^i_k-})$. Moreover we write $\tau_n$ for the $n^{\text{th}}$ time at}\\
& \quad \text{which any particle is killed (with $\tau_0:=0$).}
\end{split}  \right.
\label{eq:N-particle m label (X,eta) system}
\end{equation}
We then define
\begin{equation}\label{eq:jumps of multicolour}
J^N_t:=\frac{1}{N}\sup\{n>0:\tau_n\leq t\}
\end{equation}
to be the number of deaths up to time $t$ normalised by $\frac{1}{N}$, and define the empirical measures
\begin{equation}\label{eq:spatial and colour empirical measures}
m^N_t:=\frac{1}{N}\sum_{i=1}^N\delta_{X^{N,i}_t}\quad\text{and}\quad  \chi^{N}_t:=\frac{1}{N}\sum_{i=1}^N\delta_{\eta^{N,i}_t}.
\end{equation}
\label{defin:Multi-Colour Process}
\end{defin}
We will obtain a scaling limit for the  colours as $N\ra\infty$ and time is rescaled according to $t\mapsto Nt$. We now describe the scaling limit we will obtain.

\subsection{The Wright-Fisher process}\label{subsection:WF process intro}
Given a gene with two neutral alleles, $a$ and $A$, the SDE
\[
dp_t=\sqrt{p_t(1-p_t)}dW_t
\]
models the evolution of the proportion $p_t\in [0,1]$ of the population carrying the $a$-allele in a large population. This is the classical Wright-Fisher diffusion. Generalising this to $n$ alleles, the driftless $n$-Type Wright-Fisher diffusion process of rate $\theta>0$ takes values in the simplex $\Delta_n:=\{p=(p_1,\ldots,p_n)\in \Rm_{\geq 0}^n:\sum_j p_j=1\}$ and is characterised by the generator
\begin{equation}
    L_{\text{WF}}=\frac{1}{2}\theta\sum_{i,j=1}^np_i(\delta_{ij}-p_j)\frac{\partial^2}{\partial p_i\partial p_j},\quad \mathcal{D}(L)=C^2(\Rm^n).
\label{eq:generator of K dim WF-diff}
\end{equation}

This was generalised by Fleming and Viot \cite{Fleming1979} to a probability measure-valued process, which allows for the set of alleles to be infinite. This measure-valued process is typically called a Fleming-Viot process, but is referred to as the Wright-Fisher process in the present article to avoid confusion. In particular, letting $\K$ be the (complete, separable) colour space, we will consider the Wright-Fisher process on $\calP(\K)$. This corresponds to the set of possible alleles being $\K$, and will be our scaling limit. 

The reader is directed towards \cite{Ethier1993c} for a survey of the Wright-Fisher process due to Ethier and Kurtz. The Wright-Fisher process is defined as a solution of a martingale problem. This typically features additional terms representing mutation, selection and recombination, but we will not need this generality here. There are various possible formulations of this martingale problem, which can be found in \cite[Section 3]{Ethier1993c}. The formulation we shall employ is given by \cite[(3.20) and (3.21)]{Ethier1993c}. This definition of the Fleming-Viot process, as well as its well-posedness (which comes from \cite[Theorem 7.1]{Ethier1993c}), are given in Appendix \ref{appendix:Wright-Fisher process}. 

As in \eqref{eq:generator of K dim WF-diff}, we parametrise the Wright-Fisher process on $\calP(\K)$ with a rate $\theta>0$. The following proposition provides intuition for how one may think of the Wright-Fisher process and its relationship to the $n$-type Wright-Fisher diffusion. This proposition shall be used in the proof of our main theorem, and is proven in Appendix \ref{appendix:Wright-Fisher process}.
\begin{prop}\label{prop:basic facts WF superprocess}
We let $(\nu_t)_{t\geq 0}$ be a Wright-Fisher process on $\calP(\K)$ of rate $\theta>0$ as defined in Appendix \ref{appendix:Wright-Fisher process}. Then for all finite disjoint unions of measurable subsets, $\dot{\cup}_{j=1}^n\calA_j=\K$, we have that
\begin{equation}\label{eq:WF diff sect measure of disjoint sets}
(\nu_t(\calA_1),\ldots,\nu_t(\calA_n)),\quad t\geq 0,
\end{equation}
is an $n$-type Wright-Fisher diffusion of rate $\theta$. 
\end{prop}
\subsection{Main result}
We will establish in Appendix \ref{appendix:convergence to a QSD} the following. The absorbed process $(X_t)_{0\leq t<\tau_{\partial}}$ is Feller (there is no distinction between $C_0$-Feller and $C_b$-Feller as $\bar D$ is compact). We write $L$ for its infinitesimal generator. Then $(X_t)_{0\leq t<\tau_{\partial}}$ has a unique QSD, denoted by $\pi$, which is a left eigenmeasure of $L$. We denote the corresponding eigenvalue as $-\lambda<0$. Furthermore there exists a positive right eigenfunction $\phi\in \calD(L)\cap C^{2}(\bar D; \Rm_{>0})$, which is both the unique non-negative right eigenfunction and the unique right eigenfunction of eigenvalue $-\lambda$, up to rescaling. Throughout we normalise $\phi$ so that $\langle \pi,\phi\rangle=1$.

We may therefore define the constant
\begin{equation}\label{eq:const theta}
\Theta:=\frac{2\lambda \lvert\lvert \phi\rvert\rvert_{L^2(\pi)}^2}{\lvert\lvert \phi\rvert\rvert_{L^1(\pi)}^2}.
\end{equation}
We define the tilted empirical measure of the colours, denoted as $(\calY^N_t)_{0\leq t<\infty}$, by
\begin{equation}\label{eq:tilted empirical measure}
\begin{split}
\calY^N_t:=\frac{\frac{1}{N}\sum_{i=1}^N\phi(X^{i}_t)\delta_{\eta^{i}_t}}{\frac{1}{N}\sum_{i=1}^N\phi(X^{i}_t)}\in\calP(\K).
\end{split}
\end{equation}
Whereas consideration of this quantity shall play a crucial role in our proof, for the purposes of our theorem statement its role is to provide the initial condition of our scaling limit. To the authors' knowledge, this process is original. The proof of Theorem \ref{theo:Convergence to FV diffusion} shall be outlined in Subsection \ref{subsection:outline of proof strategy}, at which point we shall explain the role of $\calY^N_t$ in the proof.

Convergence will be stated in terms of the weak atomic metric on $\mathbb{K}$, denoted as $\Wat$. The space of probability measures on $\mathbb{K}$ equipped with the weak atomic metric is denoted by $\calP_{\Wat}(\mathbb{K})$. This metric was introduced by Ethier and Kurtz \cite{Ethier1994} in the context of population genetics. Convergence in the weak atomic metric is equivalent to having both weak convergence of measures and convergence of the sizes and locations of the atoms. We provide a definition of the weak atomic metric in Appendix \ref{appendix:Weak atomic metric}.

Our main theorem is then the following.

\begin{theo}\label{theo:Convergence to FV diffusion}
We take some deterministic initial profile $\nu^0\in\mathcal{P}(\K)$ and fix a Wright-Fisher process on $\calP(\mathbb{K})$ of rate $\Theta$ and initial condition $\nu_0=\nu^0$, which we denote as $(\nu_t)_{0\leq t<\infty}$. We consider a sequence of Fleming-Viot multi-colour Processes, denoted by $(((\vec{X}^N_t,\vec{\eta}^N_t))_{0\leq t<\infty}:2\leq N<\infty)$, such that 
\begin{equation}
\calP(\K)\ni\calY^N_0\ra \nu^0\in \calP(\K)\quad \text{in $\Wat$ in probability as}\quad N\ra\infty.
\end{equation}

We now rescale time according to $t\mapsto Nt$. Then $(\chi^{N}_{Nt})_{t> 0}$ converges to $(\nu_t)_{t> 0}$ in finite-dimensional distributions, in the following sense. We fix arbitrary $n<\infty$ and $\vec{t}=(t^1,\ldots,t^n)\in [0,\infty)^n$ such that $t^1\leq \ldots\leq t^n$. We consider arbitrary sequences $(\vec{t}^N)_{2\leq N<\infty}:=((t^N_1,\ldots,t^N_n))_{2\leq N\leq \infty}$ such that:
\begin{enumerate}
\item
$t^N_1\leq \ldots\leq t^N_n$ for all $2\leq N<\infty$;
\item
$t^N_i\ra t_i$ as $N\ra \infty$ for all $1\leq i\leq n$;
\item\label{enum:main theorem requirement that times are large}
$Nt^N_n\geq\ldots\geq Nt^N_1\ra\infty$ as $N\ra \infty$.
\end{enumerate}
We then have that
\begin{equation}\label{eq:main theorem convergence of empirical measures}
(\chi^{N}_{Nt_1^N},\ldots,\chi^{N}_{Nt_n^N}) \ra (\nu_{t_1},\ldots,\nu_{t_n})\quad\text{in}\quad (\calP_{\Wat}(\mathbb{K}))^n\quad \text{in distribution as}\quad N\ra\infty.
\end{equation}
\end{theo}

\begin{rmk}
If we take constant killing rate $\kappa\equiv 1$ and consider the corresponding Fleming-Viot multi-colour process, we recover the classical Moran model. This is well-known to converge to the Wright-Fisher process of rate $2$ \cite[(4.12)]{Ethier1993c}. On the other hand, we can check that $\Theta=2$ when $\kappa\equiv 1$.
\end{rmk}

\begin{rmk}
We observe that, unless $\phi$ is constant (which only happens if $\kappa$ is constant on $\bar D$), the empirical measures $\chi^N_0$ will, in general, not converge to the same limit as the tilted empirical measures $\calY^N_0$. We therefore no longer have \eqref{eq:main theorem convergence of empirical measures} if we drop the requirement that $Nt^N_1\ra \infty$ as $N\ra\infty$. This represents the following separation of timescales phenomenon.

We will establish in the proof of Theorem \ref{theo:Convergence to FV diffusion} that the tilted empirical measure $\calY^N_t$ evolves slowly over an $\calO(N)$ timescale, with $(\calY^N_{Nt})_{t\geq 0}$ converging to the Wright-Fisher process. We further establish that the empirical measure $\chi^N_t$ converges on a shorter $\calO(1)$ timescale to the tilted empirical measure $\calY^N_t$. Theorem \ref{theo:Convergence to FV diffusion} then follows by combining these two facts. 

Therefore for large $N$, the empirical measure $\chi^N_t$ quickly approaches $\nu^0$ over an $\calO(1)$ timescale, before evolving like the Wright-Fisher process over the longer $\calO(N)$ timescale.
\end{rmk}

\subsection{Background and related results}

A similar separation of timescales has been obtained by Méléard and Tran in \cite{Meleard2012}. They considered the evolution of traits in a population of individuals, where the individuals give birth (passing on their trait) and die in an age-dependent manner, and interact with each other through the effect of the common empirical measure of their traits upon their death rates (representing competition for resources). There the age component plays a similar role to spatial position in the present article. They found that the age component converges to a deterministic equilibria (which is dependent upon the traits) on a fast timescale, whilst the trait distribution evolves on a slow timescale, converging to a certain superprocess over the slow timescale as the population converges to infinity. 

Aside from obtaining a different limiting process, they also employ a different proof strategy. In their setup, individuals give birth and are killed at rates which ensure that the slow component does not have large drift terms on the fast timescale, whereas it does in the present setup. This necessitates the different proof strategy. In Subsection \ref{subsection:outline of proof strategy} we shall outline the proof strategy of Theorem \ref{theo:Convergence to FV diffusion}, at which point we shall elaborate on the difference between this proof and the proof in \cite{Meleard2012}. 

The ancestral paths of both the Fleming-Viot particle system and similar particle systems have been considered by a number of authors, for instance by M\'el\'eard and Tran \cite{Meleard2012}, Grigorescu and Kang \cite{Grigorescu2012} and Burdzy et. al. \cite{Bieniek2018,Burdzy2019,Burdzy2022,Burdzy2021a}. None of these make a link with the Wright-Fisher process. In a sequel to the present paper, we shall use Theorem \ref{theo:Convergence to FV diffusion} to link the ancestral paths of the Fleming-Viot particle system with a Wright-Fisher process on $\calP(C([0,T];\bar D))$. This link was included in the original preprint version of this paper \cite{Tough2021}, and earlier in the author's PhD thesis \cite[Chapter 4]{Tough2021b}. 

In \cite{Grigorescu2012}, Grigorescu and Kang constructed the immortal particle, also known as the spine, of the Fleming-Viot particle system - the unique ancestral path from time 0 to time $\infty$. They introduced the Fleming-Viot multi-colour process, with the colours belonging to a finite set, in order to construct this process. The construction of the spine of the Fleming-Viot particle system was later extended to a very general setting by Bieniek and Burdzy \cite[Theorem 3.1]{Bieniek2018}. Bieniek and Burdzy \cite[Section 5]{Bieniek2018} established that, when the state space is finite, the distribution of the spine of the Fleming-Viot particle system converges as $N\ra\infty$ to that of the driving Markov process $(X_t)_{0\leq t<\tau_{\partial}}$ conditioned never to be killed - referred to in the literature as the \textit{$Q$-process} \cite[Section 3]{Champagnat2014}. They conjectured that this is also true for general state spaces \cite[p.3752]{Bieniek2018}. Since then, Burdzy, Kołodziejek and Tadić in \cite{Burdzy2019,Burdzy2022} have established a law of the iterated logarithm \cite[Theorem 7.1]{Burdzy2022} which, as they explain, hints that the conjecture of Bieniek and Burdzy should hold in the setting they consider. None of these articles draw a link with the Wright-Fisher process. 

In a sequel to the present article, we shall prove Bieniek and Burdzy's conjecture, \cite[p.3752]{Bieniek2018}, in the setting of the present paper. This proof was included in the original preprint version of this paper \cite{Tough2021}, and earlier in the author's PhD thesis \cite[Chapter 4]{Tough2021b}. This was the first proof of the conjecture outside of the finite state space setting. Subsequent to \cite[Chapter 4]{Tough2021b} and \cite{Tough2021}, Burdzy and Engländer have established this conjecture in \cite{Burdzy2021a}, when the driving process is Brownian motion killed at the boundary of its bounded domain. We emphasise that the proof strategy due to Burdzy et al. in \cite{Bieniek2018,Burdzy2021a} is completely different to the proof due to the present author in \cite[Chapter 4]{Tough2021b} and \cite{Tough2021}, with no connection being made between the Fleming-Viot particle system and the Wright-Fisher process in \cite{Bieniek2018,Burdzy2021a}. Bieniek and Burdzy's proof when the state space is finite \cite[Section 5]{Bieniek2018} used the finiteness of the state space in a seemingly essential way - they used the fact that if two particles are at the same location they must have the same probability of being the spine, and moreover the particles can only be at a finite number of possible locations. Burdzy and Engländer were able to use the same argument in \cite{Burdzy2021a} when the driving process is Brownian motion killed at the boundary of its domain by dividing the domain up into cubes and using the form of the multidimensional Gaussian distribution to argue that any two particles in the same cube must have almost the same probability of being the spine. On the other hand, the proof appearing in \cite[Chapter 4]{Tough2021b} and \cite{Tough2021}, and which will appear in a sequel to the present article, instead leverages the connection between the Fleming-Viot particle system and the Wright-Fisher process established in Theorem \ref{theo:Convergence to FV diffusion}.

The $N$-branching Brownian motion ($N$-BBM) consists of $N$ particles evolving in between killing times as independent Brownian motions. At rate $N$, one kills the particle minimising or maximising a given fixed function. At the same time, as with the Fleming-Viot particle system, another particle chosen uniformly at random branches, so that the number of particles remains fixed. Clearly this particle system is similar to the Fleming-Viot particle system. Particle systems of this form were first introduced by Brunet and Derrida in \cite{Brunet1997}. Such particle systems have been studied, for instance, by Brunet and Derrida \cite{Brunet2001}, Durrett and Reminik \cite{Durrett2011}, Maillard \cite{Maillard2016}, and Berestycki, Brunet, Nolen and Penington \cite{Berestycki2022}. The genealogy of these particles systems has received particular attention - see also the work of Brunet, Derrida, Mueller and Munier \cite{Brunet2006,Brunet2007}, Mallein \cite{Mallein2017} and Penington, Roberts and Talyigás \cite{Penington2022}. 

For the $N$-BBM studied in \cite{Maillard2016}, the particles are in $1$ dimension with the leftmost particle being killed at each killing time. It is a hard open problem to show that the genealogy of this particle system is given by a Bolthausen-Sznitman coalescent \cite[p.1066]{Maillard2016}, so we should not expect a Wright-Fisher process scaling limit. This conjecture has been proven for the similar near-critical branching Brownian motion by Berestycki, Berestycki and Schweinsberg in \cite{Berestycki2013}. On the other hand, in the ``Brownian bees'' particle system considered in \cite{Berestycki2022}, it is the particle furthest away from $0$ which is killed. In contrast to the $N$-BBM, we should expect the this particle system to have a Wright-Fisher process limit after rescaling time by $t\mapsto Nt$ as in Theorem \ref{theo:Convergence to FV diffusion}, in the opinion of the present author. The key distinction between these two Brunet-Derrida-type particle systems is that the killing mechanism in the latter has the effect of constraining the mass of particles. However, the genealogy of the Brownian bees particle system has not yet been addressed, nor has a Wright-Fisher process limit previously been established for any variant of this particle system.

A scaling limit for the geneaology of a sequential Markov chain Monte Carlo algorithm was established by Brown, Jenkins, Johansen and Koskela in \cite[Theorem 3.2]{Brown2021}. This captures the phenomenon of ancestral degeneracy, which has a substantial impact on the performance of the algorithm. They established that the geneaology of an $n$-particle sample converges to Kingman's $n$-coalescent as the number of particles goes to infinity and time is suitably rescaled. This is suggestive of a Wright-Fisher process, since Kingman's coalescent is dual to the Wright-Fisher process (see \cite[Appendix A]{Labbe2013}), but no such connection is made.

In the engineering literature, Mulatier, Dumonteil, Rosso and Zoia \cite{DeMulatier2015} considered a particle system whereby $N$ Brownian particles branch at a rate $\lambda$, at which point another particle chosen uniformly at random is removed, conserving the number of particles. Clearly this is very similar to the Fleming-Viot particle system, with the difference being that here particle births trigger another particle chosen uniformly at random to be killed, rather than vice-versa. This is used as a toy model for neutrons in a nuclear reactor and their Monte Carlo simulation. They investigated the phenomenon of ``clustering'', in which particles cluster together in Monte-Carlo simulations of nuclear reactors, which has a substantial impact on the accuracy of these simulations. They explained this phenomenon as occurring when particle ancestries coalesce more quickly than particles are able to explore the space. They argued that this should occur on a timescale of $\frac{N}{\lambda}$. However, it is unknown how quickly ancestries coalesce for such systems (when the branching rate is non-constant), even at the level of a conjecture. It should be straightforward to replicate the proof in the present paper for these systems, thereby quantifying how quickly ancestries coalesce via an analogue of Theorem \ref{theo:Convergence to FV diffusion}. This would indicate how large $N$ should be to avoid clustering. We will see in the following subsection that ancestral coalescence occurs more quickly when $\phi$ is non-constant (but $N$ and $\lambda$ are the same), so that a larger $N$ would be needed to avoid clustering.

\subsection{Effective population size}

In population genetics, variance effective population size refers to the population of an idealised, spatially unstructured population with the same genetic drift per generation. For a variety of reasons, this effective population size is generally observed to be considerably less than the census population size \cite{Frankham1995}. 

We recall that $(\pi,-\lambda,\phi)$ is the principal eigentriple of the infinitesimal generator $L$. We obtained in Theorem \ref{theo:Convergence to FV diffusion} that, after rescaling time by $t\mapsto Nt$, the Fleming-Viot multi-colour process converges to a Wright-Fisher process of rate $\Theta:=\frac{2\lambda \lvert\lvert \phi\rvert\rvert_{L^2(\pi)}^2}{\lvert\lvert \phi\rvert\rvert_{L^1(\pi)}^2}$. It is straightforward to combine Theorem \ref{theo:hydrodynamic limit for multicolour process} with Theorem \ref{theo:convergence to QSD for reflected diffusion with soft killing} to establish that individuals in the Fleming-Viot multi-colour process die, on average, $\lambda$ times per unit time. If we remove space, and instead assume that each individual is killed at fixed Poisson rate $\kappa\equiv\lambda$, we obtain the classical static Moran model. We therefore define the variance effective population here to be the size of an equivalent static Moran model. 

The Wright-Fisher process is well-known to arise as the limit of suitably rescaled Moran models \cite[(4.12)]{Ethier1993c}. If we let $(\vec{\eta}^{\text{Moran},N}_t)_{0\leq t<\infty}$ be the $N$-individual static Moran model (where each individual dies at Poisson rate $\lambda$), and define the constant $c=\frac{\Theta}{2\lambda}=\Big(\frac{\lvert\lvert \phi\rvert\rvert_{L^1(\pi)}}{\lvert\lvert \phi\rvert\rvert_{L^2(\pi)}}\Big)^2$, we have that $\vec{\eta}^{\text{Moran},\lfloor cN\rfloor}_{Nt}$ converges to a Wright-Fisher process of rate $\Theta$. It follows that 
\begin{equation}\label{eq:effective population}
N_{\text{eff}}\sim \Big(\frac{\lvert\lvert \phi\rvert\rvert_{L^1(\pi)}}{\lvert\lvert \phi\rvert\rvert_{L^2(\pi)}}\Big)^2N.
\end{equation}
We observe that $N_{\text{eff}}\leq N$, with equality if and only if $\phi$ is constant on $\bar D$, which is equivalent to $\kappa$ being constant on $\bar D$.

We offer the following heuristic interpretation of \eqref{eq:effective population}. We have from Theorem \ref{theo:convergence to QSD for reflected diffusion with soft killing} that
\[
\Pm_x(\tau_{\partial}>t)\sim \phi(x)e^{-\lambda t}.
\]

On the other hand, the profile of the particles in the Fleming-Viot particle system will settle upon a close approximation of $\pi$. Therefore if $\lvert\lvert \phi\rvert\rvert_{L^2(\pi)}$ is much larger than $\lvert\lvert \phi\rvert\rvert_{L^1(\pi)}$, then a small subset of individuals at any given time should be expected to subsequently survive for much longer than the average. These individuals will therefore have far more children than the average, having the effect of speeding up the coalescence time, hence reducing the effective population size.

\subsection{A hydrodynamic limit theorem for the Fleming-Viot multi-colour process}

Both the proof of Theorem \ref{theo:Convergence to FV diffusion} and our heuristic explanation of it will make use of the following hydrodynamic limit theorem for the Fleming-Viot multi-colour process. The hydrodynamic limit we obtain is given by the laws of the following killed Markov process.
\begin{defin}\label{defin:limit for multi colour over fixed times}
We define a $\bar D\times\K$-valued killed strong Markov process, denoted by $((X_t,\eta_t))_{0\leq t<\tau_{\partial}}$, as follows. The process evolves in the first variable like the killed normally-reflected diffusion $(X_t)_{0\leq t<\tau_{\partial}}$ defined in \eqref{eq:reflected diffusion SDE}, with the killing time of $((X_t,\eta_t))_{0\leq t<\tau_{\partial}}$ being the same as the killing time of $(X_t)_{0\leq t<\tau_{\partial}}$. In the second variable $\eta_t$ is a constant element of $\K$ up to the killing time $\tau_{\partial}$, so that $\eta_t=\eta_0$ for all $0\leq t<\tau_{\partial}$. After the killing time the process is sent to a fixed cemetery state.
\end{defin}

\begin{theo}\label{theo:hydrodynamic limit for multicolour process}
We consider the Fleming-Viot multi-colour process $((\vec{X}^N_t,\vec{\eta}^N_t))_{t\geq 0}$ for $N\geq 2$. Then there exists constants $C_{T,N}$ for $0\leq T<\infty$ and $N\geq 2$ such that $C_{T,N}\ra 0$ as $N\ra \infty$, and such that for any initial condition $(\vec{X}^N_0,\vec{\eta}^N_0)$ and any $f\in \calB_b(\bar D\times \K;\Rm)$, we have that
\begin{align}\label{eq:hydrodnamic limit of FVMC in appendix}
\expE_{(\vec{X}^N_0,\vec{\eta}^N_0)}\Big[\sup_{t\leq T}\Big\lvert \Big(\frac{1}{N}\sum_{i=1}^N\delta_{(X^{N,i}_t,\eta^{N,i}_t)}-\Law_{\frac{1}{N}\sum_{i=1}^N\delta_{(X^{N,i}_0,\eta^{N,i}_0)}}((X_t,\eta_t))\Big)(f)\Big\rvert\Big]\leq C_{T,N}\lvert\lvert f\rvert\rvert_{\infty},\\
\expE_{(\vec{X}^N_0,\vec{\eta}^N_0)}\Big[\sup_{t\leq T}\Big\lvert J^N_t-\ln\Pm_{\frac{1}{N}\sum_{i=1}^N\delta_{(X^{N,i}_0,\eta^{N,i}_0)}}(\tau_{\partial}>t)\Big\rvert\wedge 1 \Big]\leq C_{T,N}.\label{eq:hydrodnamic limit of FVMC in appendix number of jumps}
\end{align}
\end{theo}

\begin{proof}[Proof of Theorem \ref{theo:hydrodynamic limit for multicolour process}]
We take the Fleming-Viot particle system associated to the killed strong Markov process $((X_t,\eta_t))_{0\leq t<\tau_{\partial}}$ defined in Definition \ref{defin:limit for multi colour over fixed times} (which is well-defined since the killing rate is bounded). We observe that its dynamics are identical to that of the Fleming-Viot multi-colour process $(\vec{X}^N_t,\vec{\eta}^N_t)_{t\geq 0}$ associated to $(X_t)_{0\leq t<\tau_{\partial}}$. We are therefore able to apply \cite[Theorem 2.2]{Villemonais2011} to the Fleming-Viot multi-colour process.

The statement of \cite[Theorem 2.2]{Villemonais2011} only gives an estimate of the particle system at fixed times. However, its proof relied on a martingale decomposition, \cite[Theorem 2.2]{Villemonais2011}, with $L^2$ controls obtained on the two martingales, \cite[(2.8) and (2.9)]{Villemonais2011}. By applying Doob's $L^2$-martingale inequality, these controls become uniform over the time horizon $[0,T]$. We thereby make \cite[Theorem 2.2]{Villemonais2011} uniform over the time horizon $[0,T]$, at the cost of the estimate in \cite[Theorem 2.2]{Villemonais2011} being multiplied by $4$. Applying this uniform estimate to the Fleming-Viot multi-colour process, we obtain \eqref{eq:hydrodnamic limit of FVMC in appendix}. We similarly obtain \eqref{eq:hydrodnamic limit of FVMC in appendix number of jumps} from \cite[$1^{\text{st}}$ eq. on p.450]{Villemonais2011}.
\end{proof}

We prove in the appendix that $((X_t,\eta_t))_{0\leq t<\tau_{\partial}}$ has the following large-time limit. 
\begin{prop}\label{prop:convergence in time for killed Markov 1 in main theorem proof}
For arbitrary sequences $(x^i,\eta^i)_{1\leq i\leq n}$ in $\bar D\times \K$ we consider the process \linebreak$(X_t,\eta_t)_{0\leq t<\tau_{\partial}}$ with initial distribution given by the empirical measure $\frac{1}{n}\sum_{i=1}^n\delta_{(x^i,\eta^i)}$. Then there exists $c_t\ra 0$ as $t\ra \infty$ such that, for all sequences $(x^i,\eta^i)_{1\leq i\leq n}$ in $\bar D\times \K$ and all $n\in \Nm$, we have
\begin{equation}\label{eq:convergence in time for killed Markov 1 in intro}
\begin{split}
\Big\lvert\Big\lvert\Law_{\frac{1}{n}\sum_{i=1}^n\delta_{(x^i,\eta^i)}}((X_t,\eta_t)\lvert \tau_{\partial}>t)-\frac{\sum_{i=1}^n\phi(x^i)\pi\otimes\delta_{\eta^i}}{\sum_{i=1}^n\phi(x^i)}\Big\rvert\Big\rvert_{\TV}\leq c_t,\quad 0\leq t<\infty.
\end{split}
\end{equation}
\end{prop}

\subsection{Heuristics for the proof of Theorem \ref{theo:Convergence to FV diffusion}}\label{subsection:outline of proof strategy}
\subsubsection*{The principal difficulty to be addressed}
Méléard and Tran considered in \cite{Meleard2012} the ancestries of a similar particle system in \cite{Meleard2012}. There the individuals in the population have a trait and an age, with the individuals giving birth (passing on their trait) and dying in an age-dependent manner. The age component plays a similar role to spatial position in the present article.  However, aside from obtaining a different scaling limit, they also employed a different proof strategy. 

The proof of Méléard and Tran in \cite{Meleard2012} extended to the particle system setting the strategy of Kurtz \cite{Kurtz1992} and Ball, Kurtz, Popovic and Rempala \cite{Ball2006}, which concerned diffusions. In contrast, the proof in the present article extends to the particle system setting techniques of Katzenberger \cite{Katzenberger1991} (the author is not aware of this technique previously having been extended to the particle system setting), which also concerned diffusion processes. This is necessitated by the following qualitative difference between the two particle systems. 

In \cite{Meleard2012}, individuals have a trait $x$ (the slow variable) and an age $a$ (the fast variable). The speed-up of the timescale is given by the parameter $n$. On the fast timescale, they give birth at rate $n r(x,a)+b(x,a)$, whilst dying at rate $nr(x,a)+d(x,a)$. We observe that the fast term, $nr(x,a)$, is the same in both the former and the latter. Consequentially, when they formulate the corresponding martingale problem, the slow variable does not have a large drift term on the fast timescale. The terms $b(x,a)$ and $d(x,a)$ may change quickly due to the fast evolution of the age term $a$ - this is dealt with via averaging - but they remain $\calO(1)$ on the fast timescale.

We may contrast this with the Fleming-Viot multi-colour process. We recall that the time change is $t\mapsto Nt$. We consider a test function $f\in C_b(\K)$ and observe that on the fast timescale the empirical measure of the colours evaluated against $f$, $\chi^N_{Nt}(f)$, satisfies
\begin{equation}\label{eq:sde for empirical measure of the colours proof explanation}
d\chi^N_{Nt}(f)=
\sum_{i=1}^N\kappa(X^i_{Nt})\Big[\frac{1}{N-1}\sum_{j\neq i}[f(\eta^j_{Nt})-f(\eta^i_{Nt})]\Big]dt+d\text{(Martingale)}_t.
\end{equation}
We see that the drift term is of $\calO(N)$ on the fast timescale. In particular the change in position of an individual particle has an $\calO(1)$ effect on the drift. A large deviations principle for the Fleming-Viot multi-colour process would provide controls on the drift valid over a sufficiently large timescale (a LDP for the Fleming-Viot particle system driven by Brownian motion with soft killing was established by Grigorescu in \cite{Grigorescu2007}), but would only control the drift on a fast timescale to $\calO(N)$. Since microscopic fluctuations in the position of individual particles have an $\calO(1)$ effect on the drift, there would not seem to be any hope of obtaining adequate controls on the drift term in order to apply a compactness-uniqueness argument (in which one characterises the martingale problem solved by subsequential limits).

The key idea allowing us to deal with these large drift terms will be to consider the tilted empirical measure $\calY^N_t$, which we recall was given in \eqref{eq:tilted empirical measure} as
\[
\calY^N_t:=\frac{\frac{1}{N}\sum_{i=1}^N\phi(X^{i}_t)\delta_{\eta^{i}_t}}{\frac{1}{N}\sum_{i=1}^N\phi(X^{i}_t)}\in\calP(\K).
\]

\subsubsection*{Motivation for choosing $\calY^N_t$}

We take inspiration from Katzeberger's approach in \cite{Katzenberger1991}. Consider a dynamical system in Euclidean space, $\dot{x}_t=b(x_t)$, with an attractive manifold of equilibrium $\calM$ and flow map $\varphi(x,s)$. Katzenberger \cite{Katzenberger1991} established (under reasonable conditions) that the long term dynamics of the randomly perturbed dynamical system,
\begin{equation}\label{eq:simple randomly perturbed dynamical system Katzenberger analogy}
    dx^{\epsilon}_t=b(x^{\epsilon}_t)dt+\epsilon dW_t,
\end{equation}
can be obtained by considering the following nonlinear projection onto the manifold of equilibria,
\begin{equation}\label{eq:pi function Katzenberger analogy}
\varpi(x):=\lim_{\substack{s\ra\infty}}\varphi(x,s)\in \calM. 
\end{equation}
We summarise Katzenberger's idea as follows. Since $\nabla \varpi\cdot b\equiv 0$, the Stratanovich chain rule implies that
\[
d\varpi(x^{\epsilon}_t)=\epsilon \nabla \varpi(x^{\epsilon}_t)\circ dW_t.
\]
In particular, the large drift term has been eliminated from the above expression. We may rescale time to see that $\varpi(x^{\epsilon}_{\frac{t}{\epsilon^2}})$ satisfies
\[
d\varpi(x^{\epsilon}_{\frac{t}{\epsilon^2}})=\nabla \varpi(x^{\epsilon}_{\frac{t}{\epsilon^2}})\circ d\tilde{W}_t,
\]
whereby $\tilde{W}_t$ is the Brownian motion $\tilde{W}_t:=\epsilon W_{\frac{t}{\epsilon^2}}$. Since the dynamical system will be pushed towards the attractive manifold of equilibrium on a fast timescale, one can then argue that
\begin{equation}\label{eq:Katzenberger close to manifold of equilibria}
x^{\epsilon}_{\frac{t}{\epsilon^2}}\approx \varpi(x^{\epsilon}_{\frac{t}{\epsilon^2}}).
\end{equation}
We can therefore obtain a scaling limit for $x^{\epsilon}_{\frac{t}{\epsilon^2}}$. This scaling limit is a diffusion on $\calM$.

Whilst Katzenberger's results in \cite{Katzenberger1991} were restricted to finite-dimensions, we may ask what the analogue of $\varpi(x^{\epsilon}_t)$ is in the present setting? We will see that $\calY^N_t$ can be thought of as being analogous to the quantity $\varpi(x^{\epsilon}_t)$ considered by Katzenberger. 

We denote by $((X_t,\eta_t))_{0\leq t<\tau_{\partial}}$ the killed Markov process defined in Definition \ref{defin:limit for multi colour over fixed times}. It follows from Theorem \ref{theo:hydrodynamic limit for multicolour process} that we can think of the Fleming-Viot multi-colour process as a random perturbation of the dynamical system with flow map
\begin{equation}\label{eq:flow of laws for katzenberger analogy}
\calP(\bar D\times \K)\times [0,\infty)\ni (\upsilon,s)\mapsto \Law_{\upsilon}((X_s,\eta_s)\lvert \tau_{\partial}>s)\in \calP(\bar D\times \K).
\end{equation}
Proposition \ref{prop:convergence in time for killed Markov 1 in main theorem proof} provides for the large-time limits of this flow. We therefore see from Proposition \ref{prop:convergence in time for killed Markov 1 in main theorem proof} that the analogue of $\varpi(x^{\epsilon}_t)$ is given by
\[
\frac{\frac{1}{N}\sum_{i=1}^N\phi(X^{N,i}_t)\pi\otimes\delta_{\eta^{N,i}_t}}{\frac{1}{N}\sum_{i=1}^N\phi(X^{N,i}_t)}=\pi\otimes \calY^N_t.
\]
We discard $\pi$ since it is constant, leaving only $\calY^N_t$. 

There is a second heuristic reason for examining $\calY^N_t$. If $x(t)$ and $y(t)$ both satisfy the ODEs $\dot{x}=c(t)x$ and $\dot{y}=c(t)y$ for the same $c(t)$, then $\frac{y(t)}{x(t)}$ is constant. If we now instead consider the SDEs $dX_t=c_tX_tdt+\epsilon dW_t$ and $dY_t=c_tY_tdt+\epsilon dW_t$, $\frac{Y_t}{X_t}$ will satisfy an SDE with only $\calO(\epsilon^2)$ drift terms, since the $\calO(1)$ terms will cancel out as in the deterministic case (one can check this using Ito's lemma).

We now define for $\E\in\calB(\K)$ the following, which shall be used throughout the proof of Theorem \ref{theo:Convergence to FV diffusion}, 
\begin{equation}\label{eq:quantity PNE etc}
P^{N,\E}_t:=\frac{1}{N}\sum_{i=1}^N\Ind_{\eta^i_t\in \E}\phi(X^i_t),\;\; Q^N_t:=P^{N,\K}=\frac{1}{N}\sum_{i=1}^N\phi(X^i_t)\;\;\text{and}\;\; Y^{N,\E}_t:=\calY^N_t(\E)=\frac{P^{N,\E}_t}{Q^N_t}.
\end{equation}
The important point is that, to leading order, both $P^{N,\E}$ and $Q^N$ evolve with drift terms proportional to themselves, with the same constant of proportionality. Indeed on the slow timescale the killed process $X_t$ satisfies
\[
d\phi(X_t)=L\phi(X_t)+\text{martingale terms}=-\lambda \phi(X_t)+\text{martingale terms.}
\]
Therefore between jumps, and including the process of killing the particles, the quantities $P^{N,\E}_t$ and $Q^N_t$ evolve with drift terms $-\lambda P^{N,\E}_t dt$ and $-\lambda Q^N_t dt$ respectively. Furthermore if particle $(X^{N,i},\eta^{N,i})$ dies at time $t$, $\frac{1}{N}\phi(X^{N,i}_t)\Ind(\eta^{N,i}_t\in \E)$ (respectively $\frac{1}{N}\phi(X^{N,i}_t)$) is added to the value of $P^{N,\E}_t$ (respectively $Q^{N,\E}_t$), the expected value of which is $P^{N,\E}_{t-}+\calO(\frac{1}{N})$ (respectively $Q^{N,\E}_{t-}+\calO(\frac{1}{N})$). This occurs at Poisson rate $\kappa(X^{i}_t)$. Thus, after the time-change $t\mapsto Nt$, we can write
\begin{equation}\label{eq:expression for PE Q proof explanation}
\begin{split}
dP^{N,\E}_{Nt}=\Big[-\lambda N+\sum_{i=1}^N\kappa(X^i_t)\Big]P^{N,\E}_{Nt}dt+\calO(1)dt+d(\text{martingale})_t,\\
dQ^{N}_{Nt}=\Big[-\lambda N+\sum_{i=1}^N\kappa(X^i_t)\Big]Q^{N}_{Nt}dt+\calO(1)dt+d(\text{martingale})_t.
\end{split}
\end{equation}
In particular, on the fast timescale given by $t\mapsto Nt$, both $P^{N,\E}_{Nt}$ and $Q^N_{Nt}$ both evolve with drift proportional to themselves, with the same constant of proportionality given by
\begin{equation}\label{eq:constant of proportionalty proof explanation}
-\lambda N+\sum_{i=1}^N\kappa(X^i_t). 
\end{equation}
We observe that the change in position of an individual particle has an $\calO(1)$ effect on the constant of proportionality \eqref{eq:constant of proportionalty proof explanation}. However these large effects cancel out by placing the normalisation at microscopic scale in the denominator, as the constant of proportionality in both the numerator and denominator must be the same.

From these considerations we see that, having rescaled time by $t\mapsto Nt$, $\calY^N_{Nt}$ should satisfy an SDE with $\calO(1)$ drift terms. It is straightforward to see that the martingale terms will have $\calO(1)$ quadratic variation on this timescale. It follows that $\calY^N_{Nt}$ should be susceptible to a compactness-uniqueness argument, in which we establish tightness before uniquely characterising subsequential limits by characterising their drift and quadratic variation. We shall thereby obtain a scaling limit for $\calY^N_{Nt}$. We note that since the leading order terms in \eqref{eq:expression for PE Q proof explanation} cancel out, we shall need to calculate the ``$\calO(1)dt$'' higher order terms, which is responsible for much of the computational complexity in the proof of Theorem \ref{theo:Convergence to FV diffusion}.

\subsubsection*{The relationship between $\chi^{N}_{Nt}$ and $\calY^N_{Nt}$}

The above will allow us to characterise the limit in distribution of $(\calY^N_{Nt})_{t\geq 0}$. Our goal, however, is to characterise the limit in distribution of $(\chi^N_{Nt})_{t\geq 0}$. We would therefore like to relate $\calY^N_{Nt}$ with $\chi^{N}_{Nt}$. 

The key observation here is that, on the original slow timescale, the colour of a particle and its spatial position become ``independent'' after an $\calO(1)$ time. To be more precise, for any given $A\subseteq \K$, the spatial profile of particles whose colours belong to $A$,
\[
\frac{\sum_{i=1}^N\Ind(\eta^i_t\in A)\delta_{X^i_t}}{\lvert\{i:\eta^i_t\in A\}\rvert},
\]
converges over an $\calO(1)$ timescale to the quasi-stationary distribution $\pi$, a deterministic profile. Thus for different subsets $A,B\subseteq \K$, the number of particles with colours belonging to $A$ and $B$ may well be different, but the spatial profiles of the two sets of particles will be the same for large $N$. Since the particles corresponding to different colours have the same spatial profile, weighting the empirical measure of the colours according to the right eigenfunction evaluated at the corresponding spatial positions will have no effect. It follows that $\chi^N_t$ and $\calY^N_t$ will be close after an $\calO(1)$ time. On the fast timescale, $\chi^N_{Nt}$ will therefore be close to $\calY^N_{Nt}$. This is analogous to the second step in Katzenberger's approach in \cite{Katzenberger1991}, described above in \eqref{eq:Katzenberger close to manifold of equilibria}. 

The proof of Theorem \ref{theo:Convergence to FV diffusion} will therefore follow by establishing that $(\calY^N_{Nt})_{t\geq 0}$ converges to the Wright-Fisher process, and showing that $\chi^N_{Nt}$ is close to $\calY^N_{Nt}$.

\subsection{Why is the limit a Wright-Fisher process?}

It follows from the above heuristic that $\chi^N_t$ should evolve over an $\calO(N)$ timescale, and that $\chi^N_{Nt}$ should converge to some $\calP(\K)$-valued process (at least on subsequences). In the proof of Theorem \ref{theo:Convergence to FV diffusion}, we will calculate that the limit is a Wright-Fisher process. However, it is not readily apparent from this why the limit should necessarily be a Wright-Fisher process. We offer here a heuristic argument for why we should expect the limit to be a Wright-Fisher process.

We let $(\hat{\nu}_t)_{t\geq 0}$ be the limit to be determined of $(\chi^N_{Nt})_{t\geq 0}$ (perhaps along a subsequence). By the aforedescribed separation of timescales phenomenon, this will be a $\calP(\K)$-valued process, with the dependence on the spatial component ``averaged out''. We consider an arbitrary measurable map $\iota:\K\ra \K$. We can think of $\iota$ as relabelling the colours. The key observation is that $\{(X^{N,i}_t,\iota(\eta^{N,i}_t)):1\leq i\leq N\}$ is itself a Fleming-Viot multi-colour process - the Fleming-Viot multi-colour process remains one after relabelling the colours. It follows that whatever dynamics $(\hat{\nu}_t)_{t\geq 0}$ has, $(\iota_{\#}\hat{\nu}_t)_{t\geq 0}$ must have the same dynamics. This allows us both to exchange colours and to relabel different colours as the same colour. 

It follows that there should exist continuous functions
\[
b,\sigma_{11}:[0,1]\ra \Rm\quad\text{and}\quad \sigma_{12}:\{(p,q)\in [0,1]^2:p+q\leq 1\}\ra \Rm
\]
such that the following are continuous martingales for all disjoint $A_1,A_2\in\calB(\K)$,
\[
\begin{split}
\hat{\nu}_t(A_1)-\int_0^tb(\hat{\nu}_s(A_1))ds,\quad (\hat{\nu}_t(A_1))^2-\int_0^t\sigma_{11}(\hat{\nu}_s(A))ds,\\\text{and}\quad  (\hat{\nu}_t(A_1))(\hat{\nu}_t(A_2))-\int_0^t\sigma_{12}(\hat{\nu}_s(A_1),\hat{\nu}_s(A_2))ds.
\end{split}
\]
Moreover, since a colour of mass $p$ and a colour of mass $q$ can be relabelled to be a single colour of mass $p+q$, it is clear that
\[
\begin{split}
b(p+q)=b(p)+b(q),\;\; \sigma_{11}(p+q)=\sigma_{11}(p)+\sigma_{11}(q)+2\sigma_{12}(p,q),\;\; 0\leq p,q\leq p+q\leq 1\\ \sigma_{12}(p_1+p_2,q_1+q_2)=\sum_{1\leq i,j\leq 2}\sigma_{12}(p_i,q_j),\;\; 0\leq p_1,p_2,q_1,q_2\leq p_1+p_2+q_1+q_2\leq 1.
\end{split}
\]
Furthermore, the whole colour space $\K$ must have total mass $1$, so $\hat{\nu}_t(\K)\equiv 1$. From these considerations, we see that the only possibility is that, for some constant $\theta$,
\[
b\equiv 0, \quad \sigma_{11}(p)=\frac{\theta}{2}(p-p^2)\quad\text{and}\quad \sigma_{12}(p,q)=-\frac{\theta}{2}pq.
\]
We recognise the Wright-Fisher diffusion described in Subsection \ref{subsection:WF process intro}. In light of Proposition \ref{prop:basic facts WF superprocess}, it is therefore natural that our unknown limit $(\hat{\nu}_t)_{t\geq 0}$ should be a Wright-Fisher process.

\subsection{Hard catalyst killing}

The setting of the present paper - in which the Fleming-Viot particle system is driven by diffusions with soft killing - has been chosen to establish the connection between the Fleming-Viot process and the Wright-Fisher process with a minimum of technical difficulties. Nevertheless, in Section \ref{section:hard killing extension} we will extend this connection to the original setting considered by Burdzy, Hołyst and March \cite{Burdzy2000}, in which the Fleming-Viot particle system is driven by Brownian motion with instantaneous killing at the boundary (\textit{hard killing}). To avoid switching back and forth between Fleming-Viot particle systems with different dynamics, we will only consider the case of hard killing in Section \ref{section:hard killing extension}, the final section prior to the appendix, and in Appendix \ref{appendix:hard killing}. Our results in the case of hard killing are therefore stated and proved in Section \ref{section:hard killing extension}.

We emphasise that the proof strategy employed in the present paper may be applied to the Fleming-Viot particle system driven by more general killed Markov processes. The principal requirements to apply this proof strategy are that: 
\begin{enumerate}
\item
the driving killed Markov process $(X_t)_{0\leq t<\tau_{\partial}}$ is Feller;
\item
its infinitesimal generator has a positive, continuous and bounded principal right eigenfunction $\phi$;
\item
$\Law_{\mu}(X_t\lvert \tau_{\partial}>t)$ converges to a unique quasi-stationary distribution for any initial condition $\mu$;
\item\label{enum:requirements in intro requirement to constrain mass large timescale}
we can constrain the empirical measure of the spatial positions of the particles $m_t^N$ to a tight set of measures over any $\calO(N)$ timescale, precluding in particular the mass from accumulating at the boundary.
\end{enumerate}

In the case of hard killing at the boundary in a bounded domain, the main additional difficulty is to establish Requirement $4$. We will obtain such controls for the Fleming-Viot particle system driven by Brownian motion with hard killing in Section \ref{section:hard killing extension}. With these controls in hand, the extension of our results to this setting proceeds by essentially the same proof. 

When the domain is unbounded, the situation is much more delicate. For diffusions on the positive real line $\Rm_{>0}$ with hard killing at $0$, one could probably establish similar results for Ornstein-Uhlenbeck dynamics, using the strong negative drift to control the particles far away from $0$ over an $\calO(N)$ time scale. For this process the principal right eigenfunction is unbounded (it's given by $\phi(x)=x$) - instead strong controls on the mass of particles far away from $0$ (where $\phi$ is large) over an $\calO(N)$ timescale would be required to replace the boundedness of $\phi$. To be more precise we would need to show, for any $T<\infty$, that $\sup_{0\leq t\leq NT}\frac{1}{N}\sum_{1\leq i\leq N}(\phi(X^{N,i}_t))^2$ is bounded by some uniform constant with probability arbitrarily close to $1$, uniformly in $N$. On the other hand, we should not expect the Fleming-Viot particle system driven by Brownian motion with drift $-1$ to have a Wright-Fisher process scaling limit, this drift being too weak to adequately control the particles. Indeed, it is a hard open problem to show that the genealogy of the very similar $N$-BBM is given by a Bolthausen-Sznitman coalescent \cite[p.1066]{Maillard2016}.

\subsection{Structure of the paper}

A summary of the notation which we shall need for our proof is given in Section \ref{section:notation for the proof}. The proof of Theorem \ref{theo:Convergence to FV diffusion} shall rely on a number of calculations of the quantity $\calY^N_t$, defined in \eqref{eq:tilted empirical measure}. To avoid obscuring our proof with calculations, we will carry out these calculations in Section \ref{section: characterisation of Y}. We shall then prove Theorem \ref{theo:Convergence to FV diffusion} in Section \ref{section:proof of convergence to FV diffusion}. We will extend our results to the Fleming-Viot multi-colour process driven by Brownian motion with instantaneous killing at the boundary in Section \ref{section:hard killing extension}. We conclude with the appendix.

\section{Notation for the proof of Theorem \ref{theo:Convergence to FV diffusion}}\label{section:notation for the proof}

We recall from \eqref{eq:quantity PNE etc} that we define for $\E\in\calB(\K)$,
\[
P^{N,\E}_t:=\frac{1}{N}\sum_{i=1}^N\Ind_{\eta^i_t\in \E}\phi(X^i_t),\;\; Q^N_t:=P^{N,\K}=\frac{1}{N}\sum_{i=1}^N\phi(X^i_t)\;\;\text{and}\;\; Y^{N,\E}_t:=\calY^N_t(\E)=\frac{P^{N,\E}_t}{Q^N_t}.
\]
We recall the definition of $m^N_t$ and $\chi^N_t$ from \eqref{eq:spatial and colour empirical measures}, and further define $m^{N,\E}_t$ for $\E\in\calB(\K)$,
\begin{equation}
m^N_t:=\frac{1}{N}\sum_{i=1}^N\delta_{X^{N,i}_t},\quad  \chi^{N}_t:=\frac{1}{N}\sum_{i=1}^N\delta_{\eta^{i}_t}\quad\text{and}\quad m_t^{N,\E}:=\sum_{i=1}^N\Ind(\eta^{i}_t\in \E)\delta_{X^{i}_t}=m^N_t(\E).
\end{equation}
We recall from Appendix \ref{appendix:reflected diffusions} that the infinitesimal generator of the reflected diffusion with (respectively without) soft killing is denoted by $L$ (respectively $L_0$). We further recall that the Carre du champs operator of the latter is denoted as $\Gamma_0$, and is defined on the algebra $\calA$. This algebra contains the principal right eigenfunction $\phi$ of $L$, by Theorem \ref{theo:convergence to QSD for reflected diffusion with soft killing}. We further define
\begin{equation}\label{eq:lambdaNE definition}
\begin{split}
\Lambda^{N,\E}_t:=\langle m^{N,\E}_t,\Gamma_{0}(\phi)+\kappa\phi^2\rangle +\langle m^{N,\E}_t,\phi^2\rangle \langle m^{N}_t,\kappa\rangle\;\; \text{for}\;\; \E\in\calB(\K),\;\;\text{and}\;\; \Lambda^N_t:=\Lambda^{N,\K}_t.
\end{split}
\end{equation}

\subsection{$\calO$ Notation}\label{section:O and U notation}

The following notation shall significantly simplify our calculations.

For any finite variation process $(X_t)_{0\leq t<\infty}$ we write $V_t(X)$ for the total variation
\begin{equation}\label{eq:finite variation notation}
V_t(X)=\sup_{0=t_0<t_1<\ldots<t_n=t}\sum_{i=0}^{n-1}\lvert X_{t_{i+1}}-X_{t_i}\rvert.
\end{equation}
Moreover for all c\`adl\`ag processes $(X_t)_{0\leq t<\infty}$ we write
\begin{equation}
\Delta X_t= X_t-X_{t-}.
\end{equation}

Given some family of random variables $\{X^{N}:N\in\mathbb{N}\}$ and non-negative random variables $\{Y^{N}:N\in\mathbb{N}\}$, we say that $X^{N}=\calO(Y^{N})$ if there exists a uniform constant $C<\infty$ such that $\lvert X^{N}\rvert\leq CY^{N}$. Note that we shall abuse notation by using an equals sign, rather than an inclusion sign.

We now define the notion of process sequence class. Given sequences of processes $\{(X^N_t)_{t\geq 0}:N\in\mathbb{N}\}$ and $\{(Y^N_t)_{t\geq 0}:N\in\mathbb{N}\}$, we say that:
\begin{enumerate}
\item\label{enum:process sequence classes MG controlled QV}
$X^N_t=\calO^{\MG}_t(Y^N)$ if for all $N\geq N_0$ (for some $N_0<\infty$) and for some $C<\infty$, $X^N_t$ is a martingale whose quadratic variation is such that
\begin{equation}\label{eq:equation for OMG notation}
[X^N]_t-\int_0^tCY^N_sds\quad\text{is a supermartingale.}
\end{equation}
\item
$X^N_t=\calO^{\FV}_t(Y^N)$ if for all $N\geq N_0$ (for some $N_0<\infty$) and for some $C<\infty$, $X^N_t$ is a finite variation process whose total variation is such that
\begin{equation}\label{eq:equation for OFV notation}
V_t(X^N)-\int_0^tCY^N_sds\quad\text{is a supermartingale.}
\end{equation}
\item\label{enum:process seqeunce classes FV controlled jumps}
$X^N_t=\calO^{\Delta}_t(Y^N)$ if for all $N\geq N_0$ (for some $N_0<\infty$) and for some $C<\infty$, $X^N_t$ is such that
\begin{equation}\label{eq:equation for ODELTA notation}
\lvert \Delta X^N_t\rvert\leq CY^N_{t-}\quad\text{for all }0\leq t<\infty,\quad \text{almost surely.}
\end{equation}
\item $X_t^N=\calO^{\Cts}_t$ or $X^N_t=\calO^{\Lip}_t$ if for all $N\geq N_0$ (for some $N_0<\infty$), $X_t^N$ has continuous (respectively Lipschitz) sample paths, almost surely.
\end{enumerate}
We refer to $\calO^{\MG}_t(Y^N),\calO^{\FV}_t(Y^N)$ and $\calO^{\FV}_t(Y^N)$ for $((Y^N_t)_{0\leq t<\infty}:N\in\mathbb{N})$ a given sequence of processes, and $\calO^{\Cts}_t$ and $\calO_t^{\Lip}$, as process sequence classes. Note that as with sequences of random variables, we abuse notation by using an equals sign rather than an inclusion sign.

Suppose that we have constants $r_N>0$ ($N\in\mathbb{N}$). For a given sequence of processes $Y^N$, write $Z^N_s:=Y^N_{r_Ns}$. The statements $X^N_t=\calO^{\MG}_t(Y^N_{r_N\cdot})$, $X^N_t=\calO^{\FV}_t(Y^N_{r_N\cdot})$ and $X^N_t=\calO^{\Delta}_t(Y^N_{r_N\cdot})$ should be interpreted as the statements $X^N_t=\calO^{\MG}_t(Z^N_{\cdot})$, $X^N_t=\calO^{\FV}_t(Z^N_{\cdot})$ and $X^N_t=\calO^{\Delta}_t(Z^N_{\cdot})$ respectively.

Given an index set $\mathbb{A}$, a family of sequences of processes $\{((X^{N,\alpha}_t)_{0\leq t<\infty})_{N=1}^{\infty}:\alpha\in \mathbb{A}\}$, and a family of process sequence classes $\{\calA^{N,\alpha}_t:\alpha\in\mathbb{A}\}$, we say that $X^{N,\alpha}_t=\calA^{N,\alpha}_t$ uniformly if the constants $C^{\alpha}$ and $N^{\alpha}_0$ used to define $X^{N,\alpha}_t=\calA^{N,\alpha}_t$ as in \ref{enum:process sequence classes MG controlled QV} - \ref{enum:process seqeunce classes FV controlled jumps} can be chosen uniformly in $\alpha\in\mathbb{A}$.

It will be useful to take the sum and intersection of process sequence classes and specific sequences of processes. To be more precise, for any process sequence classes $\calA^N_t$ and $\calB^N_t$, and the sequence of processes $F^N_t$, we say that:
\begin{enumerate}
\item
$X^N_t=\calA^N_t\cap\calB^N_t$ if $X^N_t=\calA^N_t$ and $X^N_t=\calB^N_t$;
\item
$X^N_t=F^N_t+\calA^N_t$ if there exists a sequence of processes $G^N_t$ such that $G^N_t=\calA^N_t$ and $X^N_t=F^N_t+G^N_t$;
\item
$X^N_t=\calA^N_t+\calB^N_t$ if there exists sequences of processes $G^N_t$ and $H^N_t$ such that $G^N_t=\calA^N_t$, $H^N_t=\calB^N_t$ and $X^N_t=G^N_t+H^N_t$;
\item
$dX_t=dF^N_t+d\calA^N_t+d\calB^N_t$ if there exists sequences of processes $G^N_t$ and $H^N_t$ such that $G^N_t=\calA^N_t$, $H^N_t=\calB^N_t$ and $dX^N_t=dF^N_t+dG^N_t+dH^N_t$.
\end{enumerate}
For example, if $X^N_t=\calO_t^{\MG}(1)+\calO_t^{\FV}(\frac{1}{N})\cap\calO^{\Delta}_t(\frac{1}{N^2})$ then there exists $G^N_t$ and $H^N_t$ such that $X^N_t=G^N_t+H^N_t$ whereby $G^N_t=\calO_t^{\MG}(1)$ and $H^N_t=\calO_t^{\FV}(\frac{1}{N})\cap\calO^{\Delta}_t(\frac{1}{N^2})$. Thus $X^N_t=\calO_t^{\MG}(1)+\calO_t^{\FV}(\frac{1}{N})\cap\calO^{\Delta}_t(\frac{1}{N^2})$ means that for some $0<C<\infty$ there exists for all $N$ large enough martingales $G^{N}_t$ and finite-variation processes $H^{N}_t$ such that:

\begin{align*}
X^N_t=G^N_t+H^N_t,\\
\quad [G^N]_t-Ct\quad\text{is a supermartingale since}\quad G^N_t=\calO^{\MG}_t(1),\\
V_t(H^N)-\frac{t}{N}\quad\text{is a supermartingale since}\quad H^N=\calO^{\FV}_t(\frac{1}{N}),\\
\text{and}\quad  \lvert \Delta Z^N_t\rvert\leq \frac{C}{N^2}\quad\text{for all}\quad 0\leq t<\infty,\quad \text{almost surely, since}\quad H^N_t= \calO^{\Delta}_t(\frac{1}{N^2}).
\end{align*}

\section{Characterisation of $\calY^N_t$}\label{section: characterisation of Y}

In the proof of Theorem \ref{theo:Convergence to FV diffusion}, we will obtain a scaling limit for the tilted empirical measure of the colours on a fast timescale, $(\calY^N_{Nt})_{t\geq 0}$. This will rely on various calculations characterising its drift and quadratic variation. To avoid obscuring the proof of Theorem \ref{theo:Convergence to FV diffusion} with calculations, we perform these calculations here.

In this section, we write $(\Omega,\calG,(\calG_t)_{t\geq 0},\Pm)$ for the underlying filtered probability space.

\begin{rmk}
In the present section, all statements as to processes belonging to various process sequence classes should be interpreted as being uniform over all choices $\E,\F\in\calB(\K)$ (or over all sequences of $\calG_0$-measurable random $\E^N,\F^N\in\calB(\K)$, in the case of Part \ref{enum:Thm 8.2 true for random sets} of Theorem \ref{theo:characterisation of Y}). 
\end{rmk}

We recall that
\[
\begin{split}
\calY^N_t:=\frac{\frac{1}{N}\sum_{i=1}^N\phi(X^{i}_t)\delta_{\eta^{i}_t}}{\frac{1}{N}\sum_{i=1}^N\phi(X^{i}_t)}.
\end{split}
\]
In this section, we prove the following theorem.
\begin{theo}\label{theo:characterisation of Y}
We have the following, uniformly over all choices of $\calE,\calF\in\calB(\K)$:
\begin{enumerate}
\item\label{enum:bound on cov of YE YF}
The covariation $[Y^{N,\E},Y^{N,\F}]_t$ is such that $[Y^{N,\E},Y^{N,\F}]_t=\calO^{\FV}_t(\frac{Y^{N,\E}Y^{N,\F}}{N})$ for disjoint $\E, \F\in\calB(\K)$.
\item\label{enum:Y in terms of K and extra terms}
There exists martingales $\calK^{N,\E}_t$ for $\E\in\calB(\K)$ such that $Y^{N,\E}_t$ satisfies
\begin{equation}\label{eq:Y in terms of K and extra terms}
\begin{split}
Y^{N,\E}_t=Y^{N,\E}_0\\
+\int_0^t\Big[-\frac{1}{(N-1)Q_s^N}\langle m^{N,\E}_s-Y^{N,\E}_s m^N_s,\kappa\phi\rangle
-\frac{1}{NQ^N_s}\langle Y^{N,\E}_sm^N_s-m^{N,\E}_s,\kappa\phi\rangle\\+\frac{1}{N(Q^N_s)^2}\big(Y^{N,\E}_s\Lambda^{N}_s-\Lambda^{N,\E}_s\big)
\Big]ds+\calK^{N,\E}_t
+\calO_t^{\MG}\big(\frac{Y^{N,\E}}{N^3}\big)+\calO_t^{\FV}\big(\frac{Y^{N,\E}}{N^2}\big)\cap\calO^{\Delta}_t\big(\frac{1}{N^3}\big)
\end{split}
\end{equation}
for $\E\in\calB(\K)$ and such that
\begin{equation}\label{eq:covariation of KE and KF}
\begin{split}
[\calK^{N,\E},\calK^{N,\F}]_t=\int_0^t\frac{1}{N(Q^N_s)^2}\Big[
\Lambda^{N,\E\cap \F}_s
-Y^{N,\E}_s \Lambda^{N,\F}_s
-Y^{N,\F}_s \Lambda^{N,\E}_s +Y^{N,\E}_sY^{N,\F}_s \Lambda^{N}_s \Big]ds
\\+\calO_t^{\MG}\big(\frac{Y^{\E}Y^{\F}+Y^{\E\cap \F}}{N^3}\big)
+\calO_t^{\FV}\Big(\frac{Y^{\E}Y^{\F}+Y^{\E\cap \F}}{N^2}\Big)\cap\calO^{\Cts}_t\quad\text{for all}\quad \E,\F\in\calB(\K).
\end{split}
\end{equation}
\item\label{enum:Y O FV plus U MG}
Furthermore $Y^{N,\E}_t$ satisfies
\begin{equation}\label{eq:Y O FV plus U MG}
Y^{N,\E}_t=\Big[\calO^{\FV}_t\big(\frac{Y^{N,\E}}{N}\big)+\calO^{\MG}_t\big(\frac{Y^{N,\E}}{N}\big)\Big]\cap\calO^{\Delta}_t(\frac{1}{N}).
\end{equation}
\item
Parts \ref{enum:bound on cov of YE YF}- \ref{enum:Y O FV plus U MG} remain true if $\E$ and $\F$ are replaced with a sequence of $\sigma_0$-measurable random sets $\calE^N$ and $\calF^N$.\label{enum:Thm 8.2 true for random sets}
\end{enumerate}
\end{theo}

\subsection{Proof of Theorem \ref{theo:characterisation of Y}}

We firstly introduce some definitions. We define
\begin{equation}\label{eq:defin function F}
F(\vec{r})=\frac{p}{q}\quad \text{for}\quad\vec{r}=\begin{pmatrix}
p\\
q
\end{pmatrix}\in \Rm_{>0}^2.
\end{equation}
We write $H=H(\vec{r})$ for the Hessian and calculate
\begin{equation}
\begin{split}
\nabla F(\vec{r})=\begin{pmatrix}
\frac{1}{q} \\ -\frac{p}{q^2}
\end{pmatrix}\quad\text{and}\quad 
H(\vec{r})=\begin{pmatrix}
0 & -\frac{1}{q^2}\\
-\frac{1}{q^2} & 2\frac{p}{q^3}
\end{pmatrix}\quad \text{for}\quad\vec{r}=\begin{pmatrix}
p\\
q
\end{pmatrix}\in \Rm_{>0}^2.
\end{split}
\label{eq:formula for nabla F and Hessian}
\end{equation}
We have the key property
\begin{equation}
\nabla F\cdot \vec{r}=0\quad\text{and}\quad \vec{r}\cdot H(F)\vec{r}=0\quad \text{for}\quad\vec{r}=\begin{pmatrix}
p\\
q
\end{pmatrix}\in \Rm_{>0}^2.
\label{eq:formula for nabla F dot r}
\end{equation}
We further define
\begin{equation}\label{eq:defin vec R}
\begin{split}
\vec{R}^{N,\E}_t:=\begin{pmatrix}
P^{N,\E}_t\\
Q^{N}_t
\end{pmatrix}\quad\text{so that}\quad Y^{N,\E}_t=F(\vec{R}^{N,\E}_t).
\end{split}
\end{equation}
We shall firstly establish the following proposition, which characterises $P^{N,E}$.
\begin{prop}\label{prop:proposition characterising pk}
We have for all $\E\in\calB(\K)$ that
\begin{equation}
dP^{N,\E}_t=P^{N,\E}_t\big(-\lambda +\frac{N}{N-1}\langle m^{N}_t,\kappa\rangle\big)dt -\frac{1}{N-1}\langle m^{N,\E}_t,\kappa\phi\rangle  dt+dM^{N,\E}_t,
\label{eq:sde for pk}
\end{equation}
whereby $M^{N,\E}$ are martingales which satisfy for all $\E,\F\in\calB(\K)$,
\begin{equation}
\begin{split}
[M^{N,\E},M^{N,\F}]_t=-\frac{1}{N}\int_0^tP^{N,\E}_s\langle m^{N,\F}_s,\kappa\phi\rangle+P^{N,\F}_s\langle m^{N,\E}_s,\kappa\phi\rangle ds\\
+\frac{1}{N}\int_0^t\Lambda^{N,\E\cap\F}_s ds
+\calO_t^{\MG}\big(\frac{P^{\E}P^{\F}+P^{\E\cap \F}}{N^3}\big)+\calO_t^{\FV}\big(\frac{P^{\E\cap \F}}{N^2}\big)\cap\calO^{\Lip}_t.
\end{split}
\label{eq:covariation for ME MF}
\end{equation}
We write $M^N_t$ for $M^{N,\K}_t$. 
\end{prop}

We will then establish Part \ref{enum:bound on cov of YE YF} of Theorem \ref{theo:characterisation of Y}, followed by the following version of Ito's lemma.

\begin{lem}[Ito's Lemma]\label{lem:Ito for F}
We have
\begin{equation}
\begin{split}
Y^{N,\E}_t=Y^{N,\E}_0+\int_0^t\nabla F(R^{N,\E}_{s-})\cdot dR^{N,\E}_s+\frac{1}{2}dR^{N,\E}_s\cdot H(R^{N,\E}_{s-})dR^{N,\E}_s+\calO^{\FV}_t(\frac{Y^{N,\E}}{N^2})\cap\calO^{\Delta}_t(\frac{1}{N^3}).
\end{split}
\end{equation}
\end{lem}

Combining \eqref{eq:formula for nabla F and Hessian} with Proposition \ref{prop:proposition characterising pk} and Lemma \ref{lem:Ito for F} we obtain \eqref{eq:Y in terms of K and extra terms} by calculation, whereby
\begin{equation}\label{eq:formula for KNE}
\calK^{N,\E}_t:=\int_0^t\frac{1}{Q^N_{s-}}\big(dM^{N,\E}_s-Y^{N,\E}_{s-}dM^{N}_s\big).
\end{equation}
We then obtain \eqref{eq:covariation of KE and KF} from \eqref{eq:covariation for ME MF} and \eqref{eq:formula for KNE}. 

Using the boundedness of $\phi$ and the fact that there are no simultaneous killing events, we obtain \eqref{eq:Y O FV plus U MG} from \eqref{eq:Y in terms of K and extra terms}.

Since in parts \ref{enum:bound on cov of YE YF}-\ref{enum:Y O FV plus U MG} of Theorem \ref{theo:characterisation of Y} the statements of processes belonging to various process sequence classes are uniform over all choices $\E,\F\in\calB(\K)$, Part \ref{enum:Thm 8.2 true for random sets} is immediate.

It remains to prove Proposition \ref{prop:proposition characterising pk}, Part \ref{enum:bound on cov of YE YF} of Theorem \ref{theo:characterisation of Y}, and Lemma \ref{lem:Ito for F}.

\subsubsection{Proof of Proposition \ref{prop:proposition characterising pk}}

Since $N$ is fixed throughout this proof, we neglect the $N$ superscript for the sake of notation, where it would not create confusion. We recall that $\tau^i_n$ represents the $n^{\text{th}}$ killing time of particle $(X^i,\eta^i)$ ($\tau^i_0:=0$), $\tau_n$ is the $n^{\text{th}}$ killing time of any particle ($\tau_0:=0$), and $J^N_t:=\frac{1}{N}\sup\{n:\tau_n\leq t\}$ is the number of killing times up to time $t$, renormalised by $N$. 

We denote 
\[
\phi^{\E}(x,\eta):=\phi(x)\Ind(\eta\in \E),\quad \E\in\calB(\K).
\]
We define for $\E\in\calB(\K)$ the processes
\begin{equation}
\begin{split}
A^{\E}_t=\frac{1}{N}\sum_{i=1}^N\sum_{\tau^i_n\leq t}\phi^{\E}(X^i_{\tau^i_n},\eta^i_{\tau^i_n})-\frac{N}{N-1}\int_0^tP^{\E}_s\langle m^N_s,\kappa\rangle ds+\frac{1}{N(N-1)}\sum_{i=1}^N\sum_{\tau^i_n\leq t}\phi^{\E}(X^i_{\tau^i_n-},\eta^i_{\tau^i_n-}), \\
B^{\E}_t=\langle m^{N,\E}_t,\phi\rangle-\langle m^{N,\E}_0,\phi\rangle
-\frac{1}{N}\sum_{i=1}^N\sum_{\tau^i_n\leq t}\phi^{\E}(X^i_{\tau^i_n},\eta^i_{\tau^i_n})+\lambda\int_0^t\langle m^{N,\E}_s,\phi\rangle ds,\\
\text{and}\quad C^{\E}_t=\frac{1}{N}\sum_{i=1}^N\sum_{\tau^i_n\leq t}\phi^{\E}(X^i_{\tau^i_n-},\eta^i_{\tau^i_n-})-\int_0^t\langle m^{N,\E}_s,\kappa\phi\rangle ds.
\end{split}
\label{eq:defin of Ak Bk Ck}
\end{equation}
We will firstly establish that $A^{\E}_t$, $B^{\E}_t$ and $C^{\E}_t$ are martingales, so that
\begin{equation}
\begin{split}
M^{\E}_t=A^{\E}_t+B^{\E}_t-\frac{1}{N-1}C^{\E}_t
=-\frac{N}{N-1}\int_0^tP^{N,\E}_s\langle m^N_s,\kappa \rangle ds
\\+\langle m^{N,\E}_t,\phi\rangle-\langle m^{N,\E}_0,\phi\rangle
+\lambda\int_0^t\langle m^{N,\E}_s,\phi\rangle ds
+\frac{1}{N-1}\int_0^t\langle m^{N,\E}_s,\kappa\phi\rangle ds
\end{split}
\label{eq:formula for martingale Mk}
\end{equation}
is a martingale. We therefore have \eqref{eq:sde for pk}. We will then establish \eqref{eq:covariation for ME MF} by establishing it for $\E=\F$ and for $\E,\F$ disjoint.

\underline{$A^{\E}_t$ is a martingale}

We have that if particle $X^i$ dies at time $t$, then each $j\neq i$ is selected with probability $\frac{1}{N-1}$, so that the expected value of $\phi^{\E}(X^i_t,\eta^i_t)$ is given by:
\[
\frac{1}{N-1}\sum_{j\neq i}\phi^{\E}(X^j_{t-},\eta^j_{t-})=\frac{N}{N-1}\times \big[P^{\E}_{t-} - \frac{1}{N}\phi^{\E}(X^i_{t-},\eta^i_{t-}) \big].
\]
Therefore summing over $\tau_{n}^i\leq t$, we see that
\[
\begin{split}
\frac{1}{N}\sum_{i=1}^N\sum_{\tau^i_n\leq t}\phi^{\E}(X^i_{\tau^i_n},\eta^i_{\tau^i_n})-\frac{N}{N-1}\int_0^tP^{\E}_{s-}dJ^N_s+\frac{1}{N(N-1)}\sum_{i=1}^N\sum_{\tau^i_n\leq t}\phi^{\E}(X^i_{\tau^i_n-},\eta^i_{\tau^i_n-})
\end{split}
\]
is a martingale. We finally note that $J^N_t-\int_0^t\langle m^N_s,\kappa\rangle ds$ is a martingale, so that $A^{\E}_t$ is a martingale.

\underline{$B^{\E}_t$ is a martingale}

Since $L\phi=-\lambda\phi$, we see that the following is a martingale,
\[
\begin{split}
B_t^{{\E},i,n}:=\Ind(\tau^i_n\leq t<\tau^i_{n+1})\phi^{\E}(X^i_t,\eta^i_{t})-\phi^{\E}(X^i_{\tau^i_n},\eta^i_{\tau^i_n})\Ind(t \geq\tau^i_{n})
\\+\lambda\int_0^t\Ind(\tau^i_n\leq s<\tau^i_{n+1})\phi^{\E}(X^i_{s},\eta^i_{s})ds.
\end{split}
\]
Therefore $\sum_{n<n_0}B_t^{{\E},i,n}$ is a martingale for all $n_0<\infty$. Since $\sum_{n<n_0}\lvert B^{\E,i,n}_t\rvert\leq C(1+t+NJ^N_t)$ for all $n_0<\infty$, for some $C<\infty$, $(\sum_{n<\infty}B_t^{{\E},i,n})_{0\leq t<\infty}$ is a martingale. Therefore
\[
B_t^{\E}=\frac{1}{N}\sum_{i=1}^N\sum_{n<\infty}B_t^{{\E},i,n}\quad\text{is a martingale.}
\]

\underline{$C^{\E}_t$ is a martingale}

We have that
\[
C_t^{{\E},i,n}:=\Ind(t\geq \tau^i_{n+1})\phi^{\E}(X^i_{\tau^i_{n+1}-},\eta^i_{\tau^i_{n+1}-})-\int_0^{t}\Ind(\tau^i_n\leq s<\tau^i_{n+1})\kappa(X^i_s)\phi^{\E}(X^i_s,\eta^i_s)ds
\]
is a martingale. Since for some $C<\infty$, $\sum_{n<n_0}\lvert C^{\E,i,n}_t\rvert\leq C(1+t+NJ^N_t)$ for all $n_0<\infty$, $\sum_{n<\infty}C^{\E,i,n}$ is a martingale. Therefore 
\[
C^{\E}_t=\frac{1}{N}\sum_{i=1}^N\sum_{n<\infty}C_t^{{\E},i,n}\quad\text{is a martingale.}
\]

\underline{The Quadratic Variation of $M^{\E}$}

We observe from \eqref{eq:formula for martingale Mk} that $M^{N,\E}_t-\langle m^{N,\E}_t,\phi\rangle$ is a Lipschitz process. By considering seperately the quadartic variation of the continuous motion between jumps and at the jumps, it follows that
\[
\begin{split}
[M^{N,\E},M^{N,\F}]_t=\frac{1}{N}\int_0^t\langle \Gamma_0(\phi),m^{N,\E\cap \F}_s\rangle ds 
+H^{N,\E,\F}_t,
\end{split}
\]
whereby we define 
\[
H^{N,\E,\F}_t:=\frac{1}{N^2}\sum_{i=1}^N\sum_{\tau_n^i\leq t}[\phi^{\E}(X^i_{\tau_n^i},\eta^i_{\tau_n^i})-\phi^{\E}(X^i_{\tau_n^i-},\eta^i_{\tau_n^i-})][\phi^{\F}(X^i_{\tau_n^i},\eta^i_{\tau_n^i})-\phi^{\F}(X^i_{\tau_n^i-},\eta^i_{\tau_n^i-})].
\]
To characterise $H^{N,\E,\F}_t$ we split into the cases that $\E$ and $\F$ are disjoint, and that $\E=\F$.

\underline{$\E=\F$}

We write $H^{N,\E}_t$ for $H^{N,\E,\E}_t$. At time $\tau^i_n-$, the expected values of $\phi^{\E}(X^i_{\tau^i_n},\eta^i_{\tau^i_n})$ and $\phi^{\E}(X^i_{\tau^i_n},\eta^i_{\tau^i_n})^2$ are
\[
P_{\tau^i_n-}^{\E}+\calO\Big(\frac{P_{\tau^i_n-}^{\E}+\phi^{\E}(X^i_{\tau^i_n-},\eta^i_{\tau^i_n-})}{N}\Big)\quad\text{and}\quad \langle m^{N,\E}_{\tau^i_n-},\phi^2 \rangle +\calO\Big(\frac{P_{\tau^i_n-}^{\E}+\phi^{\E}(X^i_{\tau^i_n-},\eta^i_{\tau^i_n-})}{N}\Big),
\]
respectively. Therefore the expected value of $\big[\phi^{\E}(X^i_{\tau^i_n},\eta^i_{\tau^i_n})-\phi^{\E}(X^i_{\tau^i_n-},\eta^i_{\tau^i_n-})\big]^2$ at time $\tau^i_n-$ is
\[
\begin{split}
\langle m^{N,\E}_{\tau^i_n-},\phi^2\rangle-2P^{\E}_{\tau^i_n-}\phi^{\E}(X^{i}_{\tau^i_n-},\eta^{i}_{\tau^i_n-})+(\phi^{\E}(X^{i}_{\tau^i_n-},\eta^{i}_{\tau^i_n-}))^2 +\calO\Big(\frac{P_{\tau^i_n-}^{\E}+\phi^{\E}(X^i_{\tau^i_n-},\eta^i_{\tau^i_n-})}{N}\Big).
\end{split}
\]
Then using the killing rate to characterise the rate at which killing events happen, we see that
\[
\begin{split}
H^{N,\E}_t-\frac{1}{N}\int_0^t\langle m^N_s,\kappa\rangle \langle m^{N,\E}_s,\phi^2\rangle+\langle m^{N,\E}_s,\kappa \phi^2\rangle-2P^{N,\E}_s\langle m^{N,\E}_s,\kappa\phi\rangle ds\\
=\calO^{\FV}_t\big(\frac{P^{N,\E}}{N^2}\big)\cap\calO^{\Lip}_t+\tilde{M}^{N,\E}_t,
\end{split}
\]
for some martingale $\tilde{M}^{N,\E}_t$. It is straightforward to then see that for all $N$ sufficiently large (which does not depend upon $\E$), $[\tilde{M}^{N,\E}]_t =[H^{N,\E}]_t=\calO^{\FV}_t(\frac{P^{N,\E}}{N^3})$. We therefore obtain \eqref{eq:covariation for ME MF} in the case that $\E=\F$.

\underline{$\E\cap \F=\emptyset$}

Since $\phi^{\E}(x,\eta)\phi^{\F}(x,\eta)=0$ for all $(x,\eta)\in \bar D\times \K$, we have that 
\[
\begin{split}
[\phi^{\E}(X^i_{\tau_n^i},\eta^i_{\tau_n^i})-\phi^{\E}(X^i_{\tau_n^i-},\eta^i_{\tau_n^i-})][\phi^{\F}(X^i_{\tau_n^i},\eta^i_{\tau_n^i})-\phi^{\F}(X^i_{\tau_n^i-},\eta^i_{\tau_n^i-})]\\
=-\phi^{\E}(X^i_{\tau_n^i-},\eta^i_{\tau_n^i-})\phi^{\F}(X^i_{\tau_n^i},\eta^i_{\tau_n^i})-\phi^{\E}(X^i_{\tau_n^i},\eta^i_{\tau_n^i})\phi^{\F}(X^i_{\tau_n^i-},\eta^i_{\tau_n^i-}).
\end{split}
\]
The expected value of this at time $\tau_n^i-$ is then
\[
-\frac{N}{N-1}[\phi^{\E}(X^i_{\tau_n^i-},\eta^i_{\tau_n^i-})P^{N,\F}_{\tau_n^i-}+\phi^{\F}(X^i_{\tau_n^i-},\eta^i_{\tau_n^i-})P^{N,\E}_{\tau_n^i-}].
\]
It follows that
\[
H_t^{N,\E,\F}+\frac{1}{N-1}\int_0^tP^{N,\E}_s\langle m^{N,\F}_s,\kappa \phi\rangle+ P^{N,\F}_s\langle m^{N,\E}_s,\kappa \phi\rangle ds\quad\text{is a martingale,}
\]
which we denote as $\tilde{M}^{N,\E,\F}_t$. It is then straightforward to see that $[\tilde{M}^{N,\E,\F}]_t=[H^{N,\E,\F}]_t=\calO^{\FV}_t(P^{N,\E}P^{N,\F})$. We have therefore obtained \eqref{eq:covariation for ME MF} with $\E\cap \F=\emptyset$.

Having established \eqref{eq:covariation for ME MF} both in the case that $\E=\F$ and the case that $\E\cap \F=\emptyset$, the case of arbitrary $\E,\F$ follows by linearity.

\qed

\subsubsection{Proof of Part \ref{enum:bound on cov of YE YF} of Theorem \ref{theo:characterisation of Y}}\label{subsection:proof of QV of YE YF disjoint part of calculations theorem}
We recall that $F$, $H$ and $\vec{R}$ were defined in \eqref{eq:defin function F}, \eqref{eq:formula for nabla F and Hessian} and \eqref{eq:defin vec R} as
\[
\begin{split}
F(\vec{r})=\frac{p}{q},\quad H(\vec{r})=\begin{pmatrix}
0 & -\frac{1}{q^2}\\
-\frac{1}{q^2} & 2\frac{p}{q^3}
\end{pmatrix}\quad \text{for}\quad\vec{r}=\begin{pmatrix}
p\\
q
\end{pmatrix}\in \Rm_{>0}^2,\\
\vec{R}^{N,\E}_t:=\begin{pmatrix}
P^{N,\E}_t\\
Q^{N}_t
\end{pmatrix}\quad\text{so that}\quad Y^{N,\E}_t=F(\vec{R}^{N,\E}_t).
\end{split}
\]

We decompose
\[
\vec{R}^{N,\E}_t=\vec{R}^{N,\E,C}_t+\vec{R}^{N,\E,J}_t\quad \text{and} \quad Y^{N,\E}_t=F(\vec{R}^{N,\E}_t)=Y^{N,\E,C}_t+Y^{N,\E,J}_t
\]
for $\E\in\calB(\K)$, whereby
\[
Y^{N,\E,J}_t=\sum_{s\leq t}\Delta Y^{N,\E,J}_s\quad\text{and}\quad \vec{R}^{N,\E,J}_t:=\sum_{s\leq t}\Delta\vec{R}^{N,\E}_s=\begin{pmatrix}
P^{\E,J}_t\\
Q^J_t
\end{pmatrix}.
\]
Then by Ito's lemma we have
\begin{equation}\label{eq:Ito for YE cts part}
dY^{N,\E,C}_t=\nabla F(\vec{R}^{N,\E}_t)\cdot d\vec{R}^{N,\E,C}_t+\frac{1}{2}d\vec{R}^{N,\E,C}_t\cdot H(F)(\vec{R}^{N,\E}_t)d\vec{R}^{N,\E,C}_t.
\end{equation}

We can therefore calculate
\begin{equation}\label{eq:cov YEC YFC calc}
\begin{split}
d[Y^{N,\E,C},Y^{N,\F,C}]_t=\frac{1}{(Q^N_t)^2}\big(dP^{N,\E,C}_t-Y^{N,\E}_{t}dQ^{N,C}_t\big)\cdot\big(dP^{N,\F,C}_t-Y^{N,\F}_{t}dQ^{N,C}_t\big).
\end{split}
\end{equation}
Proposition \ref{prop:proposition characterising pk} implies that
\begin{equation}\label{eq:cov PEC PFC}
\begin{split}
d[P^{N,\E,C},P^{N,\F,C}]_t=\calO^{\FV}_t(\frac{Y^{N,\E}Y^{N,\F}}{N}+\frac{Y^{N,\E\cap\F}}{N^2})
\end{split}
\end{equation}
for all $\E,\F\in\calB(\K)$. Combining \eqref{eq:cov YEC YFC calc} with \eqref{eq:cov PEC PFC} we have
\begin{equation}\label{eq:cov YEC YFC}
d[Y^{N,\E,C},Y^{N,\F,C}]_t=\calO^{\FV}_t(\frac{Y^{N,\E}Y^{N,\F}}{N})
\end{equation}
for all $\E,\F\in\calB(\K)$ disjoint. We also have that
\[
[Y^{N,\E,J},Y^{N,\F,J}]_t=\sum_{\tau^i_n\leq t}\Delta Y^{N,\E}_{\tau^i_n}\Delta Y^{N,\F}_{\tau^i_n}.
\]
Since $Q^N_t$ is bounded below away from $0$, by bounding the partial derivatives of $F$ we can calculate for all $\E,\F\in\calB(\K)$ disjoint that
\[
\begin{split}
\lvert\Delta Y^{N,\E}_{\tau^i_n}\Delta Y^{N,\F}_{\tau^i_n}\rvert=\calO\Big(\lvert \Delta P^{N,\E}_{\tau^i_n}\Delta P^{N,\F}_{\tau^i_n}\rvert+\frac{P^{N,\E}_{\tau^i_n-}\lvert \Delta P^{N,\F}_{\tau^i_n}\rvert+P^{N,\F}_{\tau^i_n-}\lvert \Delta P^{N,\E}_{\tau^i_n}\rvert}{N}+\frac{P^{N,\E}_{\tau^i_n-}P^{N,\F}_{\tau^i_n-}}{N^2}\Big)\\
=\calO\Big(\frac{\lvert\phi^{\E}(X^i_{\tau^i_n},\eta^i_{\tau^i_n})\phi^{\F}(X^i_{\tau^i_n-},\eta^i_{\tau^i_n-})\rvert}{N^2}\Big)+\calO\Big(\frac{\lvert\phi^{\F}(X^i_{\tau^i_n},\eta^i_{\tau^i_n})\phi^{\E}(X^i_{\tau^i_n-},\eta_{\tau^i_n-})\rvert}{N^2}\Big)+\calO\Big(\frac{P^{N,\E}_{\tau^i_n-}P^{N,\F}_{\tau^i_n-}}{N^2}\Big)\\+\calO\Big(\frac{\lvert\phi^{\E}(X^i_{\tau^i_n},\eta^i_{\tau^i_n})-\phi^{\E}(X^i_{\tau^i_n-},\eta^i_{\tau^i_n-})\rvert}{N^2}P^{N,\F}_{\tau^i_n-}\Big)
+\calO\Big(\frac{\lvert\phi^{\F}(X^i_{\tau^i_n},\eta^i_{\tau^i_n})-\phi^{\F}(X^i_{\tau^i_n-},\eta^i_{\tau^i_n-})\rvert}{N^2}P^{N,\E}_{\tau^i_n-}\Big).
\end{split}
\]
Since $\kappa$ is bounded, it is straightforward to then see that
\[
\sum_{\tau^i_n\leq t}\lvert\Delta Y^{N,\E}_{\tau^i_n}\Delta Y^{N,\F}_{\tau^i_n}\rvert=\calO^{\FV}_t(\frac{P^{N,\E}P^{N,\F}}{N})=\calO^{\FV}_t(\frac{Y^{N,\E}Y^{N,\F}}{N}),
\]
so that
\begin{equation}\label{eq:cov YEJ YFJ}
[Y^{N,\E,J},Y^{N,\F,J}]_t=\sum_{\tau^i_n\leq t}\lvert\Delta Y^{N,\E}_{\tau^i_n}\Delta Y^{N,\F}_{\tau^i_n}\rvert=\calO^{\FV}_t(\frac{P^{N,\E}P^{N,\F}}{N})=\calO^{\FV}_t(\frac{Y^{N,\E}Y^{N,\F}}{N})
\end{equation}
for all $\E,\F\in\calB(\K)$ disjoint. Combining \eqref{eq:cov YEC YFC} with \eqref{eq:cov YEJ YFJ} we have Part \ref{enum:bound on cov of YE YF} of Theorem \ref{theo:characterisation of Y}.

\qed
\subsubsection{Proof of Lemma \ref{lem:Ito for F}}

We take $0\leq t_0\leq t_1\leq t$ and write
\[
Y^{N,\E,J}_{t_1}-Y^{N,\E,J}_{t_0}=\sum_{t_0<s\leq t_1}\big(Y^{N,\E}_{s}-Y^{N,\E}_{s-}\big).
\]

We may calculate
\begin{equation}\label{eq:3rd derivs for Taylor}
\frac{\partial^3 F}{\partial p^3}=\frac{\partial^3 F}{\partial^2 p\partial q}=0,\quad  \frac{\partial^3 F}{\partial p\partial^2 q}=\frac{2}{q^3},\quad \frac{\partial^3 F}{\partial^3 q}=\frac{-6p}{q^4}.
\end{equation}
Thus by Taylor's theorem, \eqref{eq:3rd derivs for Taylor}, the fact that almost surely there are no simultaneous killing events, and the fact that $Q^{N}_t$ is bounded above and below away from $0$, we have
\[
\begin{split}
\Big\lvert Y^{N,\E,J}_{s}-Y^{N,\E,J}_{s-}-\nabla F(\vec{R}^{N}_{s-})\cdot (\vec{R}^{N,J}_{s}-\vec{R}^{N,J}_{s-})\\-\frac{1}{2}(\vec{R}^{N,J}_{s}-\vec{R}^{N,J}_{s-})\cdot H(F)(\vec{R}^{N}_{s-})( \vec{R}^{N,J}_{s}-\vec{R}^{N,J}_{s-})\Big\rvert\\
=\calO\Big(P^{N,\E}_{s-}\lvert\Delta Q^N_{s-}\rvert^3+\lvert\Delta P^{N,\E}_{s-}\rvert\lvert\Delta Q^{N}_{s-}\rvert^2\Big).
\end{split}
\]
Since $\kappa$ and $\phi$ are bounded, it is straightforward to then see that
\begin{equation}
\begin{split}
Y^{N,\E,J}_t-Y^{N,\E,J}_0-\int_0^t\nabla F(\vec{R}^N_{s-})\cdot d\vec{R}^{N,J}_s-\frac{1}{2}\int_0^td\vec{R}^{N,J}_s\cdot H(F)(\vec{R}^N_{s-})d\vec{R}^{N,J}_s\\=\calO_t^{\FV}\Big(\frac{Y^{N,\E}}{N^3}\Big)\cap\calO^{\Delta}(\frac{1}{N^3}).
\end{split}
\end{equation}
Combining this with \eqref{eq:Ito for YE cts part} we have Lemma \ref{lem:Ito for F}.
\qed

This completes the proof of Theorem \ref{theo:characterisation of Y}.
\qed

\section{Proof of Theorem \ref{theo:Convergence to FV diffusion}}\label{section:proof of convergence to FV diffusion}

With the calculations of Section \ref{section: characterisation of Y} in hand, we now prove Theorem \ref{theo:Convergence to FV diffusion}. We shall make use of the Wasserstein distance $\Wah$ and the weak atomic metric $\Wat$, which are defined in Appendix \ref{appendix:spaces of measures}.

We shall firstly prove the following proposition.
\begin{prop}
\label{prop:conv of dist for given colour}
For all $\E\in\calB(\K)$ and $f\in C_b(\bar D)$ we have that
\begin{equation}\label{eq:measure of mNE}
( m^{N,\E}_t-Y^{N,\E}_t\pi)(f)\ra 0\quad \text{in probability as}\quad t\wedge N\ra\infty.
\end{equation}
In particular, taking $f=1$, we have for any $\E\in\calB(\K)$ that
\begin{equation}\label{eq:number of colours in E approximates Y of E propn}
\chi^N_t(\E)-\calY^{N}_t(\E)\ra 0\quad \text{in probability as}\quad t\wedge N\ra\infty.
\end{equation}
\end{prop}
Heuristically, this says that over an $\calO(1)$ timescale, the number of particles whose colour belongs to $\E$ is given by $Y^{N,\E}_t$, and the spatial distribution of these particles is given by $\pi$, for any $\E\in \calB(\K)$. 

Using Proposition \ref{prop:conv of dist for given colour} and the calculations of Section \ref{section: characterisation of Y}, we will then establish that $(\calY^N_{Nt})_{0\leq t\leq T}$ converges in distribution to the Wright-Fisher process of rate $\Theta$.
\begin{prop}\label{prop:Convergence of tilted empirical measure to FV diffusion}
We take some deterministic initial profile $\nu^0\in\mathcal{P}(\K)$ and define $(\nu_t)_{0\leq t<\infty}$ to be a Wright-Fisher process of rate $\Theta$ and initial condition $\nu_0:=\nu^0$. We then consider a sequence of Fleming-Viot multi-colour Processes $(\vec{X}^N_t,\vec{\eta}^N_t)_{0\leq t<\infty}$. We assume that $\calY^N_0\ra \nu^0$ in $\Wat$ in probability.

We fix $T<\infty$ and rescale time by $t\mapsto Nt$.  We then have the convergence
\begin{equation}\label{eq:conv of measure-valued process to FV in Weak atomic metric}
(\calY^N_{Nt})_{0\leq t\leq T}\ra (\nu_t)_{0\leq t\leq T}\quad\text{in}\quad D([0,T];\mathcal{P}_{\Wah}(\K))\quad\text{in distribution as}\quad N\ra\infty.
\end{equation}
\end{prop}

We recall in particular that $(\nu_t)_{0\leq t\leq T}\in C([0,T];\calP_{\Wah}(\K))$ almost surely, by Theorem \ref{theo:Well-posedness of WF process}. We now take a sequence $(\vec{t}^N)_{2\leq N<\infty}=((t^N_1,\ldots,t^N_n))_{2\leq t\leq N}$ converging to $\vec{t}=(t^1,\ldots,t^n)$ as in the statement of Theorem \ref{theo:Convergence to FV diffusion}. It follows that
\[
(\calY^N_{Nt_1^N},\ldots,\calY^N_{Nt_n^N})\ra (\nu_{t_1},\ldots,\nu_{t_n})\quad\text{in}\quad (\calP_{\Wah}(\K))^n\quad\text{in distribution as}\quad N\ra\infty.
\]
Recalling the positivity and boundedness of $\phi$ from Theorem \ref{theo:convergence to QSD for reflected diffusion with soft killing}, we observe that
\begin{equation}\label{eq:bounding chi by Y}
\chi^N_{t}\leq C\calY^N_{t}\quad\text{for all}\quad t\geq 0,\; N\in\Nm,\quad\text{for some fixed uniform constant}\quad C<\infty.
\end{equation}
We now fix $1\leq k\leq n$. Since $(\calY^N_{Nt^N_k})_{N\geq 1}$ is a tight sequence of random measures, it follows from \eqref{eq:bounding chi by Y} that $(\chi^N_{Nt^N_k})_{N\geq 1}$ must also be a tight sequence of random measures. It therefore follows from \eqref{eq:number of colours in E approximates Y of E propn} and Lemma \ref{lem:tight sequences of measures converging to same limit lemma} that $\Wah(\calY^N_{Nt^N_k},\chi^N_{Nt^N_k})\ra 0$ in probability as $N\ra\infty$. We have therefore established that
\[
(\chi^N_{Nt_1^N},\ldots,\chi^N_{Nt_n^N})\ra (\nu_{t_1},\ldots,\nu_{t_n})\quad\text{in}\quad (\calP_{\Wah}(\K))^n\quad\text{in distribution as}\quad N\ra\infty.
\]
We have left only to strengthen the notion of convergence to convergence in the weak atomic metric.
After proving propositions \ref{prop:conv of dist for given colour} and \ref{prop:Convergence of tilted empirical measure to FV diffusion}, we shall establish the following proposition.
\begin{prop}\label{prop:prop for strengthening to weak atomic metric}
We recall that $\Psi(u):=(1-u)\vee 0$ is the function used to define the $\Wat$ metric in Appendix \ref{appendix:Weak atomic metric}. For all $\delta>0$ there exists $\epsilon>0$ such that
\[
\inf_{N}\Pm\Big(\sup_{0\leq t\leq T}\sum_{\substack{k,\ell\in \K\\ k\neq \ell}}\chi^{N}_{Nt}(\{k\})\chi^{N}_{Nt}(\{\ell\})\Psi\Big(\frac{d(k,\ell)}{\epsilon}\Big)\leq \delta\Big)\geq 1-\delta.
\]
Note that the above sum is well-defined as the terms are non-zero only for $k,\ell\in \text{supp}(\chi^{N}_0)$.
\end{prop}
We may therefore apply the compact containment condition, Lemma \ref{lem:relatively compact family of measures in weak atomic topology}, to conclude that $\{\Law(\chi^{N}_{Nt_k^N})\}$ is tight in $\calP(\calP_{\Wat}(\K))$ for all $1\leq k\leq n$, so that we have Theorem \ref{theo:Convergence to FV diffusion}. 

We have left to prove propositions \ref{prop:conv of dist for given colour}, \ref{prop:Convergence of tilted empirical measure to FV diffusion} and \ref{prop:prop for strengthening to weak atomic metric}

\subsection{Proof of Proposition \ref{prop:conv of dist for given colour}}\label{subsection: proof of conv of dist for given colour}

We fix $\E\in\calB(\K)$ and $f\in C_b(\bar D)$. We write
\[
\psi^N_t:=\frac{1}{N}\sum_{i=1}^N\delta_{(X^i_t,\eta^i_t)}\in \calP(\bar D\times \K),\; 0\leq t<\infty,\; N\in\Nm\quad\text{and}\quad f^{\E}(x,\eta):=f(x)\Ind(\eta\in\E).
\]
We take the $\bar D\times \K$-valued killed strong Markov process $((X_t,\eta_t))_{0\leq t<\tau_{\partial}}$ defined in Definition \ref{defin:limit for multi colour over fixed times}. It follows from Theorem \ref{theo:hydrodynamic limit for multicolour process} that there exists $c_t\ra 0$ as $t\ra \infty$ such that, for all $N<\infty$ and initial conditions $(\vec{X}^N_0,\vec{\eta}^N_0)$ we have that
\begin{equation}\label{eq:convergence in time for killed Markov 1 in main theorem proof}
 \lvert \lvert\Law_{\psi^N_0}((X_t,\eta_t)\lvert \tau_{\partial}>t)-\pi\otimes \calY^N_0 \rvert \rvert_{\TV}\leq c_t,\quad 0\leq t<\infty.
\end{equation}
It follows from Theorem \ref{theo:hydrodynamic limit for multicolour process} that for all $t<\infty$ and $N\geq 2$ there exists $C_{t,N}<\infty$ such that
\begin{equation}\label{eq:hydro convergence 1 in main theorem proof}
\expE_{(\vec{X}^N_0,\vec{\eta}^N_0)}\big[\big\lvert\big(\psi^N_t-\Law_{\psi^N_0}((X_t,\eta_t)\lvert \tau_{\partial}>t)\big)(f^{\E})\big\rvert\big]\leq C_{t,N}\lvert\lvert f\rvert\rvert_{\infty},
\end{equation}
for any initial condition $(\vec{X}^N_0,\vec{\eta}^N_0)$, with $C_{t,N}\ra 0$ as $N\ra\infty$ for fixed $t<\infty$. On the other hand we observe that
\[
\psi^N_t(f^{\E})=m^{N,\E}_t(f),\quad (\pi\otimes \calY^N_t)(f^{\E})=\frac{\sum_{i=1}^N\phi(X^{N,i}_0)\pi\otimes\delta_{\eta^{N,i}_0}}{\sum_{i=1}^N\phi(X^{N,i}_0)}(f^{\E})=Y^{N,\E}_0\pi(f).
\]
Therefore combining \eqref{eq:convergence in time for killed Markov 1 in main theorem proof} with \eqref{eq:hydro convergence 1 in main theorem proof} we obtain that
\[
\expE[\lvert (m^{N,\E}_t-Y_0^{N,\E}\pi)(f)\rvert]\leq (C_{t,N}+c_t )\lvert\lvert f\rvert\rvert_{\infty}.
\]
Proposition \ref{prop:conv of dist for given colour} then follows by applying \eqref{eq:Y O FV plus U MG}.
\qed

\subsection{Proof of Proposition \ref{prop:Convergence of tilted empirical measure to FV diffusion}}\label{subsection:Convergence of tilted empirical measure to FV diffusion}
Our proof proceeds in the following $2$ steps:
\begin{enumerate}
\item \label{enum:conv on disj meas sets to WF diff}
We fix $\epsilon>0$ and take $\{k_1,k_2,\ldots\}$ to be a dense subset of $\K$. Then for all $i$ we can find $\frac{\epsilon}{2}<r_i<\epsilon$ such that $\nu^0(\partial B(k_i,r_i))=0$. We set $A_i=B(k_i,r_i)\setminus(\cup_{j=1}^{i-1}A_j)$. Since the disjoint union of $A_i$ is $\K$, we can find $n<\infty$ such that $\nu^0((\cup_{i=1}^n A_i)^c)<\epsilon$. We set $A_0:=(\cup_{i=1}^n A_i)^c$ and pick arbitrary $k_0\in \K$. 

We shall prove that $(Y^{N,A_0}_{Nt},\ldots,Y^{N,A_n}_{Nt})_{0\leq t\leq T}$ converges in $D([0,T];\Rm^{n+1})$ in distribution to a Wright-Fisher diffusion of rate $\Theta$ and initial condition $(\nu^0(A_0),\ldots,\nu^0(A_n))$.
\item \label{enum:conv to WF super in Wass}
We then use this to prove that
\begin{equation}\label{eq:conv of measure-valued type process in Wasserstein}
(\calY^N_{Nt})_{0\leq t\leq T}\ra (\nu_t)_{0\leq t\leq T}\quad\text{in}\quad D([0,T];\mathcal{P}_{\Wah}(\K))\quad\text{in distribution.}
\end{equation}
\end{enumerate}

\subsubsection*{Step \ref{enum:conv on disj meas sets to WF diff}}

We recall that the martingale $\calK^{N,\E}_t$ was defined in Theorem \ref{theo:characterisation of Y}, whilst $\Lambda^{N,\E}_t$ and $\Lambda^N_t$ were defined in \eqref{eq:lambdaNE definition} to be given by
\[
\begin{split}
\Lambda^{N,\E}_t:=\langle m^{N,\E}_t,\Gamma_{0}(\phi)+\kappa\phi^2\rangle +\langle m^{N,\E}_t,\phi^2\rangle \langle m^{N}_t,\kappa\rangle\;\; \text{for}\;\; \E\in\calB(\K),\;\;\text{and}\;\; \Lambda^N_t:=\Lambda^{N,\K}_t.
\end{split}
\]
We further define
\[
(\vec{Y}^N_{Nt})_{0\leq t\leq T}:=((Y^{N,A_0}_{Nt},\ldots,Y^{N,A_n}_{Nt}))_{0\leq t\leq T}.
\]
We will now verify that $\{\Law((\vec{Y}^N_{Nt})_{0\leq t\leq T})\}$ is tight in $\calP(D([0,T];\Rm^{n+1}))$ by using Aldous' criterion \cite[Theorem 1]{Aldous1978}. Since $0\leq Y^{N,A_i}_{Nt}\leq 1$, \cite[Condition (3)]{Aldous1978} is satisfied. 

We now take a sequence $(\tau_N,\delta_N)_{N=1}^{\infty}$ of stopping times $\tau_N$ and constants $\delta_n>0$ satisfying \cite[Condition (1)]{Aldous1978}, for the purpose of checking \cite[Condition (A)]{Aldous1978}. In particular we have by \eqref{eq:Y O FV plus U MG} that for some $F^N=\calO^{\FV}(1)$ and $M^N=\calO^{\MG}(1)$ we have
\[
Y^{N,A_i}_{N(\tau_N+\delta_N)}-Y^{N,A_i}_{N\tau_N}=F^N_{\tau_N+\delta_N}-F^N_{\tau_N}+Z^N_{\tau_N+\delta_N}-Z_N\ra  0\quad\text{in probability.}
\]
Thus $\{(\vec{Y}^N_{Nt})_{0\leq t\leq T}\}$ satisfies \cite[Condition (A)]{Aldous1978},
\[
\vec{Y}^N_{N(\tau_N+\delta_N)}-\vec{Y}^N_{N\tau_N}\ra  0\quad\text{in probability,}
\]
and hence $\{\Law((\vec{Y}^N_{Nt})_{0\leq t\leq T})\}$ is tight in $\calP(D([0,T];\Rm^{n+1}))$ by \cite[Theorem 1]{Aldous1978}. For all $\E\in\calB(\K)$, \eqref{eq:measure of mNE} implies that
\begin{equation}\label{eq:convergence of ANE}
\begin{split}
\Lambda^{N,\E}_t-Y^{N,\E}_t[\langle \pi,\Gamma_{0}(\phi)+\kappa\phi^2\rangle+\langle \pi,\phi^2\rangle \langle \pi,\kappa\rangle],\; Q^N_t-\langle \pi,\phi\rangle \overset{p}{\ra} 0\quad \text{as}\quad t\wedge N\ra\infty.
\end{split}
\end{equation}
Then applying \eqref{eq:measure of mNE} and Fubini's theorem to \eqref{eq:Y in terms of K and extra terms}, we obtain
\begin{equation}\label{eq:K minus Y converges to 0}
\begin{split}
\sup_{0\leq t\leq T}\lvert (Y^{N,\E}_{Nt}-Y^{N,\E}_{0})-(\calK^{N,\E}_{Nt}-\calK^{N,\E}_{0})\rvert \ra 0\quad\text{in probability as}\quad N\ra\infty.
\end{split}
\end{equation}

We consider a subsequential limit in distribution of $\{(\vec{Y}^N_{Nt})_{0\leq t\leq T}\}$,
\[
(\vec{Y}_t)_{0\leq t\leq T}=((Y^{A_0}_t,\ldots,Y^{A_n}_t))_{0\leq t\leq T},
\]
which by Part \ref{enum:Y O FV plus U MG} of Theorem \ref{theo:characterisation of Y} must have continuous paths. Using \eqref{eq:K minus Y converges to 0} we conclude that $(\calK^{N,A_0}_{Nt},\ldots,\calK^{N,A_n}_{Nt})_{0\leq t\leq T}$ converges in $D([0,T];\Rm^{n+1})$ in distribution along this subsequence to 
\[
(\vec{Y}_t-\vec{Y}_0)_{0\leq t\leq T}. 
\]
Since $(\calK^{N,A_0}_{Nt},\ldots,\calK^{N,A_n}_{Nt})_{0\leq t\leq T}$ is a martingale for each $N$, $(Y^{A_0}_{t},\ldots,Y^{A_n}_{t})_{0\leq t\leq T}$ is a martingale with respect to its natural filtration $\sigma_t$. We then obtain from \eqref{eq:covariation of KE and KF} that for all $0\leq i,j\leq n$,
\[
\begin{split}
\calK^{N,A_i}_{Nt}\calK^{N,A_j}_{Nt}
-\int_0^t\frac{1}{(Q^N_s)^2}\Big[
\Ind(i=j)\Lambda^{N,A_i}_{Ns}
-Y^{N,A_i}_{Ns} \Lambda^{N,A_j}_{Ns}
-Y^{N,A_j}_{Ns} \Lambda^{N,A_i}_{Ns} +Y^{N,A_i}_{Ns}Y^{N,A_j}_{Ns} \Lambda^{N}_{Ns} \Big]ds
\\-\calO_t^{\MG}\big(\frac{Y^{N,A_i}_{N\cdot}Y^{N,A_j}_{N\cdot}+\Ind(i=j)Y^{N,A_i}_{N\cdot}}{N^2}\big)-\calO_t^{\FV}\Big(\frac{Y^{N,A_i}_{N\cdot}Y^{N,A_j}_{N\cdot}+\Ind(i=j)Y^{N,A_i\cap A_j}_{N\cdot}}{N}\Big)\cap\calO^{\Cts}_t
\end{split}
\]
is a martingale for all $N$, so that by \eqref{eq:convergence of ANE} and \eqref{eq:K minus Y converges to 0},
\[
\begin{split}
Y^{A_i}_{t}Y^{A_j}_{t}-\int_0^t\frac{\langle \pi,\Gamma_{0}(\phi)+\kappa\phi^2\rangle+\langle \pi,\phi^2\rangle \langle \pi,\kappa\rangle}{\langle \pi,\phi\rangle^2}\big(
\Ind(i=j)Y^{A_i}_{s}
-Y^{A_i}_{s} Y^{A_j}_{s}
 \big) ds
\end{split}
\]
is a $(\sigma_t)_{t\geq 0}$-martingale. Thus
\[
\begin{split}
[Y^{A_i},Y^{A_j}]_{t}=\int_0^t\frac{\langle \pi,\Gamma_{0}(\phi)+\kappa\phi^2\rangle+\langle \pi,\phi^2\rangle \langle \pi,\kappa\rangle}{\langle \pi,\phi\rangle^2}\big(
\Ind(i=j)Y^{A_i}_{s}
-Y^{A_i}_{s} Y^{A_j}_{s}
 \big) ds.
\end{split}
\]
We have that
\begin{equation}\label{eq:calculation for diffusivity equal to theta}
\begin{split}
\langle \pi,\Gamma_{0}(\phi)+\kappa\phi^2\rangle=\langle \pi,L(\phi^2)-2\phi L(\phi)\rangle=\lambda \langle \pi,\phi^2\rangle\quad \text{and}\quad 
\langle \pi,\kappa\rangle=\langle \pi,-L(1)\rangle=\lambda.
\end{split}
\end{equation}
Since $\nu^0(\partial A_i)=0$ for all $0\leq i\leq n$,
\[
\vec{Y}^N_0\ra (\nu^0(A_0),\ldots,\nu^0(A_n))\quad\text{in probability.}
\]
Thus each subsequential limit $(\vec{Y}_t)_{0\leq t\leq T}$ must be a solution of the $n+1$-type Wright-Fisher diffusion of rate $\Theta$ with initial condition $(\nu^0(A_0),\ldots,\nu^0(A_n))$, which is unique in law. Therefore we have convergence of the whole sequence in $D([0,T];\Rm^{n+1})$ in distribution to this Wright-Fisher diffusion.

\subsubsection*{Step \ref{enum:conv to WF super in Wass}}

Whereas we use $\Wah$ to denote the Wasserstein metric on $\calP(\K)$ generated by $d\wedge 1$, we metrise $\calP(D([0,T];\calP_{\Wah}(\K))))$ using the Wasserstein-$1$ metric generated by the metric $d_{D([0,T];\calP_{\Wah}(\K))}\wedge 1$, which we denote as $\overline{\Wah}$. We take $\epsilon_{\ell}\ra 0$, giving $k^{\ell}_0,k^{\ell}_1,\ldots,k^{\ell}_{n_{\ell}}\in\K$ for each $\ell\in \mathbb{N}$ as provided for in Step \ref{enum:conv on disj meas sets to WF diff}. We define for each $\ell\in\mathbb{N}$ the projection
\[
{\bf{P}}^{\ell}:\calP(\bar D)\ni \mu\mapsto \sum_{j=0}^{n_{\ell}}\mu(A_{k^{\ell}_j})\delta_{k^{\ell}_j}\in \calP(\bar D).
\]
We write $\nu^{\ell}_t:={\bf{P}}^{\ell}(\nu_t)$ for all $0\leq t<\infty$. It is immediate that
\[
\Wah({\bf{P}}^{\ell}(\mu),\mu)\leq \epsilon+\mu(A^{\ell}_0)\quad\text{for all}\quad\mu\in\calP(\bar D).
\]
Proposition \ref{prop:basic facts WF superprocess} implies that we can write
\[
\nu^{\ell}_t=\sum_{j=0}^{n_{\ell}}p^{\ell,j}_t\delta_{k^{\ell}_{n_{\ell}}},\quad 0\leq t<\infty,
\]
whereby $(p^{\ell,0}_t,\ldots,p^{\ell,n_{\ell}}_t)_{0\leq t<\infty}$ is an $n_{\ell}$-type Wright-Fisher diffusion of rate $\Theta$ and initial condition $(\nu^0(A_{k^{\ell}_0}),\ldots,\nu^0(A_{k^{\ell}_{n_{\ell}}}))$. Step \ref{enum:conv on disj meas sets to WF diff} therefore implies that
\[
\bar \Wah(\Law(({\bf{P}}^{\ell}(\calY^{N}_{Nt}))_{0\leq t\leq T}),\Law((\nu^{\ell}_t)_{0\leq t\leq T}))\ra 0\quad \text{as}\quad N\ra\infty.
\]

Therefore by the triangle inequality we have
\begin{equation}\label{eq:triangle inequality for Wass Law of YN and mu}
\begin{split}
\limsup_{N\ra\infty}\bar\Wah(\Law((\calY^{N}_{Nt})_{0\leq t\leq T}),\Law((\nu_t)_{0\leq t\leq T}))\leq
\limsup_{N\ra\infty}\bar\Wah(\Law((\calY^{N}_{Nt})_{0\leq t\leq T}),\Law(({\bf{P}}^{\ell}(\calY^{N}_{Nt}))_{0\leq t\leq T}))\\
+\bar\Wah(\Law((\nu^{\ell}_t)_{0\leq t\leq T}),\Law((\nu_t)_{0\leq t\leq T}))\leq 2\epsilon_{\ell} +\limsup_{N\ra\infty}\expE[\sup_{0\leq t\leq T}Y^{N,A_0^{\ell}}_{Nt}]+\expE[\sup_{0\leq t\leq T}\nu_{t}(A^{\ell}_0)].
\end{split}
\end{equation}

Using again Step \ref{enum:conv on disj meas sets to WF diff} we have that
\[
\limsup_{N\ra\infty}\expE[\sup_{0\leq t\leq T}Y^{N,A_0^{\ell}}_{Nt}]=\expE[\sup_{0\leq t\leq T}\nu_{t}(A^{\ell}_0)].
\]
Since $(\nu_t(A_0))_{0\leq t\leq T}$ is a Wright-Fisher diffusion of rate $\Theta$ and initial condition $\nu^0(A^{\ell}_0)<\epsilon_l$,
\[
\expE[\sup_{0\leq t\leq T}\nu_{t}(A^{\ell}_0)]\ra 0\quad\text{as}\quad \ell\ra\infty.
\]
Therefore taking $\limsup_{\ell\ra\infty}$ of both sides of \eqref{eq:triangle inequality for Wass Law of YN and mu} we obtain \eqref{eq:conv of measure-valued type process in Wasserstein}.

\subsection{Proof of Proposition \ref{prop:prop for strengthening to weak atomic metric}}\label{subsection:proof of prop for strengthening to weak atomic metric}

It follows from \eqref{eq:bounding chi by Y} that it suffices to verify the following condition.
\begin{cond}
For every $\delta>0$, there exists $\epsilon>0$ such that
\begin{equation}\label{eq:sum of prod of nearby Ys}
\limsup_{N\ra\infty}\Pm\Big(\sup_{0\leq t\leq T}\sum_{\substack{k,\ell\in \K\\ k\neq \ell}}Y^{N,\{k\}}_{Nt}Y^{N,\{\ell\}}_{Nt}\Psi\Big(\frac{d(k,\ell)}{\epsilon}\Big)\leq \delta\Big)\geq 1-\delta.
\end{equation}
\label{cond:condition equivalent to weak atomic compact containment}
\end{cond}
We calculate using parts \ref{enum:bound on cov of YE YF}, \ref{enum:Y O FV plus U MG} and \ref{enum:Thm 8.2 true for random sets} of Theorem \ref{theo:characterisation of Y} that
\[
\begin{split}
d(Y^{N,\{k\}}_{t}Y^{N,\{\ell\}}_{t})=Y^{N,\{k\}}_{t-}dY^{N,\{\ell\}}_{t}+Y^{N,\{\ell\}}_{t-}dY^{N,\{k\}}_{t}+d[Y^{N,\{k\}},Y^{N,\{\ell\}}]_t\\
=Y^{N,\{k\}}_{t-}\Big[d\calO^{\FV}_t\Big(\frac{Y^{N,\{\ell\}}}{N}\Big)+d\calO^{\MG}_t\Big(\frac{Y^{N,\{\ell\}}}{N}\Big)\Big]\\
+Y^{N,\{\ell\}}_{t-}\Big[d\calO^{\FV}_t\Big(\frac{Y^{N,\{k\}}}{N}\Big)+d\calO^{\MG}_t\Big(\frac{Y^{N,\{k\}}}{N}\Big)\Big]+d\calO^{\FV}_t\Big(\frac{Y^{N,\{k\}}Y^{N,\{\ell\}}}{N}\Big)\\
=d\calO^{\FV}_t\Big(\frac{Y^{N,\{k\}}Y^{N,\{\ell\}}}{N}\Big)+d\calO^{\MG}_t\Big(\frac{Y^{N,\{k\}}Y^{N,\{\ell\}}}{N}\Big),
\end{split}
\]
uniformly over all random $k,\ell\in\text{supp}(\calY^{N}_0)$. Thus 
\[
\sum_{k,\ell\in\mathbb{K}} Y^{N,\{k\}}_{Nt}Y^{N,\{\ell\}}_{Nt}=\calO^{\FV}_t\Big(\sum_{k,\ell\in\mathbb{K}} Y^{N,\{k\}}_{N\cdot}Y^{N,\{\ell\}}_{N\cdot}\Big)+(\text{martingale})_t.
\]
Therefore, using Gronwall's inequality, there exists uniform $C<\infty$ such that
\begin{equation}\label{eq:supermartingale for weak atomic metric proof}
e^{-Ct}\sum_{\substack{k,\ell\in \K\\ k\neq \ell}}Y^{N,\{k\}}_{Nt}Y^{N,\{\ell\}}_{Nt}\Psi\Big(\frac{d(k,\ell)}{\epsilon}\Big)
\end{equation}
is a supermartingale for all $N$ large enough. Therefore we have for all $N$ large enough that
\[
\begin{split}
\Pm\Big(\sup_{0\leq t\leq T}\sum_{\substack{k,\ell\in \K\\ k\neq \ell}}Y^{N,\{k\}}_{Nt}Y^{N,\{\ell\}}_{Nt}\Psi\Big(\frac{d(k,\ell)}{\epsilon}\Big)\leq \delta\Big)\\
\geq \Pm\Big(\sup_{0\leq t\leq T}e^{-Ct}\sum_{\substack{k,\ell\in \K\\ k\neq \ell}}Y^{N,\{k\}}_{Nt}Y^{N,\{\ell\}}_{Nt}\Psi\big(\frac{d(k,\ell)}{\epsilon}\big)\leq e^{-CT}\delta\Big)\\
\geq 1-\frac{1}{e^{-CT}\delta}\expE\Big[\sum_{\substack{k,\ell\in \K\\ k\neq \ell}}Y^{N,\{k\}}_{0}Y^{N,\{\ell\}}_{0}\Psi\Big(\frac{d(k,\ell)}{\epsilon}\Big)\Big].
\end{split}
\]

We have assumed that the initial conditions $\calY^N_0$ converge in the weak atomic metric, so Lemma \ref{lem:relatively compact family of measures in weak atomic topology} implies that
\[
\sup_N\expE\Big[\sum_{\substack{k,\ell\in \K\\ k\neq \ell}}Y^{N,\{k\}}_{0}Y^{N,\{\ell\}}_{0}\Psi\Big(\frac{d(k,\ell)}{\epsilon}\Big)\Big]\ra 0\quad\text{as}\quad \epsilon\ra 0.
\]

We have therefore verified Condition \ref{cond:condition equivalent to weak atomic compact containment} and hence established Proposition \ref{prop:prop for strengthening to weak atomic metric}.
\qed

This concludes the proof of Theorem \ref{theo:Convergence to FV diffusion}.
\qed

\section{Extension to the hard killing case}\label{section:hard killing extension}

The Fleming-Viot particle system was first introduced by Burdzy, Hołyst and March \cite{Burdzy2000} in the case of Brownian dynamics with instantaneous killing at the boundary (\textit{hard killing}). In this section, we extend Theorem \ref{theo:Convergence to FV diffusion} to this setting.

We assume throughout this section that $D$ is a bounded, connected, non-empty, open subset of $\Rm^d$ with $C^{\infty}$ boundary. The process $(B_t)_{0\leq t<\tau}$ evolves as a Brownian motion in $D$, killed at the time $\tau_{\partial}:=\inf\{t>0:B_{t-}\in \partial D\}$. The colour space $\K$ remains an arbitrary complete, seperable metric space.

The Fleming-Viot particle system and Fleming-Viot multi-colour process are then defined as before, except that the particles evolve as independent Brownian motions between killing times, and are killed instantaneously upon contact with the boundary $\partial D$ (i.e. when $B^{N,i}_{t-}\in \partial D$), rather than according to a Poisson clock. In particular, we define the following. 
\begin{defin}[Fleming-Viot multi-colour process with hard killing]
The Fleming-Viot multi-colour process, $((\vec{B}^N_t,\vec{\eta}^N_t))_{0\leq t<\infty} =\{(B^{N,i}_t,\eta^{N,i}_t)_{0\leq t<\infty}:i=1,\ldots,N\}$, is a $(D\times \K)^N$-valued process defined as follows:
\begin{equation}  
\left \{ \begin{split}
(i) & \quad \text{Initial condition: $((B^{N,1}_0,\eta^{N,1}_0),\ldots,(B^{N,N}_0,\eta^{N,N}_0))\sim \upsilon^N\in \mathcal{P}(( D\times\K)^N)$.}\\
(ii) & \quad \text{For $t \in [0,\infty)$ and between killing times the particles $(B^{N,i}_t,\eta^{N,i}_t)$ evolve as Brownian}\\
& \quad\text{motions in the first variable, and are constant in the second variable.}\\
(iii) & \quad \text{The particle $(B_t^{N,i},\eta_t^{N,i})$ is killed instantaneously whenever the first variable makes}\\
& \quad \text{contact with the boundary, i.e. when $B^{N,i}_t\in \partial D$. We write $\tau^i_k$ for the death times}\\
& \quad \text{of particle $(B^{N,i},\eta^{N,i})$ (with $\tau^i_0:=0$). When particle $(B^{N,i},\eta^{N,i})$ is killed at time $\tau^i_k$,}\\
& \quad\text{it jumps to the location of particle $(B^{N,j},\eta^{N,j})$, with $j=U^i_k\in\{1,\ldots,N\}\setminus\{i\}$ chosen}\\
& \quad\text{independently and uniformly at random, at which time we set}\\
& \quad \text{$(B^{N,i}_{\tau^i_k},\eta^{N,i}_{\tau^i_k}):=(B^{N,j}_{\tau^i_k-},\eta^{N,j}_{\tau^i_k-})$. Moreover we write $\tau_n$ for the $n^{\text{th}}$ time at which any particle}\\
& \quad \text{ is killed (with $\tau_0:=0$).}
\end{split}  \right.
\label{eq:hard killing N-particle m label (X,eta) system}
\end{equation}
We further define as before
\begin{equation}\label{eq:hard killing spatial and colour empirical measures}
J^N_t:=\frac{1}{N}\sup\{n>0:\tau_n\leq t\},\quad m^N_t:=\frac{1}{N}\sum_{i=1}^N\delta_{B^{N,i}_t}\quad\text{and}\quad  \chi^{N}_t:=\frac{1}{N}\sum_{i=1}^N\delta_{\eta^{N,i}_t}.
\end{equation}
\label{defin:hard killing Multi-Colour Process}
\end{defin}

It is an open problem to establish the well-posedness of this Fleming-Viot particle system without imposing constraints upon the boundary regularity of $\partial D$ \cite{Burdzya}. The issue is the possibility of there being infinitely many jumps in finite time. Implicit in the proof of \cite[Theorem 1.4]{Burdzy2000} is a proof of the well-posedness of the particle system when the domain satisfies an interior ball condition. Another proof under this condition is due to L\"obus \cite{Lobus2009}. A proof when the domain is Lipschitz with Lipschitz constant less than a given value (dependent upon the dimension and number of particles) is given in \cite{Bieniek2009}. In particular, the assumption that $D$ is bounded and the boundary $\partial D$ is $C^{\infty}$ certainly suffices to ensure the Fleming-Viot particle system (and therefore also the Fleming-Viot multi-colour process) is well-posed.

Brownian motion with hard killing, $(B_t)_{0\leq t<\tau_{\partial}}$, defines the $C_0$-Feller semigroup
\[
P_t:C_0(D)\ni f\mapsto (x\mapsto P_tf(x):=\expE_x[f(X_t)\Ind(\tau_{\partial}>t)])\in C_0(D).
\]
We write $L$ for its infitesimal generator, which is just the half Dirichlet Laplacian. We write $\phi\in C_0(D;\Rm_{>0})\cap C^{\infty}(D)$ for the unique principal right eigenfunction of $L$, of eigenvalue $-\lambda<0$. In general, quasi-stationary distributions correspond to left eigenmeasures of the infinitesimal generator \cite[Proposition 4]{Meleard2011}, which in this case corresponds to the normalised right eigenfunction $\phi$. Therefore the unique QSD of $(X_t)_{0\leq t<\tau_{\partial}}$, denoted as $\pi$, is given by
\[
\pi(dx)=\frac{\phi(x)dx}{\int_D\phi(x')dx'}.
\]

As in the soft killing case, the rate of the limiting Wright-Fisher process is given by
\begin{equation}\label{eq:hard killing const theta}
\Theta:=\frac{2\lambda \lvert\lvert \phi\rvert\rvert_{L^2(\pi)}^2}{\lvert\lvert \phi\rvert\rvert_{L^1(\pi)}^2}.
\end{equation}
We again define the tilted empirical measure of the colours by
\begin{equation}\label{eq:hard killing tilted empirical measure}
\begin{split}
\calY^N_t:=\frac{\frac{1}{N}\sum_{i=1}^N\phi(B^{i}_t)\delta_{\eta^{i}_t}}{\frac{1}{N}\sum_{i=1}^N\phi(B^{i}_t)}\in\calP(\K).
\end{split}
\end{equation}

We prove the following analogue of Theorem \ref{theo:Convergence to FV diffusion}.

\begin{theo}\label{theo:hard killing Convergence to FV diffusion}
We take some deterministic initial profile $\nu^0\in\mathcal{P}(\K)$ and fix a Wright-Fisher process on $\calP(\mathbb{K})$ of rate $\Theta$ and initial condition $\nu_0=\nu^0$, which we denote as $(\nu_t)_{0\leq t<\infty}$. We consider a sequence of Fleming-Viot multi-colour Processes, denoted by $(((\vec{B}^N_t,\vec{\eta}^N_t))_{0\leq t<\infty}:2\leq N<\infty)$, such that
\begin{equation}
\calP(\K)\ni\calY^N_0\ra \nu^0\in \calP(\K)\quad \text{in $\Wat$ in probability as}\quad N\ra\infty.
\end{equation}
We further require the following condition,
\begin{equation}\label{eq:initial condition doesn't accumulate mass on the boundary}
\limsup_{N\ra\infty}\expE[\frac{1}{N}\#\{i\in \{1,\ldots,N\} :d(B^{N,i}_0,\partial D)<\delta\})]\ra 0\quad\text{as} \quad \delta\ra 0.
\end{equation}

We now rescale time according to $t\mapsto Nt$. Then $(\chi^{N}_{Nt})_{t> 0}$ converges to $(\nu_t)_{t> 0}$ in finite-dimensional distributions, in the following sense. We fix arbitrary $n<\infty$ and $\vec{t}=(t^1,\ldots,t^n)\in [0,\infty)^n$ such that $t^1\leq \ldots\leq t^n$. We consider arbitrary sequences $(\vec{t}^N)_{2\leq N<\infty}:=((t^N_1,\ldots,t^N_n))_{2\leq N\leq \infty}$ such that:
\begin{enumerate}
\item
$t^N_1\leq \ldots\leq t^N_n$ for all $2\leq N<\infty$;
\item
$t^N_i\ra t_i$ as $N\ra \infty$ for all $1\leq i\leq n$;
\item\label{enum:hard killing main theorem requirement that times are large}
$Nt^N_n\geq\ldots\geq Nt^N_1\ra\infty$ as $N\ra \infty$.
\end{enumerate}
We then have that
\begin{equation}\label{eq:hard killing main theorem convergence of empirical measures}
(\chi^{N}_{Nt_1^N},\ldots,\chi^{N}_{Nt_n^N}) \ra (\nu_{t_1},\ldots,\nu_{t_n})\quad\text{in}\quad (\calP_{\Wat}(\mathbb{K}))^n\quad \text{in distribution as}\quad N\ra\infty.
\end{equation}
\end{theo}

We observe that the only difference with Theorem \ref{theo:Convergence to FV diffusion} is the condition \eqref{eq:initial condition doesn't accumulate mass on the boundary}, which is necessitated by the fact that the domain is no longer compact. Indeed when we considered reflected diffusions with soft killing, the domain $\bar D$ was compact, with the principal eigenfunction $\phi$ being bounded away from $0$. However in the case of hard killing, the domain $D$ is non-compact, with $\phi$ vanishing at the boundary. As a consequence of this, we must establish controls on the mass near the boundary. In order to obtain a hydrodynamic limit theorem over a fixed time horizon for the Fleming-Viot particle system with hard killing, it is important to obtain such controls over a fixed time horizon, as have been established in \cite{Burdzy2000,Villemonais2011,Tough2022}. Since Theorem \ref{theo:Convergence to FV diffusion} is a statement about the Fleming-Viot particle system over an $\calO(N)$ time horizon, however, we require controls on the mass near the boundary over an $\calO(N)$ time horizon. Such controls have not previously been established, and represent the principle obstacle to extending Theorem \ref{theo:Convergence to FV diffusion} to include hard killing. We will obtain such controls in Subsection \ref{subsection:additional control for hard killing} in the case of Brownian dynamics with hard killing, allowing us to prove Theorem \ref{theo:hard killing Convergence to FV diffusion}.

The rest of this Section is devoted to the proof of Theorem \ref{theo:hard killing Convergence to FV diffusion}. We will outline in Subsection \ref{subsection:hard killing notation} the notation we will use for the proof of Theorem \ref{theo:hard killing Convergence to FV diffusion}, and in particular where it differs from the notation outlined in Section \ref{section:notation for the proof} for the soft killing case. We will then obtain controls on the mass near the boundary $\partial D$ in Subsection \ref{subsection:additional control for hard killing}. The rest of the proof follows the same outline as the proof of Theorem \ref{theo:Convergence to FV diffusion}. In Subsection \ref{section:hard killing characterisation of Y} we will then perform calculations for $\calY^N_t$ which are analogous to those of Section \ref{section: characterisation of Y}. Finally, we conclude the proof of Theorem \ref{theo:hard killing Convergence to FV diffusion} in Subsection \ref{subsection:hard killing theorem proof}, analogously to Section \ref{section:proof of convergence to FV diffusion}. We collect the proofs of technical lemmas needed for the proof of Theorem \ref{theo:hard killing Convergence to FV diffusion} in Appendix \ref{appendix:hard killing}. We will not repeat calculations which are identical to those found in the proof of Theorem \ref{theo:Convergence to FV diffusion}, pointing out only where the proof differs.

\subsection{Notation for the proof of Theorem \ref{theo:hard killing Convergence to FV diffusion}}\label{subsection:hard killing notation}

Recalling that $L$ is the (half) Dirichlet Laplacian (there is no analogue of $L_0$ here), we define on the domain $\calD(\Gamma):=\{f:f,f^2\in \calD(L)\}\subseteq C_0(D)$ the Carre du champs operator 
\begin{equation}\label{eq:hard killing Carre d champs}
\Gamma(f):=L(f^2)-2fL(f)=\lvert \nabla f\rvert^2.
\end{equation}
We note in particular that $\nabla \phi$ is globally bounded by \cite[Theorem 1.1]{Xu2009}, so that $\phi\in \calD(\Gamma)$.

In the soft killing case, the definitions of $P_t^{N,\E}$, $Q_t^N$, $Y_t^{N,\E}$ and $m_t^{N,\E}$ for $\E\in\calB(\K)$ are given in Section \ref{section:notation for the proof}. We adopt the same definitions here, except that $\phi$ is now the principal eigenfunction of the half Dirichlet Laplacian. Moreover we define $F$, $H$ and $\vec{R}^N_t$ as in \eqref{eq:defin function F}, \eqref{eq:formula for nabla F and Hessian} and \eqref{eq:defin vec R} respectively, so that
\[
\begin{split}
F(\vec{r})=\frac{p}{q},\quad H(\vec{r})=\begin{pmatrix}
0 & -\frac{1}{q^2}\\
-\frac{1}{q^2} & 2\frac{p}{q^3}
\end{pmatrix}\quad \text{for}\quad\vec{r}=\begin{pmatrix}
p\\
q
\end{pmatrix}\in \Rm_{>0}^2,\quad
\vec{R}^{N,\E}_t:=\begin{pmatrix}
P^{N,\E}_t\\
Q^{N}_t
\end{pmatrix}.
\end{split}
\]

We adopt the notation described in Subsection \ref{section:O and U notation}, so that $V_t$, $\calO_t^{\MG}$, $\calO_t^{\FV}$, $\calO^{\Delta}_t$, $\calO^{\Lip}_t$ and $\calO^{\Cts}_t$ are defined as in Subsection \ref{section:O and U notation} in particular. We further define the following:
\begin{enumerate}
\item
$X^N_t=\calO^{\MG}_t(Y^N,J^N)$ if for all $N\geq N_0$ (for some $N_0<\infty$) and for some $C<\infty$, $X^N_t$ is a martingale whose quadratic variation is such that
\begin{equation}
[X^N]_t-\int_0^tCY^N_sdJ^N_s\quad\text{is a supermartingale.}
\end{equation}
\item
$X^N_t=\calO^{\FV}_t(Y^N,J^N)$ if for all $N\geq N_0$ (for some $N_0<\infty$) and for some $C<\infty$, $X^N_t$ is a finite variation process whose total variation, $V_t(X^N)$, is such that
\begin{equation}
V_t(X^N)-\int_0^tCY^N_sdJ^N_s\quad\text{is a supermartingale.}
\end{equation}
\end{enumerate}

\subsection{Control on the mass near the boundary}
\label{subsection:additional control for hard killing}

\begin{lem}\label{lem:mass near boundary controls}
We fix $T<\infty$. We consider a sequence of Fleming-Viot particle systems $(\vec{X}^N_t)_{t\geq 0}$ satisfying \eqref{eq:initial condition doesn't accumulate mass on the boundary}. Then for all $\epsilon>0$ there exists $\delta=\delta(\epsilon)>0$ such that
\begin{equation}\label{eq:mass near boundary control 1}
\liminf_{N\ra\infty}\Pm_{\vec{X}^N_0}(m^N_t(B(\partial D,\delta))\leq \epsilon\quad\text{for all} \quad 0\leq t\leq NT)>1-\epsilon.
\end{equation}
\end{lem}
We define for all $\epsilon>0$ the stopping time
\begin{equation}\label{eq:stopping time tauNepsilon}
\tau^N_{\epsilon}:=\inf\{t>0:m^N_t(B(\partial D,\delta(\epsilon)))> 2\epsilon\}.
\end{equation}
We observe in particular that $Q^N_t$ is uniformly bounded from below by a strictly positive constant dependent only upon $\epsilon$, for all $0\leq t\leq\tau^N_{\epsilon}$. Moreover it follows from \eqref{eq:mass near boundary control 1} that
\begin{equation}\label{eq:hard killing stopping time larger than NT with large prob}
\liminf_{N\ra\infty}\Pm(\tau^N_{\epsilon}>NT)\geq 1-\epsilon.
\end{equation}

\begin{proof}[Proof of Lemma \ref{lem:mass near boundary controls}]
Controls on the mass of particles near the boundary over a fixed time horizon were established by Burdzy, Hołyst and March in \cite{Burdzy2000}. These involve a coupling between the particles and a family of Bessel processes. Similar couplings were established by Villemonais in \cite[Section 3]{Villemonais2011}, and by Nolen and the present author in \cite[Section 4]{Tough2022}. We utilise the coupling obtained in \cite[Section 4]{Tough2022}, since this coupled family of processes is jointly independent, allowing us to apply Cramer's theorem.

Since $\partial D$ is $C^{\infty}$ and $D$ is bounded, $D$ satisfies an interior ball condition for some radius $r>0$. For this $r>0$, \cite[Proposition 4.2]{Tough2022} provides a family of independent $[0,r]$-valued continuous processes $(\eta^{N,1}_t)_{t\geq 0},\ldots, (\eta^{N,N}_t)_{t\geq 0}$, with $(\eta^{N,i}_t)_{t\geq 0}$ having the same distribution for all $i$ and all $N$, such that
\begin{equation}\label{eq:comparison of distance to the boundary and Bessel}
d(B^{N,i}_t,\partial D)\geq r-\eta^{N,i}_t\quad\text{for}\quad 0\leq t<\infty.
\end{equation}

We define $M_N:=\lfloor \frac{N}{k}\rfloor$. We then define for $1\leq m\leq M_N$, $1\leq k\leq N<\infty$, $0<T_0<T_1<\infty$ and $\delta_0>0$ the events
\begin{equation}
A_N^{k,T_0,T_1,\delta_0}(m):=\{\sup_{T_0\leq t\leq T_1}\lvert \{(m-1)k+1\leq i\leq mk:\eta^{N,i}_t\geq r-\delta\} \rvert\geq 2\}.
\end{equation}
We observe that for fixed $k,T_0,T_1,\delta$, the events $A_N^{k,T_0,T_1,\delta}(m)$ for $ m\leq M_N$ and $N\geq k$ have the same probability. Moreover it follows from \cite[Lemma 4.4]{Tough2022} that, for fixed $k,T_0,T_1$,
\begin{equation}
\Pm(A^{k,T_0,T_1,\delta}_N(1))\ra 0 \quad \text{as}\quad \delta \ra 0.
\end{equation}

Moreover it follows from \eqref{eq:comparison of distance to the boundary and Bessel} that
\begin{equation}
\sup_{T_0\leq t\leq T_1}m^N_t(B(\partial D,\delta))\leq \frac{k}{N}\sum_{1\leq m\leq M_N}\Ind(A_N^{k,T_0,T_1,\delta}(m))+\frac{1}{k}+\frac{k}{N}.
\end{equation}

We may therefore apply Cramer's theorem to conclude that for all $\epsilon>0$ and $0<T_0\leq T_1<\infty$, there exists $\delta_0>0$ and $c_0>0$ such that
\begin{equation}
\Pm_{\vec{X}^N_0}(\sup_{T_0\leq t\leq T_1}m^N_t(B(\partial D,\delta))\geq \epsilon)\leq e^{-c_0 N}\quad\text{for all $N$ large enough}.
\end{equation}

It follows from a union bound that, for all $\epsilon,T_0>0$, there exists $\delta_0>0$ such that 
\begin{equation}\label{eq:control from Cramer}
\liminf_{N\ra\infty}\Pm_{\vec{X}^N_0}[m^N_t(B(\partial D,\delta_0)]\leq \epsilon\quad\text{for all} \quad t_0\leq t\leq NT)=1.
\end{equation}

All that remains is to deal with the initial time $[0,T_0]$, which can be addressed with a crude bound. We observe that, for any $T_0,\delta_1>0$, in order for a given particle to enter $B(\partial D,\delta_1)$, it either has to start within $B(\partial D,2\delta_1)$, or else travel at least a distance $\delta_1$ in time $T_0$ (note that killing only occurs at the boundary). The former possibility can be controlled by \eqref{eq:initial condition doesn't accumulate mass on the boundary}, the latter by controlling the distance travelled by Brownian motion in time $T_0$. We obtain that for all $\epsilon>0$ there exists $T_0>0$ and $\delta_1>0$ such that
\begin{equation}\label{eq:initial time crude bound}
\liminf_{N\ra\infty}\Pm(m^N_t(B(\partial D,2\delta_1))\leq \epsilon\quad\text{for all} \quad 0\leq t\leq T_0)> 1-\epsilon.
\end{equation}

Therefore, for given $\epsilon>0$, we choose $\delta_1,T_0>0$ for which we have \eqref{eq:initial time crude bound}. For this same $T_0,\epsilon>0$, we then obtain $\delta_0>0$ such that we have \eqref{eq:control from Cramer}. Taking $\delta:=\delta_0\wedge \delta_1$, we obtain \eqref{eq:mass near boundary control 1}.
\end{proof}

\subsection{Analogue of the calculations of Section \ref{section: characterisation of Y}}\label{section:hard killing characterisation of Y}

We obtain the following analogue of Theorem \ref{theo:characterisation of Y}.

\begin{theo}\label{theo:hard killing characterisation of Y}
We fix arbitrary $\epsilon>0$, and localise up to the stopping time $\tau^N_{\epsilon}$ defined in \eqref{eq:stopping time tauNepsilon}. None of the following statements should be understood to be uniform in $\epsilon$, but rather should be understood as statements for arbitrary fixed $\epsilon>0$. We have the following, uniformly over all choices of $\calE,\calF\in\calB(\K)$:
\begin{enumerate}
\item\label{enum:hard killing bound on cov of YE YF}
The covariation $[Y^{N,\E},Y^{N,\F}]_{t\wedge \tau^N_{\epsilon}}$ is such that for disjoint $\E, \F\in\calB(\K)$ we have
\[
[Y^{N,\E},Y^{N,\F}]_{t\wedge \tau^N_{\epsilon}}=\calO^{\FV}_t(\frac{Y^{N,\E}Y^{N,\F}}{N},J^N).
\]
\item\label{enum:hard killing Y in terms of K and extra terms}
There exists martingales $\calK^{N,\E}_t$ for $\E\in\calB(\K)$ such that $Y^{N,\E}_t$ satisfies
\begin{equation}\label{eq:hard killing Y in terms of K and extra terms}
\begin{split}
Y_{t\wedge \tau^N_{\epsilon}}^{N,\E}=Y_0^{N,\E}+\calK^{N,\E}_{t\wedge \tau^N_{\epsilon}}+\calO^{\MG}_t(\frac{Y^{N,\E}}{N^3},J^N)
+\frac{1}{N}\int_0^{t\wedge \tau^N_{\epsilon}}\frac{1}{(Q_{s-}^N)^2}\langle Y_{s-}^{N,\E}m_{s-}^N-m_{s-}^{N,\E},\Gamma (\phi)\rangle ds\\+\calO^{\FV}_{t}(\frac{Y^{N,\E}}{N^2},J^N)\cap\calO^{\Delta}_t(\frac{1}{N^3})
+\frac{1}{N-1}\int_0^{t\wedge \tau^N_{\epsilon}}\frac{1}{(Q_{s-}^N)^2}\langle Y_{s-}^{N,\E}m_{s-}^N-m_{s-}^{N,\E},\phi^2\rangle dJ^N_{s}
\end{split}
\end{equation}
for $\E\in\calB(\K)$, and such that
\begin{equation}\label{eq:hard killing covariation of KE and KF}
\begin{split}
[\calK^{N,\E},\calK^{N,\F}]_{t\wedge \tau^N_{\epsilon}}\\=\frac{1}{N-1}\int_0^{t\wedge \tau^N_{\epsilon}}\frac{1}{(Q_{s-}^N)^2}\langle m_{s-}^{N,\E\cap F}-Y^{N,\E}_{s-}m_{s-}^{N,\F}
-Y^{N,\E}_{s-}m_{s-}^{N,\F}+Y^{N,\E}_{s-}Y^{N,\F}_{s-}m^N_s,\phi^2\rangle dJ^N_s\\
+\frac{1}{N}\int_0^{t\wedge \tau^N_{\epsilon}}\frac{1}{(Q_{s-}^N)^2}\langle m_{s-}^{N,\E\cap F}-Y^{N,\E}_{s-}m_{s-}^{N,\F}-Y^{N,\E}_{s-}m_{s-}^{N,\F}+Y^{N,\E}_{s-}Y^{N,\F}_{s-}m^N_s,\Gamma (\phi)\rangle ds.
\end{split}
\end{equation}
\item\label{enum:hard killing Y O FV plus U MG}
Furthermore $Y^{N,\E}_t$ satisfies
\begin{equation}\label{eq:hard killing Y O FV plus U MG}
\begin{split}
Y^{N,\E}_{t\wedge \tau^N_{\epsilon}}=\Big[\calO^{\FV}_t\big(\frac{Y^{N,\E}}{N}\big)+\calO^{\FV}_t\big(\frac{Y^{N,\E}}{N},J^N\big)+\calO^{\MG}_t\big(\frac{Y^{N,\E}}{N}\big)+\calO^{\MG}_t\big(\frac{Y^{N,\E}}{N},J^N\big)\Big]\cap\calO^{\Delta}_t(\frac{1}{N}).
\end{split}
\end{equation}
\item
Parts \ref{enum:bound on cov of YE YF}- \ref{enum:Y O FV plus U MG} remain true if $\E$ and $\F$ are replaced with a sequence of $\sigma_0$-measurable random sets $\calE^N$ and $\calF^N$.\label{enum:hard killing Thm 8.2 true for random sets}
\end{enumerate}
\end{theo}

\begin{proof}[Proof of Theorem \ref{theo:hard killing characterisation of Y}]

It is useful (and simplifies our calculations) to note that since $\phi$ vanishes on the boundary and killing only occurs on the boundary, we necessarily have that $\phi(B^i_{t-})=0$ if $B^i$ is killed at time $t$.

It is straightforward to obtain the following analogue of Proposition \ref{prop:proposition characterising pk}, by examining the martingale
\[
P_t^{N,\E}-P_0^{N,\E}-\int_0^tP^{N,\E}_{s-}\big(-\lambda ds+\frac{N}{N-1}dJ^{N}_s\big).
\]
\begin{prop}\label{prop:hard killing proposition characterising pk}
We have for all $\E\in\calB(\K)$ that
\begin{equation}
dP^{N,\E}_t=P^{N,\E}_{t-}\big(-\lambda dt+\frac{N}{N-1}dJ^{N}_t\big)+dM^{N,\E}_t,
\label{eq:hard killing sde for pk}
\end{equation}
whereby $M^{N,\E}$ are martingales which satisfy for all $\E,\F\in\calB(\K)$
\begin{equation}
\begin{split}
[M^{N,\E},M^{N,\F}]_t=\frac{1}{N}\int_0^t\langle m_s^{N,\E\cap \F},\Gamma_0(\phi)\rangle ds\\+\frac{1}{N}\int_0^t\frac{N}{N-1}\langle m_{s-}^{N,\E\cap \F},\phi^2\rangle-\Big(\frac{N}{N-1}\langle m_{s-}^{N,\E},\phi\rangle\Big)\Big(\frac{N}{N-1}\langle m_{s-}^{N,\F},\phi\rangle\Big)dJ^N_s.
\end{split}
\label{eq:hard killing covariation for ME MF}
\end{equation}
Moreover it is apparent that $\sum_{s\leq t}\Delta M^{N,\E}_t= \calO^{\MG}_t(\frac{P^{N,\E}}{N},J^N_t)$ for $\calE\in \calB(\K)$. We write $M^N_t$ for $M^{N,\K}_t$. 
\end{prop}

Using that $Q^N_t$ is bounded from below away from $0$ for $t\leq \tau^N_{\epsilon}$, uniformly in $N$, we obtain Part \ref{enum:hard killing bound on cov of YE YF} of Theorem \ref{theo:hard killing characterisation of Y} in precisely the same manner that we obtained Part \ref{enum:bound on cov of YE YF} of Theorem \ref{theo:characterisation of Y} in Subsection \ref{subsection:proof of QV of YE YF disjoint part of calculations theorem}.

Replacing the boundedness of the jump rate with a bound in terms of the number of jumps, we obtain the following analogue of Lemma \ref{lem:Ito for F}.
\begin{lem}[Ito's Lemma]\label{lem:hard killing Ito for F}
We have
\begin{equation}
\begin{split}
Y^{N,\E}_t=Y^{N,\E}_0+\int_0^t\nabla F(R^{N,\E}_{s-})\cdot dR^{N,\E}_s+\frac{1}{2}dR^{N,\E}_s\cdot H(R^{N,\E}_{s-})dR^{N,\E}_s+\calO^{\FV}_{t}(\frac{Y^{N,\E}}{N^2},J^N)\cap\calO^{\Delta}_t(\frac{1}{N^3}).
\end{split}
\end{equation}
\end{lem}
We then obtain parts \ref{enum:hard killing Y in terms of K and extra terms}-\ref{enum:hard killing Thm 8.2 true for random sets} of Theorem \ref{theo:hard killing characterisation of Y} as in the proof of Theorem \ref{theo:characterisation of Y}.
\end{proof}

\subsection{Proof of Theorem \ref{theo:hard killing Convergence to FV diffusion}}\label{subsection:hard killing theorem proof}

We firstly prove the following analogue of proposition \ref{prop:conv of dist for given colour}.
\begin{prop}
\label{prop:hard killing conv of dist for given colour}
For all $\E\in\calB(\K)$ and $f\in C_b(\bar D)$ we have that
\begin{equation}\label{eq:hard killing measure of mNE}
\Ind(\tau^N_{\epsilon}>t)\sup_{t\leq s_1,s_2\leq t+1}\lvert ( m^{N,\E}_{s_1}-Y^{N,\E}_{s_2}\pi)(f)\rvert\ra 0\quad \text{in probability as}\quad t\wedge N\ra\infty.
\end{equation}
In particular, taking $f=1$, we have for any $\E\in\calB(\K)$ that
\begin{equation}\label{eq:hard killing number of colours in E approximates Y of E propn}
\Ind(\tau^N_{\epsilon}>t)\sup_{t\leq s_1,s_2\leq t+1}\lvert \chi^N_{s_1}(\E)-\calY^{N}_{s_2}(\E))\rvert\ra 0\quad \text{in probability as}\quad t\wedge N\ra\infty.
\end{equation}
\end{prop}

We note that $\tau^N_{\epsilon}>0$ implies a uniform positive lower bound on $\Pm_{m^N_0}(\tau_{\partial}>T)$, where we think of the empirical measure $m^N_0$ as the initial condition of a single killed Brownian motion, killed at time $\tau_{\partial}$. Using this fact, and using Theorem \ref{theo:hard killing characterisation of Y}, Theorem \ref{theo:hard killing hydrodynamic limit for multicolour process} and Proposition \ref{prop:hard killing convergence in time for killed Markov 1 in main theorem proof} in place of Theorem \ref{theo:hard killing characterisation of Y}, Therem \ref{theo:hydrodynamic limit for multicolour process} and Proposition \ref{prop:convergence in time for killed Markov 1 in main theorem proof} respectively, the proof of Proposition \ref{prop:hard killing conv of dist for given colour} is identical to the proof of Proposition \ref{prop:conv of dist for given colour} found in Subsection \ref{subsection: proof of conv of dist for given colour}.

The characterisation of $\calY^N_t$ in the setting of soft killing given in Section \ref{section: characterisation of Y} does not involve $dJ^N_t$ terms. This is a consequence of the fact that the jumps occur at a (position dependent) Poisson rate. On the other hand, as a consequence of the hard catalyst killing, the characterisation of $\calY^N_t$ given in Subsection \ref{section:hard killing characterisation of Y} does involve such terms. Consequentially we shall require the following lemma.
\begin{lem}\label{lem:convergence of jump process}
We consider a sequence of uniformly bounded processes $(Z^N_t)_{t\geq 0}$ such that \linebreak$\sup_{t\leq s\leq t+1}\lvert Z_s^N-Z_t^N\rvert\ra 0$ in probability as $t\wedge N\ra \infty$. Then, after rescaling time by $t\mapsto Nt$, we have that 
\begin{equation}
\int_0^{T}Z^N_{Ns}(\frac{1}{N}dJ_{Ns}^N-\lambda ds)\ra 0\quad\text{in probability as}\quad N\ra \infty.
\end{equation} 
In particular we have the convergence
\begin{equation}
\big(\frac{1}{N}J^N_{Nt}\big)_{0\leq t\leq T}\ra (\lambda t)_{0\leq t\leq T}\quad\text{uniformly in probability as}\quad N\ra\infty.
\end{equation}
\end{lem}

\begin{proof}[Proof of Lemma \ref{lem:convergence of jump process}]
We fix $\epsilon>0$. We firstly take $\E=\K$ in \eqref{eq:hard killing measure of mNE} and apply Proposition \ref{prop:convergence in W in prob equivalent to weakly prob} to see that $m_t^N\Ind(\tau_{\epsilon}^N>t)+\pi\Ind(\tau_{\epsilon}^N\leq t)$ converges to $\pi$ in $\Wah$ in probability. We then obtain from Theorem \ref{theo:hard killing hydrodynamic limit for multicolour process} that \begin{equation}\label{eq:jumps over time 1 stopping time is approx lambda}
\Ind(\tau^N_{\epsilon}>t)(J^N_{t+1}-J^N_t-\lambda)\ra 0\quad\text{in probability as}\quad N\wedge t\ra \infty.
\end{equation}

It follows from the proof of \cite[Proposition 4.10]{Tough2022} that there exists $M<\infty$ and $p\in (0,1)$, dependent upon neither $N$ nor $t$, such that the number of jumps of any particle between the times $t\wedge \tau_{\epsilon}^N$ and $(t+1)\wedge \tau_{\epsilon}^N$ is stochastically dominated by the sum of $M$ independent geometric random variables. It follows that $\{J^N_{(t+1)\wedge \tau_{\epsilon}^N}-J^N_{t\wedge \tau_{\epsilon}^N}:t\geq 0,N\geq 2\}$ is uniformly bounded in $L^2(\Pm)$, hence uniformly integrable.

We now calculate
\[
\begin{split}
\int_0^T\int_{Nt}^{Nt+1}Z^N_sdJ^N_sdt=
\int_0^T\int_{0}^{NT+1}\Ind(Nt\leq s\leq Nt+1)Z^N_sdJ^N_sdt\\
=\int_{0}^{NT+1}Z^N_s\int_{\frac{s-1}{N}\vee 0}^{\frac{s}{N}\wedge T}dt dJ^N_s=\int_{0}^{NT+1}\min(\frac{1}{N},\frac{s}{N},\frac{T-s-1}{N}) Z^N_s  dJ^N_s.
\end{split}
\]
We see that
\[
\begin{split}
\Ind(\tau_{\epsilon}^N>NT)\Big\lvert \frac{1}{N}\int_0^TZ^N_{Ns}dJ^N_{Ns}-\int_0^T\int_{Nt}^{Nt+1}Z^N_sdJ^N_sdt\Big\rvert\\ \leq \frac{C}{N}\Ind(\tau_{\epsilon}^N>NT)(J_1-J_0+J_{NT+1}-J_{NT})\overset{p}{\ra}
 0\quad\text{as}\quad N\ra\infty.
\end{split}
\]
It follows from the above that
\[
\begin{split}
\expE\Big[\Big\lvert\int_0^T\Ind(\tau_{\epsilon}^N>NT)\Big(\int_{Nt}^{Nt+1}Z^N_s(dJ^N_s-\lambda ds)\Big) dt\Big\rvert\Big]\\
\leq\int_0^T\expE\Big[\Big\lvert\Ind(\tau_{\epsilon}^N>NT)\int_{Nt}^{Nt+1}Z^N_s(dJ^N_s-\lambda ds)\Big\rvert\Big] dt\ra 0\quad\text{as}\quad N\ra\infty.
\end{split}
\]
Since $\epsilon>0$ is arbitrary, Lemma \ref{lem:convergence of jump process} follows from \eqref{eq:hard killing stopping time larger than NT with large prob}.
\end{proof}

We then obtain the following analogue of Proposition \ref{prop:Convergence of tilted empirical measure to FV diffusion}. 

\begin{prop}\label{prop:hard killing Convergence of tilted empirical measure to FV diffusion}
We take some deterministic initial profile $\nu^0\in\mathcal{P}(\K)$ and define $(\nu_t)_{0\leq t<\infty}$ to be a Wright-Fisher process of rate $\Theta$ and initial condition $\nu_0:=\nu^0$. We then consider a sequence of Fleming-Viot multi-colour Processes $(\vec{B}^N_t,\vec{\eta}^N_t)_{0\leq t<\infty}$. We assume that $\calY^N_0\ra \nu^0$ in $\Wat$ in probability.

We fix $T<\infty$ and rescale time by $t\mapsto Nt$. We then have the convergence
\begin{equation}\label{eq:hard killing conv of measure-valued process to FV in Weak atomic metric}
(\calY^N_{Nt})_{0\leq t\leq T}\ra (\nu_t)_{0\leq t\leq T}\quad\text{in}\quad D([0,T];\mathcal{P}_{\Wah}(\K))\quad\text{in distribution as}\quad N\ra\infty.
\end{equation}
\end{prop}

\begin{proof}[Proof of Proposition \ref{prop:hard killing Convergence of tilted empirical measure to FV diffusion}]

The proof of Proposition \ref{prop:hard killing Convergence of tilted empirical measure to FV diffusion} follows in the same two steps as that of the proof of Proposition \ref{prop:Convergence of tilted empirical measure to FV diffusion}. In both steps, we fix $\epsilon>0$ and localise up to time $\tau^N_{\epsilon}$. We then repeat the proof found in Subsection \ref{subsection:Convergence of tilted empirical measure to FV diffusion}, with Theorem \ref{theo:characterisation of Y} and Proposition \ref{prop:conv of dist for given colour} replaced by Theorem \ref{theo:hard killing characterisation of Y} and Proposition \ref{prop:hard killing conv of dist for given colour} respectively, and using Lemma \ref{lem:convergence of jump process} in the obvious manner. We then conclude each step by observing that $\epsilon>0$ is arbitrary, and applying \eqref{eq:hard killing stopping time larger than NT with large prob}.
\end{proof}

We continue as in the proof of Theorem \ref{theo:Convergence to FV diffusion}. We recall from Theorem \ref{theo:Well-posedness of WF process} that $(\nu_t)_{0\leq t\leq T}\in C([0,T];\calP_{\Wah}(\K))$ almost surely. We take a sequence $(\vec{t}^N)_{2\leq N<\infty}=((t^N_1,\ldots,t^N_n))_{2\leq t\leq N}$ converging to $\vec{t}=(t^1,\ldots,t^n)$ as in the statement of Theorem \ref{theo:hard killing Convergence to FV diffusion}. It then follows that
\[
(\calY^N_{Nt_1^N},\ldots,\calY^N_{Nt_n^N})\ra (\nu_{t_1},\ldots,\nu_{t_n})\quad\text{in}\quad (\calP_{\Wah}(\K))^n\quad\text{in distribution as}\quad N\ra\infty.
\]
Recalling the positivity and boundedness of $\phi$ from Theorem \ref{theo:hard killing convergence to QSD for reflected diffusion with soft killing}, and the definition \eqref{eq:stopping time tauNepsilon} of $\tau^N_{\epsilon}$, we observe that for every $\epsilon>0$ there exists a uniform constant $C_{\epsilon}<\infty$ and random measures $\delta_t^{N,\epsilon}$ for $0\leq t<\tau^N_{\epsilon}$ and $N\in \Nm$ such that
\begin{equation}\label{eq:hard killing bounding chi by Y}
\chi^N_{t}\leq C_{\epsilon}\calY^N_{t}+\delta_t^{N,\epsilon}\quad\text{and}\quad \delta_t^{N,\epsilon}(\K)\leq \epsilon\quad\text{for all}\quad 0\leq t<\tau^N_{\epsilon},\; N\geq 2.
\end{equation}
Note that the term $\delta_t^{N,\epsilon}$ does not appear in the soft killing case, \eqref{eq:bounding chi by Y}. It arises here from the fact that $\phi$ is no longer bounded from below, but instead vanishes at the boundary.

We now fix $1\leq k\leq n$. Since $(\calY^N_{Nt^N_k})_{N\geq 1}$ is a tight sequence of random measures, it follows from \eqref{eq:hard killing bounding chi by Y} that $(\chi^N_{Nt^N_k})_{N\geq 1}$ must also be a tight sequence of random measures. It therefore follows from \eqref{eq:hard killing number of colours in E approximates Y of E propn} and Lemma \ref{lem:tight sequences of measures converging to same limit lemma} that $\Wah(\calY^N_{Nt^N_k},\chi^N_{Nt^N_k})\ra 0$ as $N\ra\infty$. Thus we have established that
\[
(\chi^N_{Nt_1^N},\ldots,\chi^N_{Nt_n^N})\ra (\nu_{t_1},\ldots,\nu_{t_n})\quad\text{in}\quad (\calP_{\Wah}(\K))^n\quad\text{in distribution as}\quad N\ra\infty.
\]
We have left only to strengthen the notion of convergence to convergence in the weak atomic metric. This is accomplished with the following analogue of Proposition \ref{prop:prop for strengthening to weak atomic metric}.
\begin{prop}\label{prop:hard killing prop for strengthening to weak atomic metric}
We recall that $\Psi(u):=(1-u)\vee 0$ is the function used to define the $\Wat$ metric in Appendix \ref{appendix:Weak atomic metric}. We fix $\epsilon>0$, thereby defining the stopping defined $\tau_{\epsilon}^N$ by \eqref{eq:stopping time tauNepsilon}. For all $\delta>0$ there exists $\epsilon'>0$ such that
\[
\liminf_{N\ra\infty}\Pm\Big(\tau_{\epsilon}^N>NT,\sup_{0\leq t\leq T}\sum_{\substack{k,\ell\in \K\\ k\neq \ell}}\chi^{N}_{Nt}(\{k\})\chi^{N}_{Nt}(\{\ell\})\Psi\Big(\frac{d(k,\ell)}{\epsilon'}\Big)\leq \delta+\epsilon\Big)\geq 1-\delta.
\]
Note that the above sum is well-defined as the terms are non-zero only for $k,\ell\in \text{supp}(\chi^{N}_0)$.
\end{prop}

\begin{proof}[Proof of Proposition \ref{prop:hard killing prop for strengthening to weak atomic metric}]
We follow the same proof strategy as the proof of Proposition \ref{prop:prop for strengthening to weak atomic metric}, replacing Theorem \ref{theo:characterisation of Y} with Theorem \ref{theo:hard killing characterisation of Y}, applying Lemma \ref{lem:convergence of jump process} in the obvious manner, and replacing the supermartingale \eqref{eq:supermartingale for weak atomic metric proof} with
\[
e^{-C\big(t+\frac{J^N_{Nt\wedge \tau_{\epsilon}^N}}{N}\big)}\sum_{\substack{k,\ell\in \K\\ k\neq \ell}}Y^{N,\{k\}}_{Nt\wedge \tau_{\epsilon}^N}Y^{N,\{\ell\}}_{Nt\wedge \tau_{\epsilon}^N}\Psi\Big(\frac{d(k,\ell)}{\epsilon}\Big),
\]
for some sufficiently large constant $C<\infty$.
\end{proof}
Having established Proposition \ref{prop:hard killing prop for strengthening to weak atomic metric}, we may then apply Lemma \ref{lem:relatively compact family of measures in weak atomic topology} along with \eqref{eq:hard killing stopping time larger than NT with large prob} to conclude that $\{\Law(\chi^{N}_{Nt_k^N})\}$ is tight in $\calP(\calP_{\Wat}(\K))$ for all $1\leq k\leq n$, so that we have Theorem \ref{theo:Convergence to FV diffusion}. 
\qed

\begin{appendix}
\section{Reflected diffusions with soft killing}\label{appendix:reflected diffusions}

\subsection{Definition}

We consider a normally reflected diffusion $(X^0_t)_{0\leq t<\infty}$ in the domain $\bar D$ corresponding to a solution of the Skorokhod problem. In particular, for any filtered probability space on which is defined the $m$-dimensional Brownian motion $W_t$ and initial condition $x\in\bar D$, there exists by \cite[Theorem 3.1]{Lions1984a} a pathwise unique strong solution of the Skorokhod problem
\begin{equation}\label{eq:Skorokhod problem}
\begin{split}
X^0_t=x+\int_0^tb(X^0_s)ds+\int_0^t\sigma(X^0_s)dW_s+\int_0^t\vec{n}(X^0_s)d\xi_s\in \bar D,\quad 0\leq t<\infty,\quad \int_0^{\infty}\Ind_D(X^0_t)d\xi_t=0,
\end{split}
\end{equation} 
where $W_s$ is a Brownian motion and the local time $\xi_t$ is a non-decreasing process with $\xi_0=0$. 

This corresponds to a solution of the submartingale problem introduced by Stroock and Varadhan \cite{Stroock1971}, and is a Feller process \cite[Theorem 5.8, Remark 2]{Stroock1971} (and hence strong Markov). It is then straightforward (using a separate probability space on which is defined an exponential random variable) to construct an enlarged filtered probability space on which $(X^0,W,\xi)$ is a solution of the Skorokhod problem and on which there is a stopping time $\tau_{\partial}$ corresponding to the ringing time of a Poisson clock with position dependent rate $\kappa(X^0_t)$, from which is constructed the killed process $(X_t)_{0\leq t<\tau_{\partial}}$. This killed process is a solution to
\begin{equation}\label{eq:killed process SDE}
\begin{split}
\Ind(t\geq \tau_{\partial})-\int_0^{t\wedge \tau_{\partial}}\kappa(X^0_s)ds\quad\text{is a martingale},\quad
X_t:=\begin{cases}
X^0_t,\quad t<\tau_{\partial}\\
0,\quad t\geq \tau_{\partial}
\end{cases},
\end{split}
\end{equation} 
where $W_s$ is an $m$-dimensional Brownian motion and the local time $\xi_t$ is a non-decreasing process with $\xi_0=0$. Since $X^0_t$ is Feller, the process $X_t$ is therefore also Feller (and hence strong Markov). 

We write $L^0/L=L^0-\kappa$ for the infinitesimal generators of $X^0$ and $X$ respectively, having the same domains $\mathcal{D}(L^0)=\mathcal{D}(L)$. We further define the Carre du Champs operator $\Gamma_0$ on the algebra $\calA$,
\begin{equation}\label{eq:Carre du champs}
\begin{split}
\Gamma_0(f,g):=L_0(fg)-fL_0(f)-gL_0(f),\quad \Gamma_0(f):=\Gamma_0(f,f),\\
f,g\in \calA:=\{f\in C^2(\bar D):\vec{n}\cdot \nabla f\equiv 0\quad\text{on}\quad \partial D\},
\end{split}
\end{equation}
so that for $f\in \calA$ we have
\begin{equation}\label{eq:quadratic variation from Carre du champs appendix}
[f(X^0)]_t=\int_0^t\Gamma_0(f)(X^0_s)ds.
\end{equation}

\subsection{Convergence to a unique quasi-stationary distribution}\label{appendix:convergence to a QSD}

\begin{theo}\label{theo:convergence to QSD for reflected diffusion with soft killing}
There exists a unique quasi-stationary distribution (QSD) $\pi\in \mathcal{P}(\bar D)$ for $X_t$. Moreover there exist constants $C<\infty$ and $k>0$ such that
\begin{equation}\label{eq:exponential convergence to QSD reflected diffusion}
\lvert\lvert \mathcal{L}_{\mu}(X_t\lvert \tau_{\partial}>t)-\pi\rvert\rvert_{\TV}\leq Ce^{-kt}\quad\text{for all}\quad\mu\in\mathcal{P}(\bar D)\quad\text{and}\quad t\geq 0.
\end{equation}
Furthermore $\pi$ is a left eigenmeasure of $L$ with eigenvalue $-\lambda<0$,
\begin{equation}\label{eq:pi left eigenmeasure}
\langle \pi, Lf\rangle=-\lambda \langle \pi,f\rangle,\quad f\in\calD(L),
\end{equation}
and with corresponding positive right eigenfunction $\phi\in \calA\cap C^{2}(\bar D; \Rm_{>0})$. This right eigenfunction is both the unique non-negative right eigenfunction and the unique right eigenfunction of eigenvalue $-\lambda$, up to rescaling. 
\end{theo}

\begin{proof}[Proof of Theorem \ref{theo:convergence to QSD for reflected diffusion with soft killing}]
Our strategy is to check \cite[Assumption (A)]{Champagnat2014}, starting with \cite[Assumption (A1)]{Champagnat2014}. 

We fix arbitrary $t_0>0$. It follows from the boundedness of the killing rate $\kappa$ and the parabolic Harnack inequality that there exists $c_0>0$ and $\nu\in \calP(\bar D)$ such that
\[
\Law_x(X_{t_0}\lvert \tau_{\partial}>t_0)\geq c_0\nu\quad\text{for all}\quad x\in \bar D.
\]
Thus \cite[Assumption (A1)]{Champagnat2014} is satisfied. We now turn to checking \cite[Assumption (A2)]{Champagnat2014}.

In \cite[p.6]{Schwab2005} they use the Krein-Rutman theorem to prove that there exists $\phi\in \calA\cap C^{2}(\bar D; \Rm_{>0})$ and $\lambda\in \Rm$ such that
\[
L\phi+\lambda\phi=0\quad\text{on}\quad \bar D,\quad \phi>0\quad\text{on}\quad \bar D,\quad \nabla\phi\cdot\vec{n}=0\quad\text{on}\quad \partial D.
\]

We see that $e^{\lambda t}\phi(X_t)\Ind(t<\tau_{\partial})$ is a martingale, so that
\begin{equation}\label{eq:integral of right eigenfunction against dist}
\langle P_{t}(\mu,\cdot),\phi\rangle= e^{-\lambda t}\langle\mu,\phi\rangle\quad \text{for all}\quad \mu\in \calP(\bar D).
\end{equation}
Therefore by \eqref{eq:integral of right eigenfunction against dist} we have for all $\mu\in\calP(\bar D)$ that
\begin{equation}\label{eq:lower bound on Pnu of killing in time t proof for reflected diffusion}
\Pm_{\mu}(t<\tau_{\partial})\geq \frac{\langle P_{t}(\mu,\cdot),\phi\rangle}{\sup_{x'\in \bar D}\phi(x')}= \frac{e^{-\lambda t}\langle \mu,\phi\rangle}{\sup_{x'\in \bar D}\phi(x')}\geq \frac{\inf_{x'\in \bar D}\phi(x')}{\sup_{x'\in \bar D}\phi(x')}e^{-\lambda t}.
\end{equation}
Similarly, \eqref{eq:integral of right eigenfunction against dist} gives that for all $x\in \bar D$ we have
\begin{equation}\label{eq:bound of Px killing in time t proof for reflected diffusion}
\Pm_{x}(t<\tau_{\partial})\leq \frac{\langle P_{t}(x,\cdot),\phi\rangle}{\inf_{x'\in \bar D}\phi(x')}\leq \frac{\sup_{x'\in \bar D}\phi(x')}{\inf_{x'\in \bar D}\phi(x')}e^{-\lambda t}.
\end{equation}
Therefore we have
\[
\Pm_{x}(t<\tau_{\partial})\leq \Big(\frac{\sup_{x'\in \bar D}\phi(x')}{\inf_{x'\in \bar D}\phi(x')}\Big)^2\Pm_{\mu}(t<\tau_{\partial}),\quad \text{for all}\quad t\geq 0\quad\text{and}\quad x\in \bar D.
\]

Thus we have verified \cite[Assumption (A)]{Champagnat2014}, so that \cite[Theorem 1.1]{Champagnat2014} implies the existence of a quasi-stationary distribution $\pi\in\calP(\bar D)$ satisfying \eqref{eq:exponential convergence to QSD reflected diffusion}, which must be the unique quasi-stationary distribution. Moreover the uniqueness of $\phi$ up to renormalisation, both as a non-negative right eigenfunction and eigenfunction of eigenvalue $-\lambda$, is given by  \cite[Corollary 2.4]{Champagnat2014}. Finally, the QSD $\pi$ corresponds to the left eigenmeasure of $L$ for some eigenvalue $-\lambda'<0$ by \cite[Proposition 4]{Meleard2011}. This eigenvalue must be equal to $-\lambda$, since
$-\lambda'\langle \pi,\phi\rangle=\langle \pi, L\phi\rangle=-\lambda \langle \pi,\phi\rangle$, so that we have \eqref{eq:pi left eigenmeasure}.
\end{proof}

\section{Proof of Proposition \ref{prop:convergence in time for killed Markov 1 in main theorem proof}}

We recall that $\phi$ is normalised so that $\langle \pi,\phi\rangle =1$. We take $(x^i,\eta^i)_{1\leq i\leq n}\in (\bar D\times \K)^n$ and calculate
\begin{equation}\label{eq:rewriting of dist of (X,eta) process}
\begin{split}
\Law_{\frac{1}{n}\sum_{i=1}^n\delta_{(x^i,\eta^i)}}((X_s,\eta_s)\lvert \tau_{\partial}>s)= \frac{\frac{1}{n}\sum_{i=1}^n\Pm_{x^i}(\tau_{\partial}>s)\Law_{x^i}(X_s\lvert \tau_{\partial}>s)\otimes\delta_{\eta^i}}{\frac{1}{n}\sum_{i=1}^n\Pm_{x^i}(\tau_{\partial}>s)}\\
= \frac{\frac{1}{n}\sum_{i=1}^ne^{\lambda s}\Pm_{x^i}(\tau_{\partial}>s)\Law_{x^i}(X_s\lvert \tau_{\partial}>s)\otimes\delta_{\eta^i}}{\frac{1}{n}\sum_{i=1}^ne^{\lambda s}\Pm_{x^i}(\tau_{\partial}>s)}.
\end{split}
\end{equation}
It then follows from the triangle inequality that
\[
\begin{split}
\Big\lvert\Big\lvert\frac{1}{n}\sum_{i=1}^ne^{\lambda s}\Pm_{x^i}(\tau_{\partial}>s)\Law_{x^i}(X_s\lvert \tau_{\partial}>s)\otimes\delta_{\eta^i}-\frac{1}{n}\sum_{i=1}^n\phi(x^i)\pi\otimes \delta_{\eta^i}\Big\rvert\Big\rvert_{\TV}
\\\leq 
\Big\lvert\Big\lvert\frac{1}{n}\sum_{i=1}^n[e^{\lambda s}\Pm_{x^i}(\tau_{\partial}>s)-\phi(x^i)]\Law_{x^i}(X_s\lvert \tau_{\partial}>s)\otimes\delta_{\eta^i}\Big\rvert\Big\rvert_{\TV}\\+\Big\lvert\Big\lvert\frac{1}{n}\sum_{i=1}^n\phi(x^i)[\Law_{x^i}(X_s\lvert \tau_{\partial}>s)\otimes\delta_{\eta^i}-\pi\otimes \delta_{\eta^i}]\Big\rvert\Big\rvert_{\TV}\\
\leq \frac{1}{n}\sum_{i=1}^n\lvert e^{\lambda s}\Pm_{x^i}(\tau_{\partial}>s)-\phi(x^i)\rvert+\frac{1}{n}\sum_{i=1}^n\phi(x^i)\lvert\lvert \Law_{x^i}(X_s\lvert \tau_{\partial}>s)\otimes\delta_{\eta^i}-\pi\otimes \delta_{\eta^i}\rvert\rvert_{\TV}.
\end{split}
\]
We can apply \cite[Theorem 2.1]{Champagnat2017} by Theorem \ref{theo:convergence to QSD for reflected diffusion with soft killing}. It follows from Theorem \ref{theo:convergence to QSD for reflected diffusion with soft killing} and \cite[Theorem 2.1]{Champagnat2017} that there exists $\epsilon_t\ra 0$ such that, for any $n<\infty$ and $(x_i,\eta_i)_{1\leq i\leq n}\in (D\times \K)^n$, we have that
\[
\Big\lvert\Big\lvert\frac{1}{n}\sum_{i=1}^ne^{\lambda s}\Pm_{x^i}(\tau_{\partial}>s)\Law_{x^i}(X_s\lvert \tau_{\partial}>s)\otimes\delta_{\eta^i}-\frac{1}{n}\sum_{i=1}^n\phi(x^i)\pi\otimes \delta_{\eta^i}\Big\rvert\Big\rvert_{\TV}\leq \epsilon_t\frac{1}{n}\sum_{i=1}^n \phi(x^i).
\]
We apply this to both the numerator and denominator of the right hand side of \eqref{eq:rewriting of dist of (X,eta) process} to obtain \eqref{eq:convergence in time for killed Markov 1 in intro}.
\qed

\section{Spaces of measures}\label{appendix:spaces of measures}

For a given topological space ${\bf S}$ we write $\calB({\bf S})$ for the Borel $\sigma$-algebra on ${\bf S}$, and write $\calP({\bf S})$ for the space of probability measures on $\calB({\bf S})$, equipped with the topology of weak convergence of measures. We write $\mathcal{M}({\bf S})$ for the space of all bounded Borel measurable functions on ${\bf S}$. 

\subsection{The Wasserstein metric}\label{appendix:Wasserstein metric}

For general separable metric spaces $({\bf S},d)$ we let $\Wah$ denote the Wasserstein-$1$ metric on $\mathcal{P}({\bf S})$ generated by the metric $d\wedge 1$, which metrises $\calP({\bf S})$ \cite[Theorem 6]{Gibbs2002}. We write $\mathcal{P}_{\Wah}({\bf S})$ for the metric space $(\mathcal{P}({\bf S}),\Wah)$. The following therefore follows from the Skorokhod representation theorem.
\begin{prop}\label{prop:convergence in W in prob equivalent to weakly prob}
Let $({\bf S},d)$ be a separable metric space. Let $(\mu_n)_{n\geq 1}$ be a sequence of $\calP({\bf S})$-valued random measures, and $\mu\in\calP({\bf S})$ a deterministic measure. Then the following are equivalent:
\begin{enumerate}
\item $\Wah(\mu_n,\mu)\overset{p}{\ra} 0$ as $n\ra\infty$;\label{enum:equiv prop conv in W}
\item $\mu_n(f)\overset{p}{\ra} \mu(f)$ as $n\ra\infty$ for all $f\in C_b({\bf S})$.\label{enum:equiv prop conv against test fns}
\end{enumerate}
\end{prop}

We similarly obtain the following lemma.
\begin{lem}\label{lem:tight sequences of measures converging to same limit lemma}
Let $({\bf S},d)$ be a separable metric space. Let $(\mu^{(i)}_n)_{n\geq 1}$ for $i=1,2$ be tight sequences of $\calP({\bf S})$-valued random measures, with $\mu^{(1)}_n$ and $\mu^{(2)}_n$ defined on the same probability space, for all $n$. We suppose that 
\begin{equation}\label{eq:mu1(A)ra mu2(A) in lemma}
\lvert \mu^{(1)}(A)-\mu^{(2)}_n(A)\rvert \overset{p}{\ra} 0\quad\text{as}\quad  n\ra \infty \quad\text{for all} \quad A\in \calB(\bfS).
\end{equation}
Then $\Wah(\mu^{(1)}_n,\mu^{(2)}_n)\overset{p}{\ra} 0$ as $n\ra\infty$.
\end{lem}

\begin{proof}[Proof of Lemma \ref{lem:tight sequences of measures converging to same limit lemma}]
Since $((\mu^{(1)}_n,\mu^{(2)}_n):1\leq n<\infty)$ is a tight sequence of random measures, we may consider arbitrary subequential limits to which we apply Skorokhod's representation theorem. Then on this new probability space and along the subsequential limit we have $(\mu^{(1)}_{n_k},\mu^{(2)}_{n_k})\ra (\mu^{(1)},\mu^{(2)})$ in $\calP(\bfS)\times \calP(\bfS)$ as $k\ra\infty$. We now use \eqref{eq:mu1(A)ra mu2(A) in lemma} to conclude that $\mu^{(1)}=\mu^{(2)}$ almost surely, from which we conclude Lemma \ref{lem:tight sequences of measures converging to same limit lemma}.
\end{proof}

\subsection{The Weak Atomic Metric}\label{appendix:Weak atomic metric}

Convergence in our scaling limit is given in terms of the weak atomic metric, introduced by Ethier and Kurtz in \cite{Ethier1994}. We shall define the weak atomic metric on the colour space $(\K,d)$ (which we recall is assumed to be a complete, separable metric space). We write $\calP_{\Wat}(\K)$ for $\calP(\K)$ equipped with the metric $\Wat$.

In \cite{Ethier1994}, Ethier and Kurtz defined the weak atomic metric on the space of all finite, positive, Borel measures, whereas we restrict our attention to probability measures on $\K$. We fix $\Psi(u)=(1-u)\vee 0$ and define the weak atomic metric to be
\begin{equation}
\begin{split}
\Wat(\mu,\nu):=\Wah(\mu,\nu)\\
+\sup_{0<\epsilon\leq 1}\Big\lvert \int_{\K}\int_{\K}\Psi\Big(\frac{d(x,y)}{\epsilon}\Big)\mu(dx)\mu(dy)-\int_{\K}\int_{\K}\Psi\Big(\frac{d(x,y)}{\epsilon}\Big)\nu(dx)\nu(dy)\Big\rvert.
\end{split}
\label{eq:Weak-atomic metric}
\end{equation}

In \cite{Ethier1994} they used the Levy-Prokhorov metric instead of the $\Wah$-metric, and let $\Psi$ be an arbitrary continuous, non-decreasing function such that $\Psi(0)=1$ and $\Psi(1)=0$. We make the above choices for simplicity (note that $\Wah$ is equivalent to the Levy-Prokhorov metric \cite[Theorem 2]{Gibbs2002}). Convergence in the weak atomic metric is equivalent to weak convergence of measures and convergence of the location and sizes of the atoms.
\begin{lem}[Lemma 2.5, \cite{Ethier1994}]\label{lem:Weak atomic metric = conv of atoms and weak}
Consider a sequence of probability measures $(\mu_n)_{n=1}^{\infty}$ and a probability measure $\mu$, all in $\calP(\K)$. The following are equivalent:
\begin{enumerate}
    \item $\Wat(\mu_n,\mu)\ra 0$ as $n\ra\infty$.
    \item We have both of the following:
    \begin{enumerate}
        \item $\Wah(\mu_n,\mu)\ra 0$ as $n\ra\infty$;\label{eq:Wah convergence equivalence to Wat lemma}
        \item there exists an ordering of the atoms $\{\alpha_i\delta_{x_i}\}$ of $\mu$ and the atoms $\{\alpha^n_i\delta_{x^n_i}\}$ of $\mu_n$ so that $\alpha_1\geq \alpha_2\geq \ldots$ and $\lim_{n\ra\infty}(\alpha^n_i,x^n_i)=(\alpha_i,x_i)$ for all $i$.
    \end{enumerate}
\end{enumerate}
\end{lem}
\begin{rmk}
Note that \eqref{eq:Wah convergence equivalence to Wat lemma} is equivalent to $\mu_n\ra \mu$ weakly by Proposition \ref{prop:convergence in W in prob equivalent to weakly prob}.
\end{rmk}

Thus measures are close in the weak atomic metric if and only if they are both close in the Wasserstein-$1$ metric $\Wah$ and have similar atoms. For instance $\frac{1}{2}\text{Leb}_{[0,1]}+\frac{1}{2}\delta_{\frac{1}{2}}$ is close in the weak atomic metric to $(\frac{1}{2}-\epsilon)\text{Leb}_{[0,1]}+(\frac{1}{2}+\epsilon)\delta_{\frac{1}{2}+\epsilon}$ (for small $\epsilon>0$) but not to $\frac{1}{2}\text{Leb}_{[0,1]}+(\frac{1}{4}\delta_{\frac{1}{2}-\epsilon}+\frac{1}{4}\delta_{\frac{1}{2}+\epsilon})$ nor to $\frac{1}{3}\text{Leb}_{[0,1]}+\frac{2}{3}\delta_{\frac{1}{2}}$.

We note by \cite[p.5]{Ethier1994} that $\calB(\calP(\K))=\calB(\calP_{\Wat}(\K))$, so that probability measures in $\calP(\calP_{\Wat}(\K))$ are probability measures in $\calP(\calP(\K))$ and vice-versa. It will be useful to be able to characterise tightness in both $\calP(\calP_{\Wat}(\K))$ and $\calP(D([0,T];\calP_{\Wat}(\K)))$. 

Ethier and Kurtz established in \cite[Lemma 2.9]{Ethier1994} the following tightness criterion.
\begin{lem}[Lemma 2.9, \cite{Ethier1994}]\label{lem:relatively compact family of measures in weak atomic topology}
Consider a sequence of measures $(\mu_n)_{n=1}^{\infty}$ in $\calP(\calP(\K))$. Then the following are equivalent:
\begin{enumerate}
    \item $(\mu_n)_{n=1}^{\infty}$ is tight in $\calP(\calP_{\Wat}(\K))$.
    \item $(\mu_n)_{n=1}^{\infty}$ is tight in $\calP(\calP_{\Wah}(\K))$ and we also have
    \begin{equation}\label{eq:Compact containment condition}
    \sup_n\expE\Big[\int_{\K}\int_{\K}\Psi\Big(\frac{d(x,y)}{\epsilon}\Big)\Ind(x\neq y)\mu_n(dx)\mu_n(dy)\Big]\ra 0\quad\text{as}\quad \epsilon\ra 0.
    \end{equation}
\end{enumerate}
\end{lem}

Note that the above statement is slightly different from the statement of \cite[Lemma 2.9]{Ethier1994}. It is straightforward to see that the two statements are equivalent for families of probability measures; for our purposes this lemma statement will be easier to use.

\section{The Wright-Fisher process}\label{appendix:Wright-Fisher process}

The Wright-Fisher process is defined as a solution of a martingale problem. There are various possible formulations of this martingale problem, which can be found in \cite[Section 3]{Ethier1993c}. The formulation we employ is given by \cite[(3.20) and (3.21)]{Ethier1993c}, and is defined as follows.
\begin{defin}[Wright-Fisher process]
A Wright-Fisher process on $\mathcal{P}(\K)$ of rate $\theta> 0$ with initial condition $\nu^0\in\calP(\K)$ is a continuous $\mathcal{P}(\K)$-valued process $(\nu_t)_{0\leq t<\infty}$ such that $\nu_0:=\nu^0$ and which is a solution of the following martingale problem.

We define for all $n\geq 2$ the maps
\[
\begin{split}
\Phi^{(n)}_{ij}:\calB_b(\K^n)\ra \calB_b(\K^{n-1})\\
f\mapsto (\Phi^{(n)}_{ij}f:(x_1,\ldots,x_{n-1})\mapsto f(x_1,\ldots,x_{j-1},x_i,x_j,\ldots,x_{n-1}).
\end{split}
\]
We further define for all $n\geq 1$ and $f\in \calB_b(\K^n)$ the map $\varphi_f\in \calB_b(\calP(\K))$ by
\begin{equation}\label{eq:function in domain of generator for WF process}
\varphi_f(\nu):=\nu^{\otimes n}(f).
\end{equation}
We then define the generator
\begin{equation}
    (\mathcal{L}\varphi_f)(\nu):=\theta\sum_{1\leq i<j\leq n}[\nu^{\otimes (n-1)}(\Phi^{(n)}_{ij}f)-\nu^{\otimes n}(f)],\quad \calD(\mathcal{L})=\{\varphi_f\quad\text{given by}\quad \eqref{eq:function in domain of generator for WF process}\}.
\label{eq:generator of MVWF diff}
\end{equation} 
The martingale problem defining the Wright-Fisher process of rate $\theta>0$ is then the condition that, for all $\varphi_f$ given by \eqref{eq:function in domain of generator for WF process},
\begin{equation}\label{eq:martingales in WF martingale problem}
\varphi_f(\nu_t)-\varphi_f(\nu_0)-\int_0^t(\mathcal{L}\varphi_f)(\nu_s)ds\quad\text{is a martingale.}
\end{equation}
\label{defin:MVWF}
\end{defin}

\begin{theo}[\cite{Ethier1993c}]\label{theo:Well-posedness of WF process}
We fix $\nu^0\in\calP(\K)$ and $\theta>0$. There exists a unique in law Wright-Fisher process on $\calP(\K)$ with initial condition $\nu_0=\nu^0$. Moreover the sample paths are continuous in the weak atomic metric,
\begin{equation}\label{eq:weak atomic continuous}
(\nu_t)_{0\leq t<\infty}\in C([0,\infty);\mathcal{P}_{\Wat}(\K))\subseteq C([0,\infty);\mathcal{P}(\K)),\quad\text{almost surely.}
\end{equation}
\end{theo}

Existence and uniqueness of the Wright-Fisher process in $C([0,\infty);\calP(\K))$ is given by \cite[Theorem 7.1]{Ethier1993c}. Continuity of sample paths in the weak atomic metric is given by \cite[Corollary 7.4]{Ethier1993c}.

We now provide a proof of Proposition \ref{prop:basic facts WF superprocess}.
\begin{proof}[Proof of Proposition \ref{prop:basic facts WF superprocess}]
The martingales in \eqref{eq:martingales in WF martingale problem} are continuous by \cite[Proposition 7.3]{Ethier1993c}.

We fix some choice of disjoint measurable sets $A_1,\ldots,A_n$ with $\dot{\cup}_{j=1}^nA_j=\K$. For arbitrary $1\leq i,j\leq n$ we take the test functions $\varphi_f$ given by \eqref{eq:function in domain of generator for WF process} with the choices of $f(x_1)=\Ind(x_1\in A_i)$ and $f(x_1,x_2)=\Ind(x_1\in A_i,x_2\in A_j)$. It follows that
\[
\nu_t(A_i)\quad\text{and}\quad \nu_t(A_i)\nu_t(A_j)-\theta\int_0^t[\nu_t(A_i)\Ind(i=j)-\nu_t(A_i)\nu_t(A_j)]dt
\]
are continuous martingales for all $1\leq i,j\leq n$. Proposition \ref{prop:basic facts WF superprocess} follows.
\end{proof}

\section{Brownian motion with hard killing at the boundary}\label{appendix:hard killing}

In this appendix, $(B_t)_{0\leq t<\tau_{\partial}}$ is Brownian motion in an open, connected, bounded domain $D$, killed instantaneously at the boundary. The Fleming-Viot particle system $(\vec{B}^N_t)_{t\geq 0}$ and Fleming-Viot multi-colour process $(\vec{B}^N_t,\vec{\eta}^N_t)_{t\geq 0}$ are driven by this Brownian motion with hard killing (as in Definition \ref{defin:hard killing Multi-Colour Process}). 

We have the following analogue of Theorem \ref{theo:convergence to QSD for reflected diffusion with soft killing}.
\begin{theo}\label{theo:hard killing convergence to QSD for reflected diffusion with soft killing}
There exists a unique quasi-stationary distribution (QSD) $\pi\in \mathcal{P}( D)$ for $(B_t)_{0\leq t<\tau_{\partial}}$. Moreover there exist constants $C<\infty$ and $k>0$ such that
\begin{equation}\label{eq:hard killing exponential convergence to QSD reflected diffusion}
\lvert\lvert \mathcal{L}_{\mu}(B_t\lvert \tau_{\partial}>t)-\pi\rvert\rvert_{\TV}\leq Ce^{-kt}\quad\text{for all}\quad\mu\in\mathcal{P}( D)\quad\text{and}\quad t\geq 0.
\end{equation}
Furthermore $\pi$ is a left eigenmeasure of $L$ with eigenvalue $-\lambda<0$,
\begin{equation}\label{eq:hard killing pi left eigenmeasure}
\langle \pi, Lf\rangle=-\lambda \langle \pi,f\rangle,\quad f\in\calD(L).
\end{equation}
Moreover $L$ has a positive right eigenfunction belonging to the domain of the Carre du Champs operator described in \eqref{eq:hard killing Carre d champs}, $\phi\in C_0(D;\Rm_{>0})\cap C^{\infty}(D)\cap \calD(\Gamma)$. This right eigenfunction is both the unique non-negative right eigenfunction and the unique right eigenfunction of eigenvalue $-\lambda$, up to rescaling. 
\end{theo}

We have the following analogue of Theorem \ref{theo:hydrodynamic limit for multicolour process}.

\begin{defin}\label{defin:hard killing limit for multi colour over fixed times}
We define a $ D\times\K$-valued killed Markov process, denoted by $((B_t,\eta_t))_{0\leq t<\tau_{\partial}}$, as follows. The process evolves in the first variable as a Brownian motion $B_t$ killed instantaneously upon contact with the boundary. The killing time $\tau_{\partial}$ is then given by $\tau_{\partial}:=\inf\{t>0:B_{t-}\in \partial D\}$. In the second variable $\eta_t$ is a constant element of $\K$ up to the killing time $\tau_{\partial}$, so that $\eta_t=\eta_0$ for all $0\leq t<\tau_{\partial}$. After the killing time the process is sent to a fixed cemetery state.
\end{defin}
We recall that the stopping time $\tau^N_{\epsilon}$ for $\epsilon>0$ was defined in \eqref{eq:stopping time tauNepsilon} by
\[
\{t>0:m^N_t(B(\partial D,\delta(\epsilon)))> \epsilon\},
\]
whereby $\delta=\delta(\epsilon)>0$ is given by Lemma \ref{lem:mass near boundary controls}.
\begin{theo}\label{theo:hard killing hydrodynamic limit for multicolour process}
We consider the Fleming-Viot multi-colour process $(\vec{B}^N_t,\vec{\eta}^N_t)_{t\geq 0}$ for $N\geq 2$. Then there exists constants $C_{\epsilon,T,N}$ for $\epsilon>0$, $0\leq T<\infty$ and $N\geq 2$ such that $C_{\epsilon,T,N}\ra 0$ as $N\ra \infty$, and such that for any initial condition $(\vec{B}^N_0,\vec{\eta}^N_0)$ for which $\tau^N_{\epsilon}>0$, and any $f\in \calB_b(\bar D\times \K;\Rm)$, we have that
\begin{align}\label{eq:hard killing hydrodnamic limit of FVMC in appendix}
\expE_{(\vec{B}^N_0,\vec{\eta}^N_0)}\Big[\sup_{t\leq T}\Big\lvert \Big(\frac{1}{N}\sum_{i=1}^N\delta_{(B^{N,i}_t,\eta^{N,i}_t)}-\Law_{\frac{1}{N}\sum_{i=1}^N\delta_{(B^{N,i}_0,\eta^{N,i}_0)}}((B_t,\eta_t))\Big)(f)\Big\rvert\Big]\leq C_{\epsilon,T,N}\lvert\lvert f\rvert\rvert_{\infty},\\
\expE_{(\vec{B}^N_0,\vec{\eta}^N_0)}\Big[\sup_{t\leq T}\Big\lvert J^N_t-\ln\Pm_{\frac{1}{N}\sum_{i=1}^N\delta_{(B^{N,i}_0,\eta^{N,i}_0)}}(\tau_{\partial}>t)\Big\rvert\wedge 1 \Big]\leq C_{\epsilon,T,N}.\label{eq:hard killing hydrodnamic limit of FVMC in appendix number of jumps}
\end{align}

\end{theo}

Finally, we have the following large-time limit for $((B_t,\eta_t))_{0\leq t<\tau_{\partial}}$, by the same proof as the proof of Proposition \ref{prop:convergence in time for killed Markov 1 in main theorem proof}.
\begin{prop}\label{prop:hard killing convergence in time for killed Markov 1 in main theorem proof}
For arbitrary sequences $(x^i,\eta^i)_{1\leq i\leq n}$ in $ D\times \K$ we consider the process \linebreak$(B_t,\eta_t)_{0\leq t<\tau_{\partial}}$ with initial distribution given by the empirical measure $\frac{1}{n}\sum_{i=1}^n\delta_{(x^i,\eta^i)}$. Then there exists $c_t\ra 0$ as $t\ra \infty$ such that, for all sequences $(x^i,\eta^i)_{1\leq i\leq n}$ in $\bar D\times \K$ and all $n\in \Nm$, we have
\begin{equation}\label{eq:hard killing convergence in time for killed Markov 1 in intro}
\begin{split}
\Big\lvert\Big\lvert\Law_{\frac{1}{n}\sum_{i=1}^n\delta_{(x^i,\eta^i)}}((B_t,\eta_t)\lvert \tau_{\partial}>t)-\frac{\sum_{i=1}^n\phi(x^i)\pi\otimes\delta_{\eta^i}}{\sum_{i=1}^n\phi(x^i)}\Big\rvert\Big\rvert_{\TV}\leq c_t,\quad 0\leq t<\infty.
\end{split}
\end{equation}
\end{prop}
\begin{proof}[Proof of Theorem \ref{theo:hard killing convergence to QSD for reflected diffusion with soft killing}, Theorem \ref{theo:hard killing hydrodynamic limit for multicolour process} and Proposition \ref{prop:hard killing convergence in time for killed Markov 1 in main theorem proof}]
It is easy to check that $(B_t)_{0\leq t<\tau_{\partial}}$ satisfies \cite[Assumption (A)]{Champagnat2014} - a proof in the H\"ormander setting is given by the present author in \cite[Theorem 7.4]{Tough2022a} (see the proof of \cite[Proposition 7.12]{Tough2022a}). The fact that $\phi$ belongs to the domain of $\Gamma$ follows from \cite[Theorem 1.1]{Xu2009}. Otherwise, the proofs of Theorem \ref{theo:hard killing convergence to QSD for reflected diffusion with soft killing}, Theorem \ref{theo:hard killing hydrodynamic limit for multicolour process} and Proposition \ref{prop:hard killing convergence in time for killed Markov 1 in main theorem proof} are identical to those of Theorem \ref{theo:convergence to QSD for reflected diffusion with soft killing}, Theorem \ref{theo:hydrodynamic limit for multicolour process} and Proposition \ref{prop:convergence in time for killed Markov 1 in main theorem proof} respectively.
\end{proof}

\end{appendix}

{\textbf{Acknowledgement:}}  
This work was partially funded by grant 200020 196999 from the Swiss National Foundation and by the EPSRC MathRad programme
grant EP/W026899/.

\bibliography{ScalingLimitFlemingViot}

\begin{thebibliography}{10}

\bibitem{Aldous1978}
David Aldous.
\newblock Stopping times and tightness.
\newblock {\em Annals of Probability}, 6(2):335--340, 1978.

\bibitem{Asselah2011}
Amine Asselah, Pablo~A Ferrari and Pablo Groisman.
\newblock Quasistationary distributions and Fleming-Viot processes in finite
  spaces.
\newblock {\em J. Appl. Prob}, 48:322--332, 2011.

\bibitem{Asselah2016}
Amine Asselah, Pablo~A Ferrari, Pablo Groisman and Matthieu Jonckheere.
\newblock {{F}leming-{V}iot selects the minimal quasi-stationary distribution:
  the {G}alton-{W}atson case}.
\newblock {\em Annales de l'Institut Henri Poincar{\'{e}}}, 52(2):647--668,
  2016.

\bibitem{Ball2006}
Karen Ball, Thomas~G Kurtz, Lea Popovic and Greg Rempala.
\newblock Asymptotic analysis of multiscale approximations to reaction
  networks.
\newblock {\em The Annals of Applied Probability}, 16:1925--1961, 2006.

\bibitem{Berestycki2013}
Julien Berestycki, Nathanaël Berestycki and Jason Schweinsberg.
\newblock The genealogy of branching brownian motion with absorption.
\newblock {\em Annals of Probability}, 41:527--618, 2013.

\bibitem{Berestycki2022}
Julien Berestycki, Éric Brunet, James Nolen and Sarah Penington.
\newblock Brownian bees in the infinite swarm limit.
\newblock {\em The Annals of Probability}, 50:2133--2177, 2022.

\bibitem{Bieniek2018}
Mariusz Bieniek and Krzysztof Burdzy.
\newblock The distribution of the spine of a {F}leming-{V}iot type process.
\newblock {\em Stochastic Processes and their Applications},
  128(11):3751--3777, 2018.

\bibitem{Bieniek2009}
Mariusz Bieniek, Krzysztof Burdzy and Sam Finch.
\newblock Non-extinction of a {F}leming-{V}iot particle model.
\newblock {\em Probability Theory and Related Fields}, 153(1-2):293--332, 2012.

\bibitem{Brown2021}
Suzie Brown, Paul~A. Jenkins, Adam~M. Johansen and Jere Koskela.
\newblock {Simple conditions for convergence of sequential Monte Carlo
  genealogies with applications}.
\newblock {\em Electronic Journal of Probability}, 26:1--22, 2021.

\bibitem{Brunet1997}
Éric Brunet and Bernard Derrida.
\newblock Shift in the velocity of a front due to a cutoff.
\newblock {\em Physical Review E}, 56:2597--2604, 1997.

\bibitem{Brunet2001}
Éric Brunet and Bernard Derrida.
\newblock Effect of microscopic noise on front propagation.
\newblock {\em Journal of Statistical Physics}, 103:269--282, 2001.

\bibitem{Brunet2006}
Éric Brunet, B~Derrida, A.~H Mueller and S~Munier.
\newblock Noisy traveling waves: Effect of selection on genealogies.
\newblock {\em Europhysics Letters (EPL)}, 76:1--7, 2006.

\bibitem{Brunet2007}
É. Brunet, B.~Derrida, A.~H. Mueller and S.~Munier.
\newblock Effect of selection on ancestry: an exactly soluble case and its
  phenomenological generalization.
\newblock {\em Physical Review E}, 76:041104, 2007.

\bibitem{Burdzya}
Krzysztof Burdzy.
\newblock {List of open problems}.

\bibitem{Burdzy2021a}
Krzysztof Burdzy and J\'anos Engl\"ander.
\newblock The spine of the Fleming-Viot process driven by Brownian motion.
\newblock {\em Preprint, arXiv:2112.01720}, 2021.

\bibitem{Burdzy2000}
Krzysztof Burdzy, Robert Hołyst and Peter March.
\newblock A {F}leming-{V}iot particle representation of the {D}irichlet
  {L}aplacian.
\newblock {\em Communications in Mathematical Physics}, 214(3):679--703, 2000.

\bibitem{Burdzy2019}
Krzysztof Burdzy, Bartosz Ko{\l}odziejek and Tvrtko Tadi{\'{c}}.
\newblock {Inverse exponential decay: stochastic fixed point equation and
  {ARMA} models}.
\newblock {\em Bernoulli}, 25(4B):3939--3977, 2019.

\bibitem{Burdzy2022}
Krzysztof Burdzy, Bartosz Kołodziejek and Tvrtko Tadić.
\newblock Stochastic fixed-point equation and local dependence measure.
\newblock {\em The Annals of Applied Probability}, 32:2811--2840, 2022.

\bibitem{Champagnat2014}
Nicolas Champagnat and Denis Villemonais.
\newblock {Exponential convergence to quasi-stationary distribution and
  {Q}-process}.
\newblock {\em Probability Theory and Related Fields}, 164:243--283, 2015.

\bibitem{Champagnat2017}
Nicolas Champagnat and Denis Villemonais.
\newblock Uniform convergence to the $Q$-process.
\newblock {\em Electronic Communications in Probability}, 22:1--7, 1 2017.

\bibitem{Durrett2011}
Rick Durrett and Daniel Remenik.
\newblock Brunet-Derrida particle systems, free boundary problems and
  Wiener-Hopf equations.
\newblock {\em Annals of Probability}, 39:2043--2078, 2011.

\bibitem{Ethier1993c}
S.~N. Ethier and Thomas~G. Kurtz.
\newblock {Fleming-Viot processes in population genetics}.
\newblock {\em SIAM Journal on Control and Optimization}, 31(2):345--386, 1993.

\bibitem{Ethier1994}
S.~N. Ethier and Thomas~G. Kurtz.
\newblock {Convergence to {F}leming-{V}iot processes in the weak atomic
  topology}.
\newblock {\em Stochastic Processes and their Applications}, 54(1):1--27, 1994.

\bibitem{Fleming1979}
Wendell Fleming and Michel Viot.
\newblock {Some measure-valued {M}arkov processes in population genetics
  theory}.
\newblock {\em Indiana University Mathematics Journal}, 28(5):817--843, 1979.

\bibitem{Frankham1995}
R.~Frankham.
\newblock {Effective population size/adult population size ratios in wildlife:
  a review}.
\newblock {\em Genetical Research}, 66(2):95--107, 1995.

\bibitem{Gibbs2002}
Alison~L. Gibbs and Francis~Edward Su.
\newblock {On choosing and bounding probability metrics}.
\newblock {\em International Statistical Review}, 70(3):419--435, 2002.

\bibitem{Grigorescu2007}
Ilie Grigorescu.
\newblock Large deviations for a catalytic {F}leming-{V}iot branching system.
\newblock {\em Communications on Pure and Applied Mathematics},
  60(7):1056--1080, 2007.

\bibitem{Grigorescu2012}
Ilie Grigorescu and Min Kang.
\newblock Immortal particle for a catalytic branching process.
\newblock {\em Probability Theory and Related Fields}, 153(1-2):333--361, 2012.

\bibitem{Katzenberger1991}
G.~S. Katzenberger.
\newblock {Solutions of a stochastic differential equation forced onto a
  manifold by a large drift}.
\newblock {\em Annals of Probability}, 19(4):1587--1628, 1991.

\bibitem{Kurtz1992}
Thomas~G Kurtz.
\newblock Averaging for martingale problems and stochastic approximation.
\newblock Pages 186--209. Springer Berlin Heidelberg, 1992.

\bibitem{Labbe2013}
Cyril Labb{\'{e}}.
\newblock {\em {Flots stochastiques et repr{\'{e}}sentation lookdown}}.
\newblock PhD thesis, Universit{\'{e}} Pierre-et-Marie-Curie, 2013.

\bibitem{Lions1984a}
P.~L. Lions and A.~S. Sznitman.
\newblock {Stochastic differential equations with reflecting boundary
  conditions}.
\newblock {\em Communications on Pure and Applied Mathematics}, 37(4):511--537, 1984.

\bibitem{Lobus2009}
Jörg-Uwe Löbus.
\newblock A stationary Fleming-Viot type Brownian particle system.
\newblock {\em Mathematische Zeitschrift}, 263:541--581, 2009.

\bibitem{Maillard2016}
Pascal Maillard.
\newblock Speed and fluctuations of $N$-particle branching Brownian motion with
  spatial selection.
\newblock {\em Probability Theory and Related Fields}, 166:1061--1173, 12 2016.

\bibitem{Mallein2017}
Bastien Mallein.
\newblock Branching random walk with selection at critical rate.
\newblock {\em Bernoulli}, 23:1784--1821, 2017.

\bibitem{Meleard2012}
Sylvie M{\'{e}}l{\'{e}}ard and Viet~Chi Tran.
\newblock {Slow and fast scales for superprocess limits of age-structured
  populations}.
\newblock {\em Stochastic Processes and their Applications}, 122(1):250--276, 2012.

\bibitem{Meleard2011}
Sylvie M\'{e}l\'{e}ard and Denis Villemonais.
\newblock Quasi-stationary distributions and population processes.
\newblock {\em Probability Surveys}, 9:340--410, 2012.

\bibitem{DeMulatier2015}
Clélia~De Mulatier, Eric Dumonteil, Alberto Rosso and Andrea Zoia.
\newblock The critical catastrophe revisited.
\newblock {\em Journal of Statistical Mechanics: Theory and Experiment}, 2015.

\bibitem{Penington2022}
Sarah Penington, Matthew~I. Roberts and Zsófia Talyigás.
\newblock Genealogy and spatial distribution of the $N$-particle branching random
  walk with polynomial tails.
\newblock {\em Electronic Journal of Probability}, 27, 2022.

\bibitem{Schwab2005}
Christoph Schwab.
\newblock {{K}rein-{R}utman theorem and the principal eigenvalue}.
\newblock In {\em Numerical Methods for Elliptic and Parabolic PDEs (Lecture
  Notes)}, 2005.

\bibitem{Stroock1971}
Daniel~W. Stroock and S.~R.~S. Varadhan.
\newblock {Diffusion processes with boundary conditions}.
\newblock {\em Communications on Pure and Applied Mathematics}, 24(2):147--225, 1971.

\bibitem{Tough2021b}
Oliver Tough.
\newblock Asymptotic behaviour of the Fleming-Viot process, PhD Thesis, Duke University, 2021.

\bibitem{Tough2021}
Oliver Tough.
\newblock {Scaling limit of the Fleming-Viot multi-colour process}.
\newblock {\em Preprint, arXiv:2110.05049, version 1}, 2021.

\bibitem{Tough2022a}
Oliver Tough.
\newblock $L^{\infty}$-convergence to a quasi-stationary distribution.
\newblock {\em Preprint, arXiv:2210.13581}, 2022.

\bibitem{Tough2023}
Oliver Tough.
\newblock Selection principle for the Fleming-Viot process with drift $-1$.
\newblock {\em Preprint, arXiv:2306.03585},  2023.

\bibitem{Tough2022}
Oliver Tough and James Nolen.
\newblock The Fleming-Viot process with Mckean-Vlasov dynamics.
\newblock {\em Electronic Journal of Probability}, 27:1 – 72, 2022.

\bibitem{Villemonais2011}
Denis Villemonais.
\newblock {General approximation method for the distribution of {M}arkov
  processes conditioned not to be killed}.
\newblock {\em ESAIM: Probability and Statistics}, 18:441--467, 2014.

\bibitem{Xu2009}
Xiangjin Xu.
\newblock Gradient estimates for the eigenfunctions on compact manifolds with
  boundary and Hörmander multiplier theorem.
\newblock {\em Forum Mathematicum}, 21:455--476, 2009.


\end{thebibliography}
\bibliographystyle{plain}
\end{document}